\documentclass{article}
\usepackage[portrait,margin=1in]{geometry} 

\usepackage[utf8]{inputenc}
\usepackage{amsthm,amssymb,amsmath,bbm}
\usepackage{tikz-cd}
\usepackage{tikz}
\usetikzlibrary{positioning}
\usetikzlibrary{calc, matrix, patterns}
\usepackage{pgfplots}
\pgfplotsset{compat=1.10}
\usepgfplotslibrary{fillbetween}
\tikzset{fill between/on layer=main}
\usetikzlibrary{mindmap}

\DeclareFontFamily{U}{mathx}{}
\DeclareFontShape{U}{mathx}{m}{n}{<-> mathx10}{}
\DeclareSymbolFont{mathx}{U}{mathx}{m}{n}
\DeclareMathAccent{\widehat}{0}{mathx}{"70}
\DeclareMathAccent{\widecheck}{0}{mathx}{"71}

\usepackage{caption,subcaption}
\usepackage{epsfig,graphicx,graphics,color,tikz,tikz-cd,tcolorbox}
\usepackage{enumerate,enumitem}
\usepackage[sort,noadjust]{cite}
\usepackage{hyperref}
\usepackage{todonotes}
\usepackage{newtxmath}
\usepackage[normalem]{ulem}
\hypersetup{
    colorlinks=true,       
    linkcolor=blue,        
    citecolor=magenta,         
    filecolor=magenta,     
    urlcolor=cyan,         
    linktoc=all
}

\def\final{0}  
\def\iflong{\iffalse}
\ifnum\final=0  
\newcommand{\kristof}[1]{{\color{red}[{\textbf{Kristóf:} \bf #1}]\marginpar{\color{red}*}}}
\newcommand{\marci}[1]{{\color{blue}[{\textbf{Marci:} \bf #1}]\marginpar{\color{blue}*}}}
\newcommand{\llaci}[1]{{\color{orange}[{\textbf{LLaci:} \bf #1}]\marginpar{\color{orange}*}}}
\newcommand{\tlaci}[1]{{\color{purple}[{\textbf{TLaci:} \bf #1}]\marginpar{\color{purple}*}}}
\newcommand{\greg}[1]{{\color{violet}[{\textbf{Greg:} \bf #1}]\marginpar{\color{violet}*}}}

\else 
\newcommand{\kristof}[1]{}
\newcommand{\marci}[1]{}
\newcommand{\llaci}[1]{}
\newcommand{\tlaci}[1]{}
\newcommand{\greg}[1]{}
\fi

\theoremstyle{plain}
\newtheorem{thm}{Theorem}[section]
\newtheorem{lem}[thm]{Lemma}
\newtheorem{cor}[thm]{Corollary}
\newtheorem{cl}[thm]{Claim}
\newtheorem{obs}[thm]{Observation}
\newtheorem{prop}[thm]{Proposition}

\theoremstyle{definition}
\newtheorem{defn}[thm]{Definition}

\newtheorem{ex}[thm]{Example}

\newtheorem{rem}[thm]{Remark}

\newcommand{\R}{\mathbb{R}}
\newcommand{\Z}{\mathbb{Z}}
\newcommand{\N}{\mathbb{N}}
\newcommand{\Cay}{\mathrm{Cay}}

\newcommand*\diff{\mathop{}\!\mathrm{d}}

\newcommand{\cA}{\mathcal{A}}
\newcommand{\cB}{\mathcal{B}}
\newcommand{\cG}{\mathcal{G}}
\newcommand{\cS}{\mathcal{S}}

\newcommand{\cV}{\mathcal{V}}
\newcommand{\cQ}{\mathcal{Q}}
\newcommand{\cE}{\mathcal{E}}

\newcommand{\cU}{\mathcal{U}}
\newcommand{\cF}{\mathcal{F}}
\newcommand{\cH}{\mathcal{H}}
\newcommand{\cK}{\mathcal{K}}
\newcommand{\mP}{\mathscr P}
\newcommand{\mC}{\mathscr C}

\newcommand{\cP}{\mathcal{P}}
\newcommand{\cT}{\mathcal{T}}

\newcommand{\cR}{\mathcal{R}}

\newcommand{\fg}{\varphi}
\newcommand{\eps}{\varepsilon}

\tikzset{
		big dot/.style={
			circle, inner sep=0pt, 
			minimum size=1mm, fill=black
		}
	}

\title{Whitney's 2-isomorphism theorem for graphings}

\author{
Márton Borbényi \qquad
Grigory Terlov \qquad
László Márton Tóth
}

\date{}

\begin{document}

\maketitle

\begin{abstract}
We prove measurable analogues of Whitney's classical theorems on weak isomorphisms of finite graphs. For locally finite graphings, we define a weak isomorphism to be an edge-measure-preserving Borel bijection that preserves cycles and hyperfinite subgraphs, modulo null sets. Our rigidity theorem shows that every weak isomorphism of a weakly $3$-connected, infinitely-ended graphing is induced by an isomorphism of graphings. To our knowledge, this is the first general sufficient condition in measurable combinatorics for two given graphings to be isomorphic. We then prove a full measurable Whitney theorem, namely that every weak isomorphism between graphings is implemented by a finite composition of locally finite sequences of six measurable Whitney operations. 

The proofs combine measurable-combinatorial techniques, applications of spanning forests from percolation theory, and graph decomposition arguments. A central ingredient is the construction and analysis of infinitely-ended leafless subforests. 
With these tools we also show that, between weakly $2$-connected, infinitely-ended graphings, preservation of hyperfiniteness alone forces preservation of cycles, which is of independent interest.

This work further develops the limit theory of matroids recently initiated by Lovász and provides its first application in finite graph theory. We show that Whitney's rigidity theorem is robust under sublinear errors for sequences of finite graphs with uniformly bounded degrees and suitable connectivity. 
\end{abstract}

\vfill

\noindent
{\sc Márton Borbényi}\\
Department of Computer Science, Eötvös Loránd University, Budapest, Hungary,\\ and HUN-REN Alfréd Rényi Institute of Mathematics, Budapest, Hungary.\\ Email: \texttt{borbenyi.marton@renyi.hu}
\medskip
\ \\
{\sc Grigory Terlov}\\
Department of Statistics and Operations Research, University of North Carolina, Chapel Hill, NC, USA.\\ 
Email: \texttt{gterlov@unc.edu}
\medskip
\ \\
{\sc László Márton Tóth}\\
HUN-REN Alfréd Rényi Institute of Mathematics, Budapest, Hungary,\\
and Department of Algebra and Number Theory, Eötvös Loránd University, Budapest, Hungary.\\
Email: \texttt{toth.laszlo.marton@renyi.hu}

\clearpage
\setcounter{tocdepth}{2}
\tableofcontents

\clearpage
\section{Introduction and main results}
\label{sec:uniqueness}

The main focus of measurable combinatorics is to study graph-theoretic properties of infinite, measurable graphs on standard Borel probability spaces. That is, the vertex set of the graph is the space $(X,\mu)$ itself, and the edge set is a symmetric Borel subset of $X^2$. Most commonly, the theory studies a particular subclass of such graphs called graphings, which are locally finite, measure-preserving measurable graphs; see Subsection~\ref{subsec:graphings} for the precise definitions.
Graphings are connected to finite combinatorics through the limit theory of sparse graphs, while measurable combinatorics in general is also closely related to other fields, including measurable dynamics, descriptive set theory, and distributed algorithms, see e.g.\ \cite{lovasz2012large,kechris2016descriptive,BernshteynAlgo}. Some results from finite combinatorics pass to the measurable setting with relatively little change. Many others require additional assumptions or the introduction of new notions for their formulation, and reveal interesting new phenomena. Whitney's 2-isomorphism theorems belong to the latter class.

\subsection{Whitney's theorems and weak isomorphisms}

Whitney's two classical theorems describe exactly how much of a finite graph is remembered by its cycle structure. A \textbf{weak isomorphism}, or \textbf{$2$-isomorphism}, between finite graphs is a bijection of their edge sets that preserves cycles as edge sets; the order of edges in the cycles need not be preserved. Equivalently, it is an isomorphism between the cycle matroids, which are the matroids on the edge sets whose independent sets are the forests. In~\cite{whitney19332}, Whitney proved that every weak isomorphism of a $3$-connected finite graph is induced by a graph isomorphism and, more generally, that every weak isomorphism between finite graphs is implemented by splits and joins at cut vertices and by Whitney twists along vertex cuts of size two. 

In this paper, we prove measurable analogues of both statements and then use them to derive an asymptotic rigidity theorem for finite graph sequences. The measurable setting reveals new structural objects that need to be managed: the role of cycles in Whitney's finite rigidity argument is replaced by infinitely-ended leafless forests. At the same time, the cycle matroid of a graphing itself also carries information with no finite counterpart, as preservation of finite cycles is no longer sufficient. One must additionally keep track of hyperfiniteness. Our work was mainly motivated by the recent emergence of a limit theory for matroids, initiated by Lovász \cite{lovasz2023submodular, lovasz2024matroid} and further developed in \cite{berczi2026cycle,berczi2024quotient,berczi2025convergent}.
We also highlight a recent work~\cite{Measurable_Matroids} that develops a general theory of measurable matroids and derives analogues of matroid intersection and union theorems, in our opinion, significantly propelling the subject forward. Although we can state our results without reference to matroids, we still see our work as one that further develops this theory. In fact, the following definition of weak isomorphism for graphings is one of the main insights that the matroid perspective provided; it is an equivalent formulation of an isomorphism of the cycle matroids associated to the graphings, see Lemma~\ref{lem:eq_of_rankpres}. 

Recall that graphings carry measures on both their vertex and edge sets, which we assume to be finite throughout the paper. A graphing $\cG$ is \textbf{hyperfinite} if, for every $\eps>0$, there is a Borel set of edges $E_0\subseteq E(\cG)$ of measure at most $\eps$ such that removing $E_0$ leaves only finite components. We call a set of edges hyperfinite if the graph it spans is hyperfinite.

\textit{Throughout the paper, assertions concerning graphings are understood modulo null sets. Since our setting focuses on edges, we assume in all our theorems that the graphs in question have no isolated vertices.}

\begin{defn}[Weak isomorphism of graphings]\label{defn:weak_iso}
    A \textbf{weak isomorphism} of graphings $\cG_1$ and $\cG_2$ is an edge-measure-preserving Borel bijection $\varphi:E(\cG_1)\to E(\cG_2)$ that preserves cycles and hyperfiniteness of edge subsets in both directions.
\end{defn}

The two conditions in this definition have very different emphasis depending on the number of ends. An edge-measure-preserving bijection between $2$-ended graphings automatically preserves hyperfiniteness, so weak isomorphism there amounts to preserving cycles. Surprisingly, the situation is reversed for weakly $2$-connected, infinitely-ended graphings: preserving hyperfiniteness alone implies that cycles are preserved; see Theorem~\ref{thm:hyperfinite_sees_cycle}. 
We say that a countably infinite, locally finite graph is \textbf{weakly $n$-connected} if deleting any $n-1$ vertices creates no finite component, and that a graphing is weakly $n$-connected if almost every component is. Thus, somewhat counter-intuitively, an acyclic graphing in which every vertex has degree at least $3$ is weakly $3$-connected. Finally, an \textbf{isomorphism of graphings} is a measure-preserving Borel bijection of the vertex sets that is a graph isomorphism modulo null sets.

\subsection{Summary of main contributions}
The paper establishes the measurable counterparts of both parts of Whitney's theory and derives a finite asymptotic consequence.
\begin{enumerate}
    \item \emph{Measurable rigidity.} Every weak isomorphism of a weakly $3$-connected, infinitely-ended graphing is induced by an isomorphism of graphings (Theorem~\ref{thm:graph_isom}). A mixed-end version gives a general rigidity criterion for graphings with infinite components (Theorem~\ref{thm:graph_isom_general}).

    \item \emph{Classification by measurable Whitney operations.} We introduce six measurable analogues of Whitney's operations and prove that every weak isomorphism between arbitrary graphings is implemented, modulo null sets, by a finite composition of locally finite sequences of these operations (Theorem~\ref{thm:Whitney_2-isom_intro}). Here local finiteness means that almost every vertex is affected only finitely many times.

    \item \emph{Hyperfiniteness can detect cycles.} Between weakly \(2\)-connected, infinitely-ended graphings, preservation of hyperfiniteness alone already forces preservation of finite cycles (Theorem~\ref{thm:hyperfinite_sees_cycle}). Thus, in the setting responsible for the main new structural phenomena, the apparently global property of hyperfiniteness already determines the local cycle structure.

    \item \emph{Asymptotic finite rigidity.} For bounded-degree graph sequences satisfying asymptotic connectivity and end assumptions, every almost-bijective edge map that preserves cycle-matroid rank up to sublinear error is induced, after deleting negligible sets, by an almost isomorphism (Theorem~\ref{thm:finite-whitney-stability}). In particular, this applies to $3$-connected sequences of essentially large girth and to $5$-connected sequences (Corollaries~\ref{cor:min_3_deg_large_girth} and~\ref{cor:cut_quadruples}).
\end{enumerate}

\subsection{Measurable Whitney rigidity}

Our first main theorem extends Whitney's $3$-connected rigidity theorem to the genuinely new infinitely-ended setting, that is, to graphings whose almost every component has infinitely many ends.

\begin{thm}[Measurable Whitney rigidity]\label{thm:graph_isom}
    Let $\cG_1$ be a weakly $3$-connected and infinitely-ended graphing. Any weak isomorphism $\fg:E(\cG_1)\to E(\cG_2)$ is induced, modulo null sets, by an isomorphism of graphings.
\end{thm}

The infinitely-ended assumption in Theorem~\ref{thm:graph_isom} is necessary; we give a $2$-ended counterexample in Subsection~\ref{subsec:examples}. If, on the other hand, one assumes $3$-connectivity in the traditional sense---almost every component remains connected after deleting any two vertices---then the conclusion follows from the finite proof, sketched in Subsection~\ref{subsec:ideas_and_overview}. To emphasize the distinction, we call this traditional notion \textbf{strong $3$-connectivity}. A $1$-ended, weakly $3$-connected graph is automatically strongly $3$-connected, so the same conclusion holds in the $1$-ended case. We state the infinitely-ended case separately because this is where our new tools are required. In particular, infinitely-ended leafless subforests play the role that cycles play in the finite proof.

Combining this new case with the classical arguments available for components with at most two ends gives the following general rigidity theorem.

\begin{thm}[General measurable rigidity]\label{thm:graph_isom_general}
    Let $\cG_1$ be a graphing such that almost every component is infinite. Assume that, for almost every component $C$, the graph $C$ is weakly $3$-connected when it is $1$-ended or infinitely-ended, and strongly $3$-connected when it is $2$-ended. Any weak isomorphism $\fg:E(\cG_1)\to E(\cG_2)$ is induced, modulo null sets, by an isomorphism of graphings.
\end{thm}

To our knowledge, this is the first general sufficient condition in measurable combinatorics for two given graphings to be isomorphic. The scarcity of such results may seem surprising from a combinatorial point of view, since graph isomorphism is such a natural notion to study. We believe there are two main reasons why it has played a less prominent role in the measurable setting. First, constructing an isomorphism of graphings is a genuinely difficult measurable problem: one has to reconstruct the vertex incidence relation by a measure-preserving Borel map, and a componentwise graph isomorphism need not assemble measurably. Second, the neighboring theories have often been interested in coarser notions of equivalence. In measured group theory, the graphing commonly serves as a representative of an orbit equivalence relation, while in sparse graph limit theory a limiting graphing is not determined up to isomorphism. Thus changing the graphing while retaining the underlying measurable or local structure is often harmless in those settings \cite{kechris2004topics, kechris2024theory, lovasz2012large}. We discuss the relevant literature more precisely in Subsection~\ref{subsec:related_work}.

\subsection{Classification by measurable Whitney operations}

Whitney's finite classification is constructive. Besides graph isomorphisms, it uses splitting and joining at cut vertices and \textbf{Whitney twists}. In a Whitney twist, the graph is split along a vertex cut of size two, keeping both cut vertices in both pieces, and then glued back with the pairing of the two vertices switched. These operations preserve cycles and naturally induce an edge bijection between the original and resulting graphs; we say that the bijection is \textbf{implemented} by the operations. Figure~\ref{fig:weakly} illustrates the three finite operations. Since a $3$-connected graph admits no vertex cut of size at most two, Whitney's finite rigidity theorem follows as a special case of the classification.

For graphings, the operations must be defined measurably and may need to be performed at infinitely many places. In Subsection~\ref{subsec:whitney_operations} we introduce six operations: measurable versions of splits, joins, and twists in which the separated or twisted part is finite, and corresponding operations performed simultaneously along $2$-ended components. Infinite iteration is unavoidable, but it can be controlled. A sequence of operations is \textbf{locally finite} if almost every vertex is affected only finitely many times. Building on the rigidity theorem, we obtain the full measurable classification.

\begin{thm}[Measurable Whitney $2$-isomorphism theorem]\label{thm:Whitney_2-isom_intro}
   Let $\cG_1$ and $\cG_2$ be graphings and let $\varphi:E(\cG_1)\to E(\cG_2)$ be a weak isomorphism. Then $\varphi$ can be implemented, modulo null sets, by a finite composition of locally finite sequences of Whitney operations.
\end{thm}

It is not clear whether the finite composition in this statement can always be replaced by a single locally finite sequence. The theorem goes beyond the corresponding results for countable, locally finite graphs because both the operations and the resulting vertex identifications must be measurable, and because the infinitely-ended case depends essentially on preservation of hyperfiniteness rather than only on preservation of finite cycles. The $1$- and $2$-ended cases can use earlier infinite-graph versions of Whitney and Tutte theory, whereas the infinitely-ended case requires completely different, novel tools developed here.

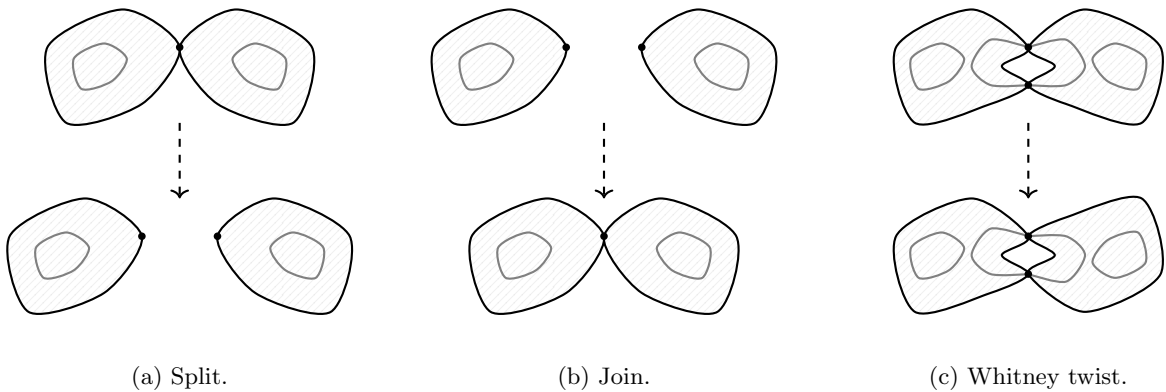
\begin{figure}[ht!]
\centering
\begin{subfigure}[t]{0.32\textwidth}
\begin{center}
	\begin{tikzpicture}[thick, scale=0.5]
    \fill[pattern=north east lines, opacity=.25, inner sep=1pt] plot[smooth cycle] coordinates {(0,0) (2,0.5) (3,2) (1.5,3) (-0.5,2)};
    \draw[thick] plot[smooth cycle] coordinates {(0,0) (2,0.5) (3,2) (1.5,3) (-0.5,2)};
    \draw[gray,thick] plot[smooth cycle] coordinates {(0.4,0.9) (1.2,1.1) (1.6,1.7) (1.0,2.1) (0.2,1.7)};
    
    \fill[pattern=north east lines, opacity=.25, inner sep=1pt] plot[smooth cycle] coordinates {(6,0) (4,0.5) (3,2) (4.5,3) (6.5,2)};
    \draw[thick] plot[smooth cycle] coordinates {(6,0) (4,0.5) (3,2) (4.5,3) (6.5,2)};
    \draw[gray,thick] plot[smooth cycle] coordinates {(5.6,0.9) (4.8,1.1) (4.4,1.7) (5.0,2.1) (5.8,1.7)};
    \node[big dot,] at (3,2) {};

    \draw [dashed,->] (3,0) -- (3,-2);

    \fill[pattern=north east lines, opacity=.25, inner sep=1pt] plot[smooth cycle] coordinates {(-1,-5) (1,-4.5) (2,-3) (0.5,-2) (-1.5,-3)};
    \draw[thick] plot[smooth cycle] coordinates {(-1,-5) (1,-4.5) (2,-3) (0.5,-2) (-1.5,-3)};
    \draw[gray,thick] plot[smooth cycle] coordinates {(-0.6,-4.1) (0.2,-3.9) (0.6,-3.3) (0,-2.9) (-0.8,-3.3)};
    \node[big dot,] at (2,-3) {};
    
    \fill[pattern=north east lines, opacity=.25, inner sep=1pt] plot[smooth cycle] coordinates {(7,-5) (5,-4.5) (4,-3) (5.5,-2) (7.5,-3)};
    \draw[thick] plot[smooth cycle] coordinates {(7,-5) (5,-4.5) (4,-3) (5.5,-2) (7.5,-3)};
    \draw[gray,thick] plot[smooth cycle] coordinates {(6.6,-4.1) (5.8,-3.9) (5.4,-3.3) (6.0,-2.9) (6.8,-3.3)};
    \node[big dot,] at (4,-3) {};
    \end{tikzpicture}	
\end{center}
\caption{
Split.
}
\label{fig:opa}
\end{subfigure}\hfill
\begin{subfigure}[t]{0.32\textwidth}
\begin{center}
    \begin{tikzpicture}[thick, scale=0.5]
    \fill[pattern=north east lines, opacity=.25, inner sep=1pt] plot[smooth cycle] coordinates {(-1,0) (1,0.5) (2,2) (0.5,3) (-1.5,2)};
    \draw[thick] plot[smooth cycle] coordinates {(-1,0) (1,0.5) (2,2) (0.5,3) (-1.5,2)};
    \draw[gray,thick] plot[smooth cycle] coordinates {(-0.6,0.9) (0.2,1.1) (0.6,1.7) (0,2.1) (-0.8,1.7)};
    \node[big dot,] at (2,2) {};

    \fill[pattern=north east lines, opacity=.25, inner sep=1pt] plot[smooth cycle] coordinates {(7,0) (5,0.5) (4,2) (5.5,3) (7.5,2)};
    \draw[thick] plot[smooth cycle] coordinates {(7,0) (5,0.5) (4,2) (5.5,3) (7.5,2)};
    \draw[gray,thick] plot[smooth cycle] coordinates {(6.6,0.9) (5.8,1.1) (5.4,1.7) (6.0,2.1) (6.8,1.7)};
    \node[big dot,] at (4,2) {};

    \draw [dashed,->] (3,0) -- (3,-2);

    \fill[pattern=north east lines, opacity=.25, inner sep=1pt] plot[smooth cycle] coordinates {(0,-5) (2,-4.5) (3,-3) (1.5,-2) (-0.5,-3)};
    \draw[thick] plot[smooth cycle] coordinates {(0,-5) (2,-4.5) (3,-3) (1.5,-2) (-0.5,-3)};
    \draw[gray,thick] plot[smooth cycle] coordinates {(0.4,-4.1) (1.2,-3.9) (1.6,-3.3) (1.0,-2.9) (0.2,-3.3)};

    \fill[pattern=north east lines, opacity=.25, inner sep=1pt] plot[smooth cycle] coordinates {(6,-5) (4,-4.5) (3,-3) (4.5,-2) (6.5,-3)};
    \draw[thick] plot[smooth cycle] coordinates {(6,-5) (4,-4.5) (3,-3) (4.5,-2) (6.5,-3)};
    \draw[gray,thick] plot[smooth cycle] coordinates {(5.6,-4.1) (4.8,-3.9) (4.4,-3.3) (5.0,-2.9) (5.8,-3.3)};
    \node[big dot,] at (3,-3) {};
\end{tikzpicture}
\end{center}
\caption{
Join.
}
\label{fig:opb}
\end{subfigure}\hfill
\begin{subfigure}[t]{0.32\textwidth}
\begin{center}
	\begin{tikzpicture}[thick, scale=0.5]
    \fill[pattern=north east lines, opacity=.25, inner sep=1pt] plot[smooth cycle] coordinates {(0,0) (2,0.5) (3,1) (2.3,1.5) (3,2) (1.5,3) (-0.5,2)};
    \draw[thick] plot[smooth cycle] coordinates {(0,0) (2,0.5) (3,1) (2.3,1.5) (3,2) (1.5,3) (-0.5,2)};
    \draw[gray,thick] plot[smooth cycle] coordinates {(0.1, 0.9) (0.9, 1.1) (1.3, 1.7) (0.7, 2.1) (-0.1, 1.7)};
    \draw[gray,thick] plot[smooth] coordinates {(3,1) (2,1) (1.5,1.5) (2,2.2) (3,2)};

    \fill[pattern=north east lines, opacity=.25, inner sep=1pt] plot[smooth cycle] coordinates {(6,0) (4,0.5) (3,1) (3.7,1.5) (3,2) (4.5,3) (6.5,2)};
    \draw[thick] plot[smooth cycle] coordinates {(6,0) (4,0.5) (3,1) (3.7,1.5) (3,2) (4.5,3) (6.5,2)};
    \draw[gray,thick] plot[smooth cycle] coordinates {(5.9,0.9) (5.1,1.1) (4.7,1.7) (5.3,2.1) (6.1,1.7)};
    \draw[gray,thick] plot[smooth] coordinates {(3,1) (4,1) (4.5,1.5) (4,2.2) (3,2)};
    \node[big dot] at (3,2) {};
    \node[big dot] at (3,1) {};
    
    \draw [dashed,->] (3,0) -- (3,-2);

    \fill[pattern=north east lines, opacity=.25, inner sep=1pt] plot[smooth cycle] coordinates {(0,-5) (2,-4.5) (3,-4) (2.3,-3.5) (3,-3) (1.5,-2) (-0.5,-3)};
    \draw[thick] plot[smooth cycle] coordinates {(0,-5) (2,-4.5) (3,-4) (2.3,-3.5) (3,-3) (1.5,-2) (-0.5,-3)};
    \draw[gray,thick] plot[smooth cycle] coordinates {(0.1,-4.1) (0.9,-3.9) (1.3,-3.3) (0.7,-2.9) (-0.1,-3.3)};
    \draw[gray,thick] plot[smooth] coordinates {(3,-4) (2,-4) (1.5,-3.5) (2,-2.8) (3,-3)};
    
    \fill[pattern=north east lines, opacity=.25, inner sep=1pt] plot[smooth cycle] coordinates {(6,-2) (4,-2.5) (3,-3) (3.7,-3.5) (3,-4) (4.5,-5) (6.5,-4)};
    \draw[thick] plot[smooth cycle] coordinates {(6,-2) (4,-2.5) (3,-3) (3.7,-3.5) (3,-4) (4.5,-5) (6.5,-4)};
    \draw[gray,thick] plot[smooth cycle] coordinates {(5.9,-2.9) (5.1,-3.1) (4.7,-3.7) (5.3,-4.1) (6.1,-3.7)};
    \draw[gray,thick] plot[smooth] coordinates {(3,-3) (4,-3) (4.5,-3.5) (4,-4.2) (3,-4)};
    \node[big dot] at (3,-3) {};
    \node[big dot] at (3,-4) {};
    \end{tikzpicture}	
\end{center}
\caption{
Whitney twist.
}
\label{fig:opc}
\end{subfigure}
\caption{Illustration of Whitney's theorem. Gray lines show how the cycles are transformed for each operation.}
\label{fig:weakly}
\end{figure}

\subsection{Cycle matroids, the rank function, and weak isomorphisms}\label{subsec:cycle_matroids_and_rank}

Matroid theory abstracts the notion of independence among finite sets. The \emph{cycle matroid} of a finite graph $G$ has ground set $E(G)$, its independent sets are the acyclic edge sets, and its rank function is
\[
    r_G(F)=|V(G)|-\#\{\text{$F$-connected components of $G$}\}.
\]
An isomorphism of cycle matroids is a bijection of the edge sets preserving rank, equivalently a cycle-preserving bijection, that is, a weak isomorphism of finite graphs. The cycle matroid remembers much, but not everything. For example, any two trees with the same number of edges have isomorphic cycle matroids, the free matroid where all subsets are independent. Whitney's theorems say precisely that $3$-connected finite graphs are rigid, in the sense that all information is retained, and that the failure of rigidity in general is completely described by splits, joins, and twists.

The cycle matroid of a graphing $\cG=(X,E,\mu)$ was introduced by Lov\'asz~\cite{lovasz2024matroid} through the rank function
\begin{equation}\label{eqn:graphing_rank}
    \rho_{\mathcal{G}}(F)=\mu(X)-\int_X\frac{1}{|[x]_F|}\,d\mu(x),
    \qquad F\subseteq E\text{ measurable},
\end{equation}
where $[x]_F$ denotes the $F$-connected component of $x$.

Contrary to most of the literature, we do not assume that $\mu$ is a probability measure; we only assume that $\mu(X)$ is finite. There are two reasons for this convention. First, for a finite graph with counting measure on its vertices,~\eqref{eqn:graphing_rank} is exactly the usual rank function of the cycle matroid. Second, splits and joins can change the total vertex measure, so this flexibility is natural.

A Borel bijection $\varphi:E(\cG_1)\to E(\cG_2)$ is \textbf{rank-preserving} if $\rho_{\cG_2}(\varphi(F))=\rho_{\cG_1}(F)$ for every Borel $F\subseteq E(\cG_1)$. The following two lemmas rephrase this condition in terms of edge measure, cycles, and hyperfiniteness.

\begin{lem}\label{lem:cylce_preserv}
    Let $\cG_1$ and $\cG_2$ be graphings, and let $\varphi:E(\cG_1)\to E(\cG_2)$ be a rank-preserving Borel bijection. Then $\varphi$ preserves the edge measure and preserves cycles and their lengths modulo null sets.
\end{lem}

B\'erczi, Lov\'asz, and the first and last authors proved in~\cite{berczi2026cycle} that the independent sets of the cycle matroid of a graphing are precisely its acyclic hyperfinite subgraphs. This gives the promised matroidal justification for Definition~\ref{defn:weak_iso}.

\begin{lem}\label{lem:eq_of_rankpres}
    Let $\cG_1$ and $\cG_2$ be graphings, and let $\varphi:E(\cG_1)\to E(\cG_2)$ be an edge-measure-preserving Borel bijection. Then the following are equivalent:
    \begin{enumerate}
        \item $\fg$ preserves the rank function;\label{lemma:eq_of_rankpres_rank}
        \item $\fg$ preserves hyperfinite subforests;\label{lemma:eq_of_rankpres_subforest}
        \item $\fg$ is a weak isomorphism of graphings.\label{lemma:eq_of_rankpres_cyclehf}
    \end{enumerate}
\end{lem}

The interaction between the two parts of Definition~\ref{defn:weak_iso} is particularly striking in the infinitely-ended case. There, under weak $2$-connectivity, the hyperfinite edge sets already determine the finite cycles.

\begin{thm}[Hyperfiniteness determines cycles]\label{thm:hyperfinite_sees_cycle}
    Let $\cG_1$ and $\cG_2$ be infinitely-ended, weakly $2$-connected graphings. Let $\varphi:E(\cG_1)\to E(\cG_2)$ be an edge-measure-preserving Borel bijection such that, for every Borel $\cA\subseteq E(\cG_1)$, the set $\cA$ is hyperfinite if and only if $\varphi(\cA)$ is hyperfinite. Then $\varphi$ preserves finite cycles in both directions modulo null sets, and hence is a weak isomorphism.
\end{thm}

\subsection{Asymptotic Whitney rigidity for finite graph sequences}

The measurable rigidity theorem has a finite asymptotic consequence. Roughly speaking, in the relevant connectivity regime, an almost-bijective edge map that preserves the cycle-matroid rank up to $o(|V|)$ error is induced by a vertex map after $o(|V|)$ modifications. The precise asymptotic notions are introduced in Section~\ref{sec:finite-applications}.

\begin{thm}[Asymptotic Whitney rigidity]\label{thm:finite-whitney-stability}
Let $(G_n)$ and $(H_n)$ be bounded-degree graph sequences without isolated vertices. Assume that $(G_n)$ is asymptotically weakly $3$-connected and asymptotically infinitely-ended. Let $f_n:E(G_n)\to E(H_n)$ be an almost-rank-preserving almost bijection. Then there is an almost isomorphism $g_n:V(G_n)\to V(H_n)$ that induces $f_n$. That is, after deleting negligible sets, $f_n(xy)=g_n(x)g_n(y)$ for every remaining edge $xy$.
\end{thm}

The hypotheses above are natural from the graphing viewpoint, but may be less familiar in finite combinatorics. The following two corollaries give concrete sufficient conditions. First, having essentially large girth and minimum degree at least 3 is enough.

\begin{cor}\label{cor:min_3_deg_large_girth}
    The conclusion of Theorem~\ref{thm:finite-whitney-stability} holds if each $G_n$ has minimum degree $3$ and the sequence has essentially large girth.
\end{cor} 

The conclusion also holds under assumptions that rule out the problematic $2$-ended structure. For a finite graph $G$, let $\operatorname{cq}(G)$ be the largest $k$ for which there exist pairwise disjoint $2$-sets $C_1,\ldots,C_k\subseteq V(G)$ such that $C_i\cup C_j$ is a cutset for every $i\neq j$.

\begin{cor}\label{cor:cut_quadruples}
    The conclusion of Theorem~\ref{thm:finite-whitney-stability} holds if each $G_n$ is connected, $|G_n|\to\infty$ and $\sup_n\operatorname{cq}(G_n)<\infty$. In particular, this condition is satisfied if the $G_n$ are $5$-connected.
\end{cor}

To our knowledge, Theorem~\ref{thm:finite-whitney-stability} is the first use of matroid limit theory to derive a new asymptotic result in finite combinatorics. The finite application also gives a retrospective explanation for the two ingredients in the definition of weak isomorphism. In the asymptotic finite setting preservation of hyperfiniteness emerges naturally as one of several equivalent manifestations of approximate rank preservation (see Proposition~\ref{prop:almost_rank_pres_characterization}). Thus the finite theory helps explain why hyperfiniteness is built into the measurable notion of the cycle matroid \cite{lovasz2024matroid, berczi2026cycle}.

The assumption of $5$-connectivity in Corollary~\ref{cor:cut_quadruples} cannot be weakened to $4$-connectivity. The $2$-ended, but not strongly $3$-connected obstruction appearing for graphings (Example~\ref{ex:ladder}) has a finite asymptotic counterpart (see Example~\ref{ex:finite_4-conn}). We also note that among $3$-connected graphs counterexamples have to be ladder-like. If the $G_n$ are $3$-connected, their limit is automatically weakly $3$-connected, so according to Theorem~\ref{thm:graph_isom_general} the only way the conclusion can fail is if, in the limit, a $2$-ended structure emerges and strong $3$-connectivity fails.

\subsection{Relation to previous work}\label{subsec:related_work}

Whitney's classification was extended to countable, locally finite graphs by Thomassen~\cite{thomassen1982duality}, and the corresponding extension of Tutte's decomposition was developed by Droms, Servatius, and Servatius~\cite{droms1995structure}. We use these results in the $1$- and $2$-ended parts of the classification. Applying them simultaneously and measurably across a graphing is a separate problem, however, and in the infinitely-ended case the graphing rank contains hyperfiniteness information that is not recorded by finite cycles alone.

Among earlier results involving isomorphisms of graphings, the closest ones of
which we are aware come from measured group theory, where the role of
isomorphism is somewhat different. First, there are celebrated results on \emph{orbit equivalence superrigidity} that establish conjugacy of probability measure preserving group actions \cite{furman1999orbit,monod2006orbit, popa2007cocycle}, see e.g.\ \cite{furman2011survey} for a survey. Consequently, this property yields isomorphisms for the graphings whose connected components are the orbits and edges are also measurably decorated by group elements encoding a group action (also known as Schreier graphs). Graphing isomorphism also appears in measured group theory under the name of \emph{isometric orbit equivalence} \cite{joseph2022isometric}, and some related negative examples can be found in \cite{Weilacher}. These results concern graphings carrying additional algebraic labels, or establish existential or negative statements, rather than deciding when two given unlabeled graphings are isomorphic.

Anderson and Bernshteyn recently studied Borel line graphs~\cite{anderson2024borel}. Their main theorem is a Borel analogue of Beineke's characterization of line graphs, rather than a rigidity theorem. Their arguments use Whitney's strong isomorphism theorem for infinite line graphs, namely that, apart from four exceptional finite graphs, isomorphisms of line graphs are induced by isomorphisms of the original graphs. The authors do not pursue this, but in the locally finite Borel setting, a Luzin-Novikov uniformization argument immediately yields a Borel version of Whitney's strong isomorphism theorem. Our rigidity problem is markedly different: a line graph records the full incidence relation between pairs of edges, whereas our weak isomorphisms see only cycles and hyperfiniteness. In our setting, reconstructing edge incidence is the main difficulty, rather than part of the data.

Our work belongs to the recent analytic theory of matroids and their limits based on submodular rank functions and graphings ~\cite{lovasz2023submodular,lovasz2024matroid,berczi2024quotient, berczi2025convergent,berczi2026cycle}. B\'erczi, Imolay, and Schweitzer recently developed a general axiomatic theory of measurable matroids, which contains the cycle matroids of graphings and includes measurable versions of matroid intersection and union as well as Nash--Williams--Tutte type theorems for graphings ~\cite{Measurable_Matroids}. A different, model-theoretic notion of first-order convergence for matroids was studied earlier in~\cite{kardovs2017first}.

Finally, a notion related to $\operatorname{cq}(G)$ for edge cuts appears in work of Kosinas~\cite{kosinas2024computing}, who studies families of pairs of edges such that the union of every two pairs is a $4$-edge cut and proves a cyclic structure for sufficiently large such families. Our condition is the analogous one for vertex pairs; the bi-infinite arrangement of cut pairs in the proof of Corollary~\ref{cor:cut_quadruples} can be viewed as an infinite counterpart of that cyclic structure.

\subsection{Proof ideas}\label{subsec:ideas_and_overview}

{\bf Rigidity.}
We first recall the proof of the rigidity part of Whitney's theorem for finite graphs. Let $\varphi$ be a weak isomorphism of a strongly $3$-connected finite graph, and consider a \emph{wedge}, that is, a path $W$ of length two. By $3$-connectivity, there is a cycle $C$ through the endpoints of $W$ that avoids its middle vertex. Analyzing $\varphi(C\cup W)$ shows that $\varphi(W)$ is again a wedge; see Claim~\ref{lem:w_to_w_cycle}. It follows that $\varphi$ is induced by a graph isomorphism. The same reasoning works for strongly $3$-connected, locally finite graphs with little extra effort. Under weak $3$-connectivity, however, one needs an alternative to cycles: a $3$-regular treeing is weakly $3$-connected and has no cycles, yet Theorem~\ref{thm:graph_isom} says that it is rigid under weak isomorphisms.

Cycles have the property that every individual edge is ``superfluous'' for the rank: deleting one edge does not decrease the rank of the edge set. Infinitely-ended leafless subforests have the same property, and among acyclic subgraphs it characterizes them; see Lemma~\ref{lem:forest_superfluous}. The main challenge is therefore to find cycles or infinitely-ended leafless subforests covering the endpoints of wedges in weakly $3$-connected, infinitely-ended graphings (Theorem~\ref{thm:cycle_or_forest}). We use the weak connectivity assumption and local applications of Menger's theorem to build finite pieces with few leaves on different sides of trifurcations, extend these pieces by a Free Minimal Spanning Forest, and prune the result until it has no leaves. We then adapt the finite argument to show that a weak isomorphism maps wedges to wedges.

\vspace{0.3cm}
\noindent
{\bf Classification.}
For the full Whitney theorem, we also prove that every weakly $2$-connected, infinitely-ended graphing is covered by infinitely-ended leafless forests (Theorem~\ref{thm:union_of_inf_ended_leafless}). We introduce a decomposition of each component into maximal finite graphs with two boundary vertices and develop measurable graph-minor techniques to show that the decomposition is preserved by weak isomorphisms. We call its parts \textbf{bananas}, since their two boundary vertices make them resemble bananas in our pictures. The $1$- and $2$-ended cases use a different set of tools: Tutte decompositions of infinite locally finite graphs and the infinite extension of Whitney's theorem recalled in Subsection~\ref{subsec:locally_finite_tools}.

\vspace{0.3cm}
\noindent
{\bf Finite application.}
The finite application uses ultraproducts, a standard tool in graph limit theory. The principal additional difficulty is that ultraproducts live on non-standard measure spaces, whereas our graphing theorems are formulated on standard Borel spaces. We overcome this by constructing simultaneous standard factors of the two ultraproduct graphings in which any failure of the edge map to preserve incidence remains visible (Proposition~\ref{prop:standard-factor-witness}).

\subsection{Examples and obstructions}\label{subsec:examples}
\hfill

We first present a counterexample to Theorem~\ref{thm:graph_isom} in the $2$-ended case.

\begin{ex}[Nontrivial example of weak isomorphism]\label{ex:ladder} 
Let $\cG_1$ be a graphing whose connected components are infinite ladders with vertex-disjoint diagonals added in each square; see Figure~\ref{fig:infinite_ladder_twist}. For convenience, assume further that $\cG_1$ admits a Borel edge-coloring that distinguishes the two rails, the rungs, and the diagonals of the ladder, and that every second diagonal can be selected measurably on almost every connected component. Let $\cG_2$ be the same as $\cG_1$, but with every second diagonal switched. (As the choice is Borel, $\cG_2$ is also a graphing.) Note that both $\cG_1$ and $\cG_2$ are weakly $3$-connected. Clearly, $\cG_1$ and $\cG_2$ are not isomorphic. Yet the edge-bijection indicated by the coloring in Figure~\ref{fig:infinite_ladder_twist} is a weak isomorphism: the diagonals and rungs are mapped to the corresponding diagonals and rungs; and in every second square, where the diagonal is switched, the two rails of the ladder are also swapped.

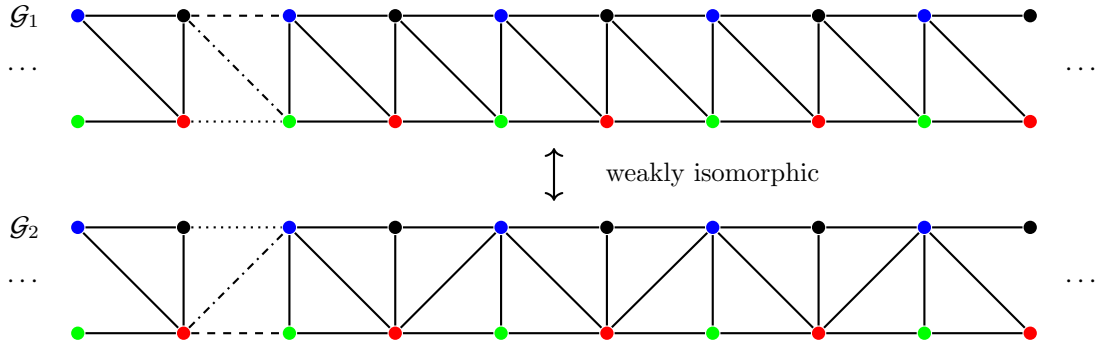
\begin{figure}[htb]
\begin{center}
	\begin{tikzpicture}[thick, scale=0.7]    
    \node[big dot, minimum size=5pt] (a0) at (-9,6) {};
	\node[big dot, minimum size=5pt] (b0) at (-9,4) {};
    \node[big dot, minimum size=5pt] (c0) at (-7,6) {};
    \node[big dot, minimum size=5pt] (d0) at (-7,4) {};
    \node[big dot, minimum size=5pt] (a1) at (-5,6) {};
	\node[big dot, minimum size=5pt] (b1) at (-5,4) {};
    \node[big dot, minimum size=5pt] (c1) at (-3,6) {};
    \node[big dot, minimum size=5pt] (d1) at (-3,4) {};
    \node[big dot, minimum size=5pt] (a2) at (-1,6) {};
	\node[big dot, minimum size=5pt] (b2) at (-1,4) {};
    \node[big dot, minimum size=5pt] (c2) at (1,6) {};
    \node[big dot, minimum size=5pt] (d2) at (1,4) {};
    \node[big dot, minimum size=5pt] (a3) at (3,6) {};
	\node[big dot, minimum size=5pt] (b3) at (3,4) {};
    \node[big dot, minimum size=5pt] (c3) at (5,6) {};
    \node[big dot, minimum size=5pt] (d3) at (5,4) {};
    \node[big dot, minimum size=5pt] (a4) at (7,6) {};
	\node[big dot, minimum size=5pt] (b4) at (7,4) {};
    \node[big dot, minimum size=5pt] (c4) at (9,6) {};
    \node[big dot, minimum size=5pt] (d4) at (9,4) {};

    \draw[red] (a0)--(c0);
    \draw[red] (c0)--(a1);
    \draw[red]  (a1)--(c1)--(a2)--(c2)--(a3)--(c3)--(a4)--(c4);
    \draw[blue] (b0)--(d0);
    \draw[blue] (d0)--(b1);
    \draw[blue] (b1)--(d1)--(b2)--(d2)--(b3)--(d3)--(b4)--(d4);
    \draw (a1)--(b1);
    \draw (a2)--(b2);
    \draw (a3)--(b3);
    \draw (a4)--(b4);
    \draw (c0)--(d0);
    \draw (c1)--(d1);
    \draw (c2)--(d2);
    \draw (c3)--(d3);

    \draw[green] (a0)--(d0);
    \draw[green] (c0)--(b1);
    \draw[green] (a1)--(d1);
    \draw[green] (c1)--(b2);
    \draw[green] (a2)--(d2);
    \draw[green] (c2)--(b3);
    \draw[green] (a3)--(d3);
    \draw[green] (c3)--(b4);
    \draw[green] (a4)--(d4);

    \draw (-10,5) node {$\dots$};
    \draw (-10,6) node {$\cG_1$};
    \draw (10,5) node {$\dots$};

    \draw[<->] (0,2.5) -- (0,3.5);
    \draw (3,3) node {weakly isomorphic};

    \node[big dot, minimum size=5pt] (aa0) at (-9,2) {};
	\node[big dot, minimum size=5pt] (bb0) at (-9,0) {};
    \node[big dot, minimum size=5pt] (cc0) at (-7,2) {};
    \node[big dot, minimum size=5pt] (dd0) at (-7,0) {};
    \node[big dot, minimum size=5pt] (aa1) at (-5,2) {};
	\node[big dot, minimum size=5pt] (bb1) at (-5,0) {};
    \node[big dot, minimum size=5pt] (cc1) at (-3,2) {};
    \node[big dot, minimum size=5pt] (dd1) at (-3,0) {};
    \node[big dot, minimum size=5pt] (aa2) at (-1,2) {};
	\node[big dot, minimum size=5pt] (bb2) at (-1,0) {};
    \node[big dot, minimum size=5pt] (cc2) at (1,2) {};
    \node[big dot, minimum size=5pt] (dd2) at (1,0) {};
    \node[big dot, minimum size=5pt] (aa3) at (3,2) {};
	\node[big dot, minimum size=5pt] (bb3) at (3,0) {};
    \node[big dot, minimum size=5pt] (cc3) at (5,2) {};
    \node[big dot, minimum size=5pt] (dd3) at (5,0) {};
    \node[big dot, minimum size=5pt] (aa4) at (7,2) {};
	\node[big dot, minimum size=5pt] (bb4) at (7,0) {};
    \node[big dot, minimum size=5pt] (cc4) at (9,2) {};
    \node[big dot, minimum size=5pt] (dd4) at (9,0) {};

    \draw[red] (aa0)--(cc0);
    \draw[blue] (cc0)--(aa1);
    \draw[red] (aa1)--(cc1);
    \draw[blue](cc1)--(aa2);
    \draw[red](aa2)--(cc2);
    \draw[blue](cc2)--(aa3);
    \draw[red](aa3)--(cc3);
    \draw[blue](cc3)--(aa4);
    \draw[red](aa4)--(cc4);
    \draw[blue] (bb0)--(dd0);
    \draw[red] (dd0)--(bb1);
    \draw[blue] (bb1)--(dd1);
    \draw[red] (dd1)--(bb2);
    \draw[blue] (bb2)--(dd2);
    \draw[red] (dd2)--(bb3);
    \draw[blue] (bb3)--(dd3);
    \draw[red] (dd3)--(bb4);
    \draw[blue] (bb4)--(dd4);
    \draw (aa1)--(bb1);
    \draw (aa2)--(bb2);
    \draw (aa3)--(bb3);
    \draw (aa4)--(bb4);
    \draw (cc0)--(dd0);
    \draw (cc1)--(dd1);
    \draw (cc2)--(dd2);
    \draw (cc3)--(dd3);

    \draw[green] (aa0)--(dd0);
    \draw[green] (dd0)--(aa1);
    \draw[green] (aa1)--(dd1);
    \draw[green] (dd1)--(aa2);
    \draw[green] (aa2)--(dd2);
    \draw[green] (dd2)--(aa3);
    \draw[green] (aa3)--(dd3);
    \draw[green] (dd3)--(aa4);
    \draw[green] (aa4)--(dd4);

    \draw (-10,1) node {$\dots$};
    \draw (-10,2) node {$\cG_2$};
    \draw (10,1) node {$\dots$};
    
    \end{tikzpicture}	
\end{center}
\caption{A weak isomorphism of weakly $3$-connected, $2$-ended graphings that is not induced by an isomorphism of graphings. \label{fig:infinite_ladder_twist}}
\end{figure}
\end{ex}

\begin{rem}
    Note that the weak isomorphism between $\cG_1$ and $\cG_2$ can be implemented by performing Whitney twists in all components of $\cG_1$ simultaneously at every cut pair connected by a black edge. This is an example of a \emph{2-ended simultaneous Whitney twist}, see Subsection~\ref{subsec:whitney_operations}.
\end{rem}

\begin{ex}[Splitting in infinitely-ended components]\label{ex:traingle_tree}
This example shows that splitting a graphing at cut vertices might not produce a weak isomorphism, as hyperfiniteness of certain edge sets might be altered. Let $\cG_1$ be a graphing whose components are free products of two triangles, see Figure~\ref{fig:split_destroys_hyperfinite}. Considering the triangles themselves as a vertex set and placing an edge between any two triangles that share a vertex yields a $3$-regular treeing. Assume this treeing admits a measurable perfect matching, giving rise to the red vertices displayed in Figure~\ref{fig:split_destroys_hyperfinite}. Splitting at every red vertex would give rise to a graphing $\cG_2$ whose every component is a bi-infinite line of triangles. The induced edge-bijection preserves cycles, but not hyperfiniteness: the entire edge set is hyperfinite in $\cG_2$, but not in $\cG_1$. 
\end{ex}

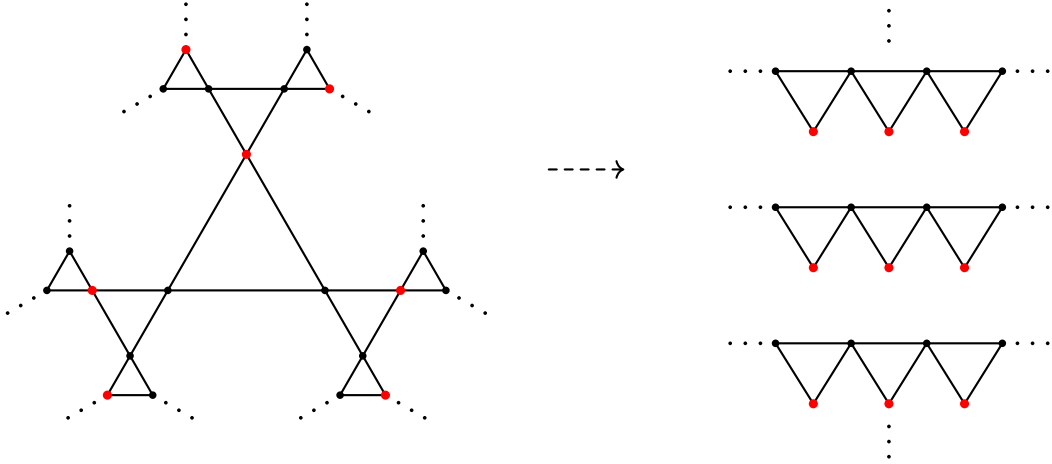
\begin{figure}[htb]
\centering
\makebox[\textwidth][c]{%
\begin{tikzpicture}[
    line join=round,
    line cap=round,
    every path/.style={thick},
    big dot/.style={circle, inner sep=0pt, minimum size=1mm, fill=black},
    red dot/.style={circle, inner sep=0pt, minimum size=1.2mm, fill=red}
]

\def\R{1.2}      
\def\s{1}     
\def\two{0.6}   
\def\d{0.2}     

\newcommand{\trianglerow}[1]{%
    \fill (-0.8,#1) circle (0.7pt);
    \fill (-0.5,#1) circle (0.7pt);
    \fill (-0.2,#1) circle (0.7pt);

    \foreach \x in {0,1,2} {
        \draw (\x,#1) -- (\x+1,#1) -- (\x+0.5,{#1-0.8}) -- cycle;
    }

    \fill (3.2,#1) circle (0.7pt);
    \fill (3.5,#1) circle (0.7pt);
    \fill (3.8,#1) circle (0.7pt);
}

\begin{scope}[shift={(-5,0)}]

    \coordinate (A1) at ({\R*cos(90)},{\R*sin(90)});
    \coordinate (B1) at ({\R*cos(210)},{\R*sin(210)});
    \coordinate (C1) at ({\R*cos(330)},{\R*sin(330)});

    \coordinate (LA1) at ($(A1)+({\s*cos(60)},{\s*sin(60)})$);
    \coordinate (RA1) at ($(A1)+({\s*cos(120)},{\s*sin(120)})$);

    \coordinate (LB1) at ($(B1)+({\s*cos(180)},{\s*sin(180)})$);
    \coordinate (RB1) at ($(B1)+({\s*cos(240)},{\s*sin(240)})$);

    \coordinate (LC1) at ($(C1)+({\s*cos(300)},{\s*sin(300)})$);
    \coordinate (RC1) at ($(C1)+({\s*cos(0)},{\s*sin(0)})$);

    \coordinate (LA1a) at ($(LA1)+({\two*cos(0)},{\two*sin(0)})$);
    \coordinate (LA1b) at ($(LA1)+({\two*cos(60)},{\two*sin(60)})$);

    \coordinate (RA1a) at ($(RA1)+({\two*cos(120)},{\two*sin(120)})$);
    \coordinate (RA1b) at ($(RA1)+({\two*cos(180)},{\two*sin(180)})$);

    \coordinate (LB1a) at ($(LB1)+({\two*cos(120)},{\two*sin(120)})$);
    \coordinate (LB1b) at ($(LB1)+({\two*cos(180)},{\two*sin(180)})$);

    \coordinate (RB1a) at ($(RB1)+({\two*cos(240)},{\two*sin(240)})$);
    \coordinate (RB1b) at ($(RB1)+({\two*cos(300)},{\two*sin(300)})$);

    \coordinate (LC1a) at ($(LC1)+({\two*cos(240)},{\two*sin(240)})$);
    \coordinate (LC1b) at ($(LC1)+({\two*cos(300)},{\two*sin(300)})$);

    \coordinate (RC1a) at ($(RC1)+({\two*cos(0)},{\two*sin(0)})$);
    \coordinate (RC1b) at ($(RC1)+({\two*cos(60)},{\two*sin(60)})$);

    \draw (A1) -- (B1) -- (C1) -- cycle;

    \draw (A1) -- (LA1) -- (RA1) -- cycle;
    \draw (B1) -- (LB1) -- (RB1) -- cycle;
    \draw (C1) -- (LC1) -- (RC1) -- cycle;

    \draw (LA1) -- (LA1a) -- (LA1b) -- cycle;
    \draw (RA1) -- (RA1a) -- (RA1b) -- cycle;

    \draw (LB1) -- (LB1a) -- (LB1b) -- cycle;
    \draw (RB1) -- (RB1a) -- (RB1b) -- cycle;

    \draw (LC1) -- (LC1a) -- (LC1b) -- cycle;
    \draw (RC1) -- (RC1a) -- (RC1b) -- cycle;

    \node[big dot] at (A1) {};
    \node[big dot] at (B1) {};
    \node[big dot] at (C1) {};

    \node[big dot] at (LA1) {};
    \node[big dot] at (RA1) {};
    \node[big dot] at (LB1) {};
    \node[big dot] at (RB1) {};
    \node[big dot] at (LC1) {};
    \node[big dot] at (RC1) {};

    \node[big dot] at (LA1a) {};
    \node[big dot] at (LA1b) {};
    \node[big dot] at (RA1a) {};
    \node[big dot] at (RA1b) {};
    \node[big dot] at (LB1a) {};
    \node[big dot] at (LB1b) {};
    \node[big dot] at (RB1a) {};
    \node[big dot] at (RB1b) {};
    \node[big dot] at (LC1a) {};
    \node[big dot] at (LC1b) {};
    \node[big dot] at (RC1a) {};
    \node[big dot] at (RC1b) {};

    \foreach \k in {1,2,3} {
        \fill ($(LA1a)+({\k*\d*cos(330)},{\k*\d*sin(330)})$) circle (0.7pt);
        \fill ($(LA1b)+({\k*\d*cos(90)},{\k*\d*sin(90)})$)   circle (0.7pt);

        \fill ($(RA1a)+({\k*\d*cos(90)},{\k*\d*sin(90)})$)   circle (0.7pt);
        \fill ($(RA1b)+({\k*\d*cos(210)},{\k*\d*sin(210)})$) circle (0.7pt);

        \fill ($(LB1a)+({\k*\d*cos(90)},{\k*\d*sin(90)})$)   circle (0.7pt);
        \fill ($(LB1b)+({\k*\d*cos(210)},{\k*\d*sin(210)})$) circle (0.7pt);

        \fill ($(RB1a)+({\k*\d*cos(210)},{\k*\d*sin(210)})$) circle (0.7pt);
        \fill ($(RB1b)+({\k*\d*cos(330)},{\k*\d*sin(330)})$) circle (0.7pt);

        \fill ($(LC1a)+({\k*\d*cos(210)},{\k*\d*sin(210)})$) circle (0.7pt);
        \fill ($(LC1b)+({\k*\d*cos(330)},{\k*\d*sin(330)})$) circle (0.7pt);

        \fill ($(RC1a)+({\k*\d*cos(330)},{\k*\d*sin(330)})$) circle (0.7pt);
        \fill ($(RC1b)+({\k*\d*cos(90)},{\k*\d*sin(90)})$)   circle (0.7pt);
    }

    \node[red dot] at (A1)   {};   
    \node[red dot] at (LB1)  {};   
    \node[red dot] at (RC1)  {};   

    \node[red dot] at (LA1a) {};   
    \node[red dot] at (RA1a) {};   
    \node[red dot] at (RB1a) {};   
    \node[red dot] at (LC1b) {};   

\end{scope}

\begin{scope}[shift={(2,0.5)}]

    \foreach \y in {1.8,0,-1.8} {

        \foreach \k in {1,2,3} {
            \fill (-\k*\d,\y) circle (0.7pt);
        }

        \foreach \x in {0,1,2} {
            \draw (\x,\y) -- (\x+1,\y) -- (\x+0.5,\y-0.8) -- cycle;
        }

        \foreach \x in {0,1,2,3} {
            \node[big dot] at (\x,\y) {};
        }

        \foreach \x in {0.5,1.5,2.5} {
            \node[red dot] at (\x,\y-0.8) {};
        }

        \foreach \k in {1,2,3} {
            \fill (3+\k*\d,\y) circle (0.7pt);
        }
    }

    \foreach \k in {1,2,3} {
        \fill (1.5,{2+\k*\d}) circle (0.7pt);
    }

    \foreach \k in {1,2,3} {
        \fill (1.5,{-2.7-\k*\d}) circle (0.7pt);
    }

\end{scope}
\draw[dashed,->] (-1,1) -- (0,1);
\path[use as bounding box] (-9,-3) rectangle (8,3);
\end{tikzpicture}
}
\caption{A split at cut vertices that does not preserve hyperfiniteness.}\label{fig:split_destroys_hyperfinite}
\end{figure}

\begin{ex}    
With little effort, one could build a weakly $3$-connected version of this example by starting from a 3-regular tree, and forming $\cG_1$ by replacing each vertex with a triangle, and each edge by a ladder of length 3, with disjoint diagonals added. Switching the diagonal in every middle square of each ladder defines a graphing $\cG_2$ and an edge-bijection that preserves cycles, but again, not hyperfiniteness. 
\end{ex}

These examples show that preservation of hyperfiniteness is essential in Theorem~\ref{thm:graph_isom}, and that splits and simultaneous Whitney twists cannot be performed without restriction. In Subsection~\ref{subsec:whitney_operations} we allow these operations only when one side of the split or twist is finite, or when the component is $2$-ended. Lemma~\ref{lem:locally_finite_sequence_gives_weak_isom} shows that, with these restrictions, locally finite sequences of operations still induce weak isomorphisms. Theorem~\ref{thm:Whitney_2-isom_intro} shows conversely that such operations suffice to implement every weak isomorphism.

\subsection*{Organization.}
Section~\ref{sec:prelim} collects preliminaries on graphings, rank functions, trifurcations, leafless forests, and double-ray cores. Section~\ref{sec:cycle_forest_cover} constructs covers by cycles and infinitely-ended leafless forests and proves the rigidity theorem and Theorem~\ref{thm:hyperfinite_sees_cycle}. Section~\ref{sec:assembling_tools} develops the countable and measurable decomposition tools needed for the classification, including Tutte and banana decompositions and measurable minor techniques. Section~\ref{sec:general_Whitney_for_graphings} introduces the six measurable Whitney operations and proves Theorem~\ref{thm:Whitney_2-isom_intro}. Finally, Section~\ref{sec:finite-applications} derives the asymptotic finite results using ultraproducts and standard factors.

\subsection*{Disclosure of generative-AI tool use}
All of our proofs before Section~\ref{sec:finite-applications} were done without relying on AI. In Section~\ref{sec:finite-applications} we used ChatGPT 5.6 Sol Pro to implement our preliminary arguments, all of which were then verified and polished by us. During this process, it suggested Lemma~\ref{lem:hyperfinite-rank-partitions} on its own. Our initial approach did not isolate this clean way of reading the hyperfiniteness of a graph sequence from the rank function. We used the same model to review the literature, identify and fix minor gaps in the proofs, and polish the presentation.

\subsection*{Acknowledgments} We thank L\'aszl\'o Lov\'asz and Krist\'of B\'erczi for introducing us to Whitney's theorems in finite combinatorics and for many insightful conversations. We also thank Konrad Wrobel, Miklós Abért, and Clark Lyons for helpful discussions.

M.B.\ and G.T. were supported in part by the ERC Synergy Grant No. 810115 - DYNASNET.

M.B.\ was also partially supported by 
the EKÖP-25 University Research Scholarship Program of the Ministry for Culture and Innovation from the Source of the National Research, Development and Innovation Fund.

G.T.\ was also partially supported by the RTG award grant (DMS-2134107) from the NSF.

L.M.T.\ was supported by the National Research, Development and Innovation Fund grants numbered KKP-139502 and STARTING 150723.

\section{Preliminaries}\label{sec:prelim}
\subsection{Graphings and the rank function} \label{subsec:graphings}

Let $X$ be a standard Borel space and $\mu$ be a finite (not necessary probability) Borel measure on $X$. A \emph{graphing} is a graph $\cG$ with vertex set $V(\cG) = X$ and Borel edge set $E \subseteq \binom{X}{2}$, in which all degrees are finite, and 
\begin{equation*} \label{eqn:graphing}
\int_{A} \deg(B,x) \diff \mu (x) = \int_{B} \deg(A,x) \diff \mu (x)
\end{equation*}
for all measurable sets $A,B \subseteq X$, where $\deg(S,x)$ is the number of edges from $x \in X$ to $S \subseteq X$. This assumption allows one to meaningfully define the \emph{edge measure} $\eta_\cG $ on Borel subsets $F\subseteq E$ by setting
\begin{equation*}
\eta_\cG(F)=\frac{1}{2}\int \deg_F(x)\diff \mu (x),
\end{equation*}
where $\deg_F(x)$ is the degree of $x$ in $F$. Contrary to previous results in \cite{lovasz2024matroid,berczi2026cycle}, we do not assume $\cG$ to have bounded degree; we only require $\cG $ to have finite edge measure.

We say that a property holds for almost every component of a graphing \(\mathcal G=(X,E,\mu)\) if there is an invariant conull Borel set \(X'\subseteq X\) such that every component of \(\mathcal G|_{X'}\) has the property.

An acyclic graphing, that is, a graphing whose almost every component is a tree, is called a \emph{treeing}. When discussing an acyclic subgraph $F$ of a given graphing $\cG$, we shall instead use the terms \emph{subforest}, or simply \emph{forest}, to emphasize that we do not assume that the $F$-connected components coincide with the $\cG$-components.
Recall that a graphing $\cG=(X,E, \mu)$ is \emph{hyperfinite} if, for every $\eps > 0$, there is a Borel set $E_0 \subseteq E$ with $\eta_\cG(E_0) \leq \eps$ such that every connected component of $E\setminus E_0$ is finite. The following results connecting hyperfinite subforests and the rank function are established in \cite{berczi2026cycle}. Although in \cite{berczi2026cycle}, a graphing is required to have bounded degree and have a probability measure on the vertices, the results there generalize without technical complications for this extended definition. 

\begin{lem}[\cite{berczi2026cycle}]\label{lem:hyperfinite_forest_rank}
    Let $\cG=(X, E , \mu)$ be a graphing.
    \begin{enumerate}
        \item $F\subseteq E$ is a hyperfinite forest if and only if $\rho_{\cG}(F) = \eta_{\cG}(F)$.
        \item For any $A\subseteq E$ Borel, we have $\rho_\cG(A)=\sup\{\eta_\cG(F)\mid F\subseteq A\ \text{is a hyperfinite forest}\}$.
    \end{enumerate}
\end{lem}

Like transitive graphs, almost all connected components of graphings have $0$, $1$, $2$, or infinitely many ends. Furthermore, a treeing is hyperfinite if and only if almost every component has at most $2$ ends \cite{Adams:trees_amenability}.

\subsection{Rank-preserving bijections}
In this subsection we characterize rank-preserving bijections as weak isomorphisms, that is, edge-measure-preserving bijections that also preserve cycles and hyperfiniteness. 

\begin{proof}[Proof of Lemma~\ref{lem:cylce_preserv}]
    By Lemma~\ref{lem:hyperfinite_forest_rank} the measure of a hyperfinite subforest is exactly its rank. The edges of a locally finite graphing can be partitioned into hyperfinite subforests (even into partial matchings). Hence $\varphi$ preserves the edge measure.
    
    The cycle-preserving property is clear for finite graphs. Consequently, the lemma also holds for graphings with bounded components.
    Now let us assume that there exists a $k$, and a measurable family $\mC$ of $k$-cycles in $\cG_1$ of positive measure such that the image of any cycle in $\mC$ is not a $k$-cycle in $\cG_2$. We will thin out $\mC$ to make it sparse enough that the cycles in $\mC$ and their images $\varphi(\mC)$ do not meet, and hence do not form infinite components.
    First, we make a Borel graph $\cH$, whose vertices are $\mC$, and two cycles, $C_1,C_2\in \mC$, are connected if $d_{\cG_1}(C_1,C_2)<3$ or $d_{\cG_2}(\varphi(C_1),\varphi(C_2))<3$. (Note that the distance can be infinite.) 
    This is a locally finite Borel graph, and therefore we can color the vertices of $\cH$ properly with countably many colors \cite{kechris1999borel}. One of the color classes has positive measure. Call this family of cycles $\mC_0$. The graphings $(V_1,\cup\mC_0)$ and $(V_2,\varphi(\cup\mC_0))$ have bounded components, and the restriction of $\varphi$ preserves the rank function. Therefore, it also preserves cycles, a contradiction. 
\end{proof}

\begin{proof}[Proof of Lemma~\ref{lem:eq_of_rankpres}]
    
    ($\ref{lemma:eq_of_rankpres_rank}\Leftrightarrow\ref{lemma:eq_of_rankpres_subforest}$) 
    By Lemma~\ref{lem:hyperfinite_forest_rank} the rank function determines the set of hyperfinite subforests, and the set of hyperfinite subforests determines the rank function. Consequently, preserving one is equivalent to preserving the other.
    
    ($\ref{lemma:eq_of_rankpres_cyclehf}\Rightarrow\ref{lemma:eq_of_rankpres_subforest}$) is clear.

    ($\ref{lemma:eq_of_rankpres_rank}\Rightarrow\ref{lemma:eq_of_rankpres_cyclehf}$) Lemma~\ref{lem:cylce_preserv} shows that $\varphi$ preserves the cycles. Now we show that it preserves hyperfiniteness. Let $A \subseteq E(\cG_1)$ such that $\cG_1' = (X_1,A,\mu_1)$ is hyperfinite. We need to prove that $\cG_2'=(X_2, \varphi(A), \mu_2)$ is also hyperfinite. As $\cG'_1$ is hyperfinite, it admits a (hyperfinite) spanning tree $\cT=(X_1, F,\mu_1)$. Then $\fg(F)\subseteq E(\cG_2')$ is a hyperfinite subforest. We claim that $\fg(F)$ spans the components of $\cG_2'$, implying the hyperfiniteness of $\cG_2'$, see e.g.~\cite[Lemma 2.4]{Miller_ends}.

    Indeed, if $\varphi(F)$ fails to span the $\cG_2'$ components, one can find a positive measure subset $D \subseteq \varphi(A) \setminus \varphi(F)$ of $\cG_2'$-edges that connect different components of the forest $\varphi(F)$. Edges $e\in \varphi^{-1}(D)$, on the other hand, define a unique base cycle $C_e$ containing $e$ and using edges from $F$. 
    However, the $\varphi$-image of such a base cycle in $\cG_2'$ contains only a single edge from $D$, and thus cannot be a cycle, a contradiction.
\end{proof}
    
\subsection{Trifurcations}

For a finite connected set $A\subset E(\cG)$ we call the components of $[A]_{\cG}\setminus A$ the \textbf{sides} of $A$.  We use the same terminology for a finite connected vertex set $A\subseteq V(\cG)$, in which case
$[A]_{\cG}\setminus A$ means vertex deletion. For $n\in\N$, an $n$-\textbf{furcation} is such a (nonempty) vertex or edge set $A$ with at least $n$ infinite sides. When we wish to emphasize the
type, we speak of a vertex- or edge-$n$-furcation. For convenience, $2$-furcations and $3$-furcations are also called \textbf{bifurcations} and \textbf{trifurcations}, respectively.%

\begin{lem}\label{lem:cells}
    Let $\cG$ be a locally finite infinitely-ended graphing. Then there is a Borel partition $\cH$ of $\cG$ into cells that enjoys the following properties:
    \begin{enumerate}
        \item each cell is a vertex-trifurcation set in $\cG$; in particular they are finite and connected;
        \item each cell neighbors only finitely many other cells; that is the quotient graph $\cQ:=\cG/\cH$ is locally finite. 
    \end{enumerate}
\end{lem}
\begin{proof}
    Let $\cA$ be a maximal Borel set of disjoint vertex-trifurcation sets in $\cG$ equipped with some Borel ordering. Define $\cH$ to be the collection of the \textbf{Voronoi} cells $\cV(A)$, $A\in \cA$. Here by a Voronoi cell of $A$ we mean the set of all vertices for whom $A$ is the smallest (in the Borel ordering) among the closest (in $\cG$-distance) elements of $\cA$. Notice that from the construction and the mass transport principle it follows that $\cV(A)$ is connected and finite. Since $\cV(A)$ is a finite connected superset of $A$, it is again a vertex-trifurcation. Moreover, by local finiteness of $\cG$ each cell has finite edge-boundary, hence it comes within distance one with at most finitely many other Voronoi cells. Hence $\cH$ is as desired.
\end{proof}

We would like to point out that the quotient graph $\mathcal{Q}$ may not be a tree. In fact, since the cells in the partition are finite, the acylicity of $\cQ$ would imply that $\cG$ admits a spanning tree, which is not true for a general graphing $\cG$. For concreteness, we give the following non-acyclic $\mathcal{Q}$ example.
\begin{ex}($(\Z/3\Z)* \Z$)\label{ex:nontree}
    Consider a graph $\cG$ whose components are given by the free product of a triangle $(\Z/3\Z)$ and a $\Z$-line. Note that each vertex is a trifurcation as its removal disconnects 3 infinite components. In particular, the quotient graph $\cQ$ is isomorphic to $\cG$ and thus not acyclic.
\end{ex}

\subsection{Superfluous edges and leafless forests}\label{subsec:superfluous_edges}
In a finite graph, the rank function counts connected components of edge sets. Consequently, edges that are part of a cycle are superfluous in the sense that their deletion does not decrease the rank. We aim to extend this terminology to graphings. 

Given an edge set $F$ in a graphing, the rank function measures how large the $F$-components are. If the deletion of an edge does not split components, or if it splits an infinite component into two components, but both of those remain infinite, then component sizes remain the same, so the rank does not decrease. The slight inconvenience is that one cannot delete a single edge from an infinite component measurably, so in our definition, we will have to delete a positive measure subset of the edges simultaneously. 

\begin{defn}[Superfluous edges]
    Given the rank function $\rho_{\cG}$ of the cycle matroid of a graphing $\cG=(X,E,\mu)$, we say that an edge set $F\subseteq E$ is \emph{disposable} if $\rho_{\cG}(E)=\rho_{\cG}(E\setminus F)$. 
    Furthermore, we say the edge set \emph{$F \subseteq E$ is superfluous in $\cG$}, if it is a countable union of disposable sets. Equivalently, there exists a Borel coloring of $F$ with countably many colors such that each color class is disposable. On the other hand, we say that \emph{$F \subseteq \cG$ has no disposable edges at all}, if for any positive measure set $F' \subseteq F$ we have $\rho_{\cG}(E\setminus F')<\rho_{\cG}(E)$. 
\end{defn}

The edge set $F$ being superfluous means that almost every edge can be part of a sufficiently sparse set of edges that is disposable. Accordingly, $F$ being superfluous does not mean that $F$ as an edge set is disposable. Disposability and superfluity are clearly preserved under rank-preserving bijections.

Note that these notions are simpler in a finite graph, where it is possible to talk about a single edge rather than a whole edge set. Edges in a set $F\subseteq E$ are superfluous if each of them is disposable on its own. However, it is not clear if there is a natural way to extend this terminology to countable graphs. One may call an edge locally disposable in a graph if its removal does not disconnect a finite component. For instance, in the Schreier graph of a free action of $\Z$ on a standard measurable space (i.e.\ a collection of $\Z$-lines), there are no disposable edges at all. That is, removing any positive measure set of edges decreases the rank, even though a.e.~edge is locally disposable in its component. The following remark shows that the other implication does hold.

\begin{obs}
    Let $\cG$ be a graphing and $\rho$ be the rank function of its  cycle matroid, if the full edge set of $\cG$ is superfluous then a.e.~edge $e\in E(\cG)$ is locally disposable in its connected component.
\end{obs}
\begin{proof}
    Assume towards contradiction that  the set of locally non-disposable edges $E'\subseteq E(\cG)$ has positive measure. That is, after removing any $e\in E'$, either a new finite component is created, or a finite component is cut into two smaller pieces.  By countable additivity, for any Borel coloring $c$ of $E(\cG)$ by countably many colors there must be a color class $I$ such that $\eta_{\cG}(E'\cap I)>0$. So removing $I$ decreases the component size of a positive measure set of vertices, implying $\rho_{\cG}(E\setminus I) <\rho_{\cG}(E)$. Thus, $E$ is not superfluous in $\cG$.
\end{proof}

\begin{rem}\label{rem:cycle_is_super}
If $\cG$ is a union of disjoint cycles then the full edge set of $\cG$ is superfluous. Indeed, there is a Borel coloring such that a.e.~$\cG$-component does not contain two edges of the same color and componentwise each edge is disposable (components are finite graphs and the rank of a path of length $n-1$ is equal to that of a cycle of length $n$).
\end{rem}

The next lemma shows that for a treeing the full edge set being superfluous is equivalent to being leafless and infinitely-ended.

\begin{lem}\label{lem:forest_superfluous}
    Let $\cT$ be a treeing. The following are equivalent:
    \begin{enumerate}
        \item\label{lem:forest_superfluous1} $\cT$ consists of leafless infinitely-ended components and isolated vertices;
        \item\label{lem:forest_superfluous2} The edge set of $\cT$ is 
        superfluous.
    \end{enumerate}
\end{lem}
\begin{proof}
$(\ref{lem:forest_superfluous1} \Rightarrow \ref{lem:forest_superfluous2})$: Let $c:E(\cT)\to\N$ be a Borel coloring of edges of $\cT$, defined on a conull set, such that the path between any two edges with the same color must contain at least two trifurcation vertices. It is enough to show that for any $i\in\N$, letting $I:=c^{-1}(i)$, we have that  $\cT\setminus I$ does not have finite components. 
Suppose the opposite and consider a point $x$ in a finite component $K(x)$ in $\cT\setminus I$. Its component must contain at least one trifurcation vertex of $\cT$, otherwise its boundary consists of edges of color $i$ that are not separated by any trifurcations. Now consider the furthest trifurcation vertex $y$ of $\cT$ from $x$ that belongs to $K(x)$. There are at least two infinite sides of $y$ in $\cT$ that do not contain $x$, and by the mass transport principle, each of those sides contains infinitely many trifurcation vertices. Thus, there are at least two disjoint paths that connect $y$ to neighboring trifurcation vertices that $y$ separates from $x$. By the choice of $y$, these two trifurcations are not in $K(x)$, and so both paths from $y$ must contain an edge from $I$. That is a contradiction, as the only trifurcation along the path connecting these edges is $y$. (We will use similar arguments exploiting trifurcations later in Lemma~\ref{lem:inf_comp} and Lemma~\ref{lem:sparse_removal_from_$3$-connected}.)

$(\ref{lem:forest_superfluous1} \Leftarrow \ref{lem:forest_superfluous2})$: Note that deleting a positive measure set of leaves of a treeing decreases the rank. As the edge set of $\cT$ is 
superfluous, $\cT$ has to be leafless. In particular, a.e.\ $\cT$-component has more than one end. A $2$-ended leafless component is a bi-infinite line, and the collection of such components forms a subgraphing which has no disposable edges at all. Hence the union of the $2$-ended components is null, so almost all components are infinitely-ended.
\end{proof}

\begin{cor}\label{cor:phi_superfluous}
    Let $\cT$ and $\cG$ be graphings and $\fg: E(\cT) \to E(\cG)$ a rank-preserving Borel bijection. If $\cT$ satisfies the property \eqref{lem:forest_superfluous1} from Lemma~\ref{lem:forest_superfluous}, then so does $\cG$.
\end{cor}
\begin{proof}
    By Lemma~\ref{lem:cylce_preserv} $\fg$ is cycle preserving and thus $\cG$ is acyclic. Since $\fg$ preserves the rank function, the edge set of $\cG$ is superfluous, and applying Lemma~\ref{lem:forest_superfluous} concludes the proof.
\end{proof}

\subsection{Free Minimal Spanning Forests}

Free and wired minimal spanning forests on infinite graphs are classical in percolation theory; see, e.g., \cite{lyons2006minimal,LyonsBook}. In recent years such constructions have also appeared in measured combinatorics. In \cite{CTT22} a generalization of the Free Minimal Spanning Forest was introduced for Borel graphs (in the full generality of measure-class preserving locally finite Borel graphs it made more sense to the authors to call the forest ``maximal''). Another example is the application of the Wired Minimal Spanning Forest to graphings from \cite{berczi2026cycle}.

\begin{defn}[FMSF]
Given a Borel linear ordering $<$ on the undirected edges of the graphing $\cG$, the \textbf{Free Minimal Spanning Forest} of $\cG$ with respect to $<$, denoted by $\mathrm{FMSF}_{<}(\cG)$ is the subforest $\cF\subseteq \cG$ obtained by deleting the $<$-largest edge from each simple cycle in $\cG$.
\end{defn}

The following results summarize some of the classical properties of $\mathrm{FMSF}$. We omit their proof, as they are straightforward to check; see e.g. \cite[Chapter 11]{LyonsBook} for a detailed introduction to Free Minimal Spanning Forests.

\begin{lem}\label{lem:FMSF_cycle_invar}
    Suppose $S\subset\cG$ is such that any cycle that intersects $S$ is contained in it. Then $\mathrm{FMSF}_{<}(\cG)$ restricted to $S$ is equal to $\mathrm{FMSF}_{<}(S)$.
\end{lem}
\begin{lem}\label{lem:FMSF_inf_sets}
    If $x$ is in an infinite component in $\cG$ then it is also in an infinite component in ${\mathrm{FMSF}_{<}(\cG)}$. 
\end{lem}

\subsection{The double-ray core}

We will also repeatedly use the following elementary fact about double-ray cores of locally finite graphs.

\begin{defn}[Double-ray core]
Let $G$ be a connected locally finite graph. The double-ray core of $G$ is the subgraph consisting of the edges that lie on a bi-infinite path in $G$.    
\end{defn}

\begin{lem}\label{lem:double_ray_core} Assume $G$ has at least two ends, and let $R$ denote its double-ray core. Then $R$ is connected and weakly $2$-connected, every component of $G\setminus E(R)$ is finite and meets $R$ in exactly one vertex, and $R$ has the same ends as $G$. 

Moreover, if $\cG$ is a locally finite Borel graph, then the graph $\cR$ obtained by taking the double-ray core in each component is Borel.
\end{lem}

In particular, if $G$ is weakly $2$-connected and has at least two ends, then its double-ray core is all of $G$. Equivalently, every edge of such a graph lies on a bi-infinite path. We remark that the lemma also holds for one-ended graphs, with the caveat that $R$ may be empty.

\begin{proof}
Since $G$ has at least two ends, it contains a bi-infinite path, so $R\neq\emptyset$. To see that $R$ is connected, let $D_1,D_2$ be two bi-infinite paths in $R$, and take a shortest path $P$ joining them. If they are disjoint, the interior of $P$ avoids $D_1\cup D_2$, and $P$, together with a suitable ray of each $D_i$, is contained in a bi-infinite path. Thus $P\subseteq R$; the intersecting case is immediate.

Moreover, $R-v$ has no finite component for any $v$. Indeed, if a component $K$ of $R-v$ contains an edge $e$, then a bi-infinite path through $e$, after deleting $v$, contains an infinite ray through $e$ contained in $K$. An isolated vertex cannot occur either, since every vertex of $R$ lies on a bi-infinite path and hence has two distinct incident edges in $R$. Thus $R$ is weakly $2$-connected.

Let $K$ be a component of $G\setminus E(R)$. We claim that $K$ meets $R$ in at most one vertex. Otherwise, taking two consecutive vertices $x,y\in K\cap R$ along a path in $K$, we obtain an $x$--$y$ path $P\subseteq K$ whose internal vertices lie outside $R$. Since $R$ is locally finite and weakly $2$-connected, finite Menger's theorem and Kőnig's infinity lemma give vertex-disjoint rays in $R$ starting at $x$ and $y$. Together with $P$ they form a bi-infinite path, contradicting $E(P)\cap E(R)=\emptyset$.

On the other hand, every such $K$ meets $R$, by connectedness of $G$. Hence it meets $R$ in exactly one vertex, say $x$. If $K$ were infinite, local finiteness would give a ray in $K$ starting at $x$. Joining it at $x$ to either tail of a bi-infinite path in $R$ through $x$ would again produce a bi-infinite path containing an edge of $K$, a contradiction. Thus $K$ is finite.

Consequently, $G$ is obtained from $R$ by attaching finite graphs at single vertices. Every ray of $G$ therefore has a tail in $R$, and such attachments neither create nor identify ends; hence $R$ and $G$ have the same ends.

Finally, in a locally finite Borel graph an edge $e$ belongs to $\cR$ if and only if, for every $n$, there is a finite simple path containing $e$ and extending at least $n$ edges from $e$ in both directions. This is a countable intersection of Borel finite-radius conditions, while the converse follows from K\"onig's infinity lemma. Hence $\cR$ is Borel.
\end{proof}

\section{Leafless forest covers and measurable rigidity}\label{sec:cycle_forest_cover}
In this section we establish our results concerning leafless subforests in infinitely-ended graphings, first in the weakly $2$-connected, and then in the weakly $3$-connected case. The proofs of these will be very similar. We conclude by using these tools to prove Theorem~\ref{thm:graph_isom}. 

Throughout, we will use the Borel partition $\cH$ from Lemma~\ref{lem:cells} into trifurcation cells. We write $\cQ:=\cG/\cH$ for the quotient graph, and $C_\cH(F)$ for the smallest union of cells in $\cH$ that contains the finite graph $F\subseteq V(\cG)$. Note that $C_\cH(F)$ is a trifurcation in $\cG$. Finally, we write $B_n^{\cQ}(A)$ for the $n$-neighborhood of a vertex- or edge set $A$ in $\cQ$. 

\subsection{Leafless forest covers in the weakly $2$-connected case}

Our first goal is to prove the following theorem. 
\begin{thm}\label{thm:union_of_inf_ended_leafless}
    Assume the graphing $\cG$ is weakly $2$-connected and infinitely-ended. Then $E(\cG)$ is a countable union of Borel leafless infinitely-ended forests.
\end{thm}

We will use this result in Section~\ref{sec:general_Whitney_for_graphings}, where we establish our measurable 2-isomorphism theorem. We prove it here well in advance, because its proof is a simpler version of the argument that we will present for covering endpoints of wedges in the weakly $3$-connected case (Theorem~\ref{thm:cycle_or_forest}). We start by establishing a few lemmas.

\begin{lem}\label{lem:bi_tri}
    Let $\cG$ be a locally finite infinitely-ended graphing, and $\cH$ a Borel partition into trifurcation cells, and let $\cQ=\cG/\cH$. Let $\cE$ be a Borel set of finite subgraphs of $\cG$ such that for any $F,G\in \cE$ the $\cQ$-distance between $C_\cH(F)$ and $C_\cH(G)$ is at least $9$.    
    Then any cell $A$ that is at least $\cQ$-distance $1$ (resp.~$\cQ$-distance $2$) from $C_\cH(F)$ for some $F\in\cE$ must be a bifurcation (resp.~trifurcation) in $\cG':=\cG\setminus \{C_{\cH}(F)\}_{F\in\cE}$.
\end{lem}
\begin{proof}
    As $A$ is a trifurcation, in order for $A$ to not be a bifurcation in $\cG'$ at least two of the infinite sides of $A$ have to be contained in $\{C_{\cH}(F)\}_{F\in \cE}$. This entails that there are $F,G\in \cE$ such that $C_\cH(F)$ and $C_\cH(G)$ are within $\cQ$-distance $2$ from each other.
    Similarly, if $A$ is not a trifurcation in $\cG'$ then at least one of its infinite sides was removed, and thus it is a $\cQ$-neighbor of $C_\cH(F)$ for some $F\in\cE$.
\end{proof}

\begin{lem}
\label{lem:inf_comp}
    In the setting of Lemma~\ref{lem:bi_tri}, the removal of $B:=B_3^{\cQ}\left(\bigcup_{F\in \cE}C_{\cH}(F)\right)$ yields only infinite connected components.  
\end{lem}

\begin{proof}
    Suppose towards a contradiction that the removal of $B$ isolates a finite connected component in $\cQ$. Given any cell $A$ in such a connected component, let $A'$ be a cell in the same connected component that is the furthest from $A$ in $\cQ$-distance. (Potentially $A=A'$.) 
    Since $A'$ is a trifurcation in $\cG$, there are at least two cells $C_1$ and $C_2$ that lie on different sides of $A'$ from $A$ and from each other. However, since $A'$ is the furthest cell from $A$, both $C_1$ and $C_2$ must be among the removed cells. Since they are $\cQ$-distance $2$ from each other we conclude that there must be $F,G\in \cE$ such that $C_\cH(F)$ and $C_\cH(G)$ are within $\cQ$-distance $8$, a contradiction. 
\end{proof}

\begin{lem}
\label{lem:disj_path}
    In the setting of Lemma~\ref{lem:bi_tri}, assume further that $\cG$ is weakly $2$-connected and the subgraphs in $\cE$ are singleton edges. 
    Then for any edge $e\in \cE$ there are two vertex-disjoint paths from the endpoints to two cells that lie in two different sides of $C_{\cH}(e)$.
\end{lem}
\begin{proof}
    Let $e=(x,y)$ and consider the collection of infinite sides of $C_{\cH}(e)$; denote them by $S_1,S_2,\ldots,S_k$. Since $C_{\cH}(e)$ is a trifurcation, $k \geq 3$. Add vertices $s_1, s_2, \ldots, s_k$ and connect them to each point in the respective external vertex boundaries of $C_{\cH}(e)$ with $S_1, S_2, \ldots, S_k$. Finally add two points $s$ and $z$ and connect them to $\{s_1,s_2,\ldots,s_k\}$ and $\{x,y\}$, respectively, as shown in Figure~\ref{fig:paths_from_edge}. 
    Notice that the induced graph on
    $C_{\cH}(e)$, its external vertex boundaries contained in $S_1$ and $S_2$, and the new vertices $s_1,s_2,\ldots,s_k,s,z$ is finite and $2$-connected. Indeed, the graph clearly remains connected if one of $s_1,s_2,\ldots,s_k,s,z$ is removed. Further, if the removal of a vertex $w\notin \{s_1,s_2,\ldots,s_k,s,z\}$ disconnected this graph into two components, then the deletion of $w$ in $\cG$ would have created a finite component, contradicting weak $2$-connectedness of $\cG$. Thus, by Menger's theorem there are two vertex-disjoint paths that connect $z$ to $s$. Equivalently, there are two vertex-disjoint paths from the endpoints of $e$ to two cells that lie in two different sides of $C_{\cH}(e)$.
\end{proof}

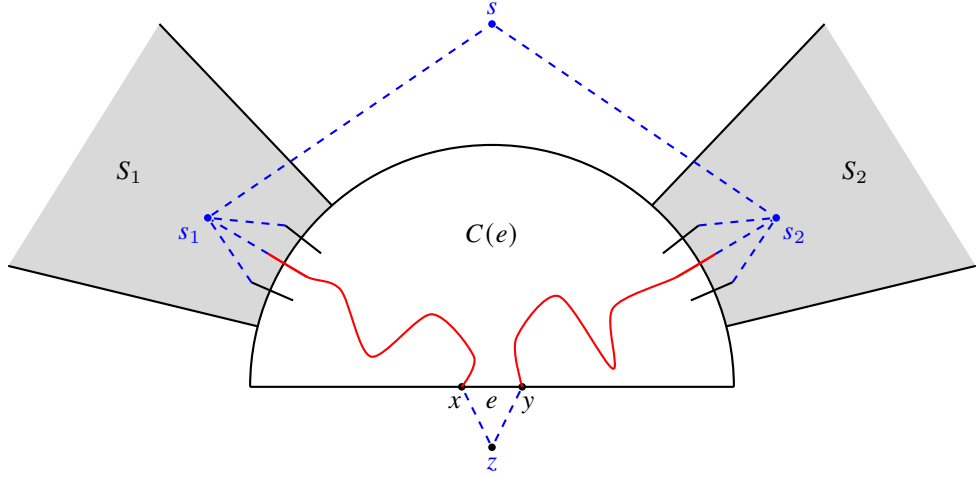
\begin{figure}[htb]
\begin{center}
\begin{tikzpicture}[thick, scale=0.8]
    \draw [name path = S1A] (-3.87298,1) to (-8,2);
    \draw [name path = S1B](-2.64575,3) to (-5.5,6);
    \tikzfillbetween[of=S1B and S1A] {opacity=.15}
    \draw [name path = S2A] (3.87298,1) to (8,2);
    \draw [name path = S2B](2.64575,3) to (5.5,6);
    \tikzfillbetween[of=S2B and S2A] {opacity=.15}
    \draw (-6,3.2) node[above] {$S_1$}; 
    \draw (6,3.2) node[above] {$S_2$};

    \draw [name path = Arc, fill=white](-4,0) arc (180:0:4);
    \node[big dot] (x) at (-0.5,0) {};
    \draw (-0.6,0) node[below] {$x$};
    \node[big dot] (y) at (0.5,0) {};
    \draw (0.6,0) node[below] {$y$};
    \draw (0,0) node[below] {$e$};
    \draw (-4,0) to (4,0);
    \node at (0,2.5) {$C_{\cH}(e)$};
    \draw [name path = T1, red] plot [smooth]  coordinates {(-0.5,0) (-0.3,0.5)  (-1,1.2) (-2,0.5) (-2.5,1.6) (-3.0554, 1.8167)(-3.7,2.2)};
    \draw [name path = T2, red] plot [smooth]  coordinates {(0.5,0) (0.4,0.8)  (1.1,1.5) (2,0.3) (2,1.3) (3.0554, 1.8167) (3.7,2.2)};

    \draw[red] (-3.0554, 1.8167) to (-3.7,2.2);
    \draw (-3.9772, 1.7280) to (-3.2893, 1.4292);
    \draw (-3.4177, 2.6689) to (-2.8266, 2.2073);
   
    \draw[red] (3.0554, 1.8167) to (3.7,2.2);
    \draw (3.9772, 1.7280) to (3.2893, 1.4292);
    \draw (3.4177, 2.6689) to (2.8266, 2.2073);

    \node[big dot,blue] (s1) at (-4.7,2.79463) {};
    \draw (-5,2.79463) node[below,blue] {$s_1$};
    \draw[dashed, blue] (-3.7,2.2) to (s1);
    \draw[dashed, blue] (-3.9772, 1.7280) to (s1);
    \draw[dashed, blue] (-3.4177, 2.6689) to (s1);
    \node[big dot,blue] (s2) at (4.7,2.79463) {};
    \draw (5,2.79463) node[below,blue] {$s_2$};
    \draw[dashed, blue] (3.7,2.2) to (s2);
    \draw[dashed, blue] (3.9772, 1.7280) to (s2);
    \draw[dashed, blue] (3.4177, 2.6689) to (s2);
    \node[big dot,blue] (s) at (0,6) {};
    \draw (s) node[above,blue] {$s$};
    \draw[dashed, blue]  (s1)--(s)--(s2);
    \node[big dot] (z) at (0,-1) {};
    \draw (0,-1) node[below,blue] {$z$};
    \draw[dashed, blue] (y)--(z)--(x);
    \end{tikzpicture}	
\end{center}
\caption{Illustration to the proof of Lemma~\ref{lem:disj_path}. The black edges that connect $C_{\cH}(e)$ with $S_1$ and $S_2$ portray the edge boundary between the respective cells. The blue vertices and dashed edges depict our additions, while the red paths represent the desired paths in the statement.}\label{fig:paths_from_edge}
\end{figure}

\begin{lem}\label{lem:trees_on_edges}
    In the setting of Lemma~\ref{lem:bi_tri} assume further that $\cG$ is weakly $2$-connected and the subgraphs in $\cE$ are singleton edges. Then one can measurably construct a collection of finite subtrees $(T_e)_{e \in \cE}$ such that for each edge $e\in \cE$ 
     \begin{enumerate}
         \item $e\in T_e$;
         \item $T_e$ is contained in the union of $B_1^{\cQ}(C_{\cH}(e))$ and its external vertex boundary;
         \item the only leaves of $T_e$ lie in four distinct infinite sides of $B_1^{\cQ}(C_{\cH}(e))$.
     \end{enumerate}
    \end{lem}

\begin{proof}
     Let $\pi_1$ and $\pi_2$ be the paths given by Lemma~\ref{lem:disj_path}. Then $\pi_1\cup\{e\}\cup\pi_2$ is a simple path that connects two different sides of $C_{\cH}(e)$, call them $S_1$ and $S_2$. Denote its endpoints by $x_1\in S_1$ and $x_2\in S_2$. For $i\in\{1,2\}$, consider the cells $C_{\cH}(x_i)$. Since they are trifurcations in $\cG$, there are at least two distinct sides for each $C_{\cH}(x_i)$ that are infinite and do not contain $e$, call them $S_{i,1}$ and $S_{i,2}$, respectively. Find a subtree $T_{x_i}$ inside $C_{\cH}(x_i)$ that has only three leaves, $x_i$ and two vertices $x_{i,1}$ and $x_{i,2}$ that belong to the internal vertex boundary of $S_i$ with $S_{i,1}$ and $S_{i,2}$, respectively. Such a tree exists by connectedness of the cells $C_{\cH}(x_i)$. Finally, for each $x_{i,j}$ select a single edge $e_{i,j}$ that connects to $S_{i,j}$.
     Define $T_e$ to be the union of $\{e\}$, $\pi_i$, $T_{x_i}$, and $e_{i,j}$ for $i,j\in\{1,2\}$, as shown in Figure~\ref{fig:Te}. 

    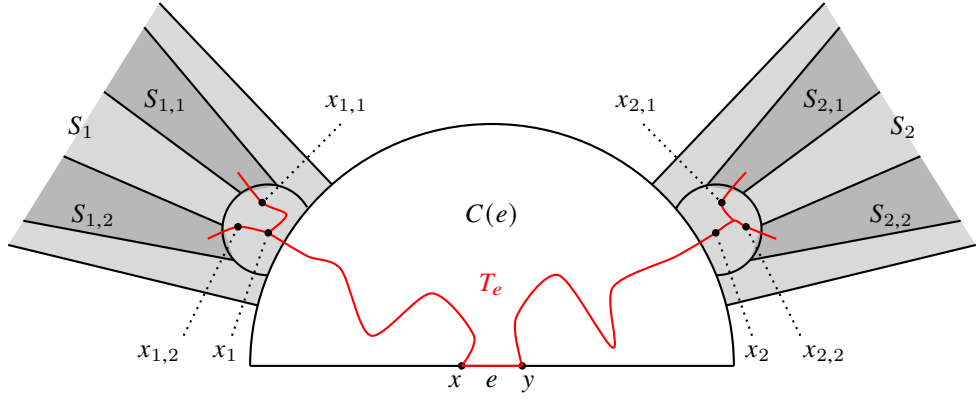
\begin{figure}[htb]
    \begin{center}
    \begin{tikzpicture}[thick, scale=0.8]
        \draw [name path = S1A] (-3.87298,1) to (-8,2);
        \draw [name path = S1B](-2.64575,3) to (-5.5,6);
        \tikzfillbetween[of=S1B and S1A] {opacity=.15}
    
        \draw [name path = S11A] (-4.2654,1.7482) to (-7.75,2.4);
        \draw [name path = S11B] (-4.4539,2.3284) to (-7.0833,3.4667);
        \tikzfillbetween[of=S11B and S11A] {opacity=0.15}    
        \draw [name path = S12A] (-4.1488,2.8568) to (-6.4167,4.5333);
        \draw [name path = S12B] (-3.5521,2.9836) to (-5.7500,5.6000);
        \tikzfillbetween[of=S12B and S12A] {opacity=0.15} 
        \draw [name path = Arc2, black] (-3.708,1.5) arc (270:30:0.75);
        \draw (-6.8,3.6) node[above] {$S_1$};
        \draw (-5.4,4) node[above] {$S_{1,1}$};
        \draw (-6.6,2.1) node[above] {$S_{1,2}$};
    
        \draw [name path = S2A] (3.87298,1) to (8,2);
        \draw [name path = S2B](2.64575,3) to (5.5,6);
        \tikzfillbetween[of=S2B and S2A] {opacity=.15}
        \draw [name path = S21A] (4.2654,1.7482) to (7.75,2.4);
        \draw [name path = S21B] (4.4539,2.3284) to (7.0833,3.4667);
        \tikzfillbetween[of=S21B and S21A] {opacity=0.15}    
        \draw [name path = S22A] (4.1488,2.8568) to (6.4167,4.5333);
        \draw [name path = S22B] (3.5521,2.9836) to (5.7500,5.6000);
        \tikzfillbetween[of=S22B and S22A] {opacity=0.15} 
        \draw [name path = Arc3, black] (3.708,1.5) arc (270:-90:0.75);
        \draw (6.8,3.6) node[above] {$S_2$};
        \draw (5.5,4) node[above] {$S_{2,1}$};
        \draw (6.6,2.1) node[above] {$S_{2,2}$};
    
        \draw [name path = Arc, fill=white](-4,0) arc (180:0:4);
        \node[big dot] (x) at (-0.5,0) {};
        \draw (-0.6,0) node[below] {$x$};
        \node[big dot] (y) at (0.5,0) {};
        \draw (0.6,0) node[below] {$y$};
        \draw (0,0) node[below] {$e$};
        \draw (-4,0) to (4,0);
        \draw[red] (-0.5,0) to (0.5,0);
        \node at (0,2.5) {$C_{\cH}(e)$};
        \draw [name path = T1, red] plot [smooth]  coordinates {(-0.5,0) (-0.3,0.5)  (-1,1.2) (-2,0.5) (-2.5,1.6) (-3.0554, 1.8167)(-3.7,2.2)};
        \draw [name path = T2, red] plot [smooth]  coordinates {(0.5,0) (0.4,0.8)  (1.1,1.5) (2,0.3) (2,1.3) (3.0554, 1.8167) (3.7,2.2)};

        \draw (0,1) node[above,red] {$T_e$};

        \draw [name path = T11, red] plot [smooth]  coordinates {(-3.7,2.2) (-4.2,2.3) (-4.7,2.1)};
        \draw [name path = T12, red] plot [smooth]  coordinates {(-3.7,2.2) (-3.4,2.5) (-3.8,2.7) (-4.2,3.2)};
        \node[big dot] (x1) at (-3.7,2.2) {};
        \draw (-4.4,0.5) node[below] {$x_1$};
        \node[big dot] (x11) at (-4.2,2.3) {};
        \node[big dot] (x12) at (-3.8,2.7) {};
        \draw[dotted] (-4.3,0.5) to (x1);
        \draw[dotted] (-5.1,0.5) to (x11);
        \draw (-5.5,0.5) node[below] {$x_{1,2}$};
        \draw[dotted] (-2.5,4) to (x12);
        \draw (-2.4,4) node[above] {$x_{1,1}$};
    
    
        \draw [name path = T22, red] plot [smooth]  coordinates {(3.7,2.2) (4,2.4) (4.2,2.3) (4.7,2.1)};
        \draw [name path = T21, red] plot [smooth]  coordinates {(4,2.4) (3.8,2.7) (4.2,3.2)};
        \node[big dot] (x21) at (4.2,2.3) {};
        \node[big dot] (x22) at (3.8,2.7) {};
        \node[big dot] (x2) at (3.7,2.2) {};
        \draw[dotted] (4.3,0.5) to (x2);
        \draw (4.4,0.5) node[below] {$x_2$};
        \draw[dotted] (5.1,0.5) to (x21);
        \draw (5.5,0.5) node[below] {$x_{2,2}$};
        \draw[dotted] (2.5,4) to (x22);
        \draw (2.4,4) node[above] {$x_{2,1}$};
        \end{tikzpicture}	
    \end{center}
    \caption{Illustration to the proof of Lemma~\ref{lem:trees_on_edges}. The red tree depicts the desired $T_e$.}\label{fig:Te}
    \end{figure}

\end{proof} 

We are now ready to prove Theorem~\ref{thm:union_of_inf_ended_leafless}. 

\begin{proof}[Proof of Theorem~\ref{thm:union_of_inf_ended_leafless}]
    Let $\cH$ be the partition of $\cG$ into Voronoi cells as in Lemma~\ref{lem:cells}. Consider any Borel coloring $c:E(\cG) \to \N$ such that each color class $\cE_i=c^{-1}(i)$ satisfies the assumptions of Lemma~\ref{lem:bi_tri}. (Such a coloring exists, as $\cQ$ is locally finite.) We will show that for each $i$ there exists a leafless infinitely-ended subforest $\cF_i\subseteq \cG$ that contains $\cE_i$.

    Fix $i\in\N$, and construct the family of trees $(T_e)_{e \in \cE_i}$ by Lemma~\ref{lem:trees_on_edges}. Note that for $e\in \cE_i$ the $2$-neighborhoods $B_2^{\cQ}(C_{\cH}(e))$ are disjoint. Let $\cT_i:=\bigcup_{e\in\cE_i} T_e$ and $\cG'_i\subseteq\cG$ be obtained from $\cG$ by removing all non-$\cT_i$-edges of $B_1^{\cQ}(C_{\cH}(e))$ and their edge boundaries. Clearly, $\cT_i\subseteq \cG'_i$. Moreover, each $T_e$ is a trifurcation in $\cG'_i$ because it has leaves in at least four distinct infinite sides of $B_2^{\cQ}(C_{\cH}(e))$ and each of those sides is infinite in $\cG'_i$ by Lemma~\ref{lem:inf_comp}.  

    Fix an arbitrary Borel linear ordering $<$ of the edges of $\cG'_i$ and let $\cF_i$ be the double-ray core of $\mathrm{FMSF}_{<}(\cG'_i)$. That is, $\cF_i$ is the subgraph consisting of the edges of $\mathrm{FMSF}_{<}(\cG'_i)$ that lie on bi-infinite paths. For each $e\in\cE_i$, the corresponding tree $T_e\subseteq\cT_i$ does not intersect any cycle in $\cG'_i$, and hence $\cT_i\subseteq \mathrm{FMSF}_{<}(\cG'_i)$. Moreover, each leaf of $T_e$ belongs to an infinite side $S$ of $T_e$ in $\cG'_i$, and by Lemmas~\ref{lem:FMSF_cycle_invar} and~\ref{lem:FMSF_inf_sets} this side contains an infinite subtree of $\mathrm{FMSF}_{<}(\cG'_i)\cap S$. Thus $T_e$ remains a trifurcation in $\mathrm{FMSF}_{<}(\cG'_i)$. Consequently, almost every $\mathrm{FMSF}_{<}(\cG'_i)$-component meeting $\cE_i$ is infinitely-ended.
    
    Furthermore, every edge $f\in T_e$ lies on a bi-infinite path in    $\mathrm{FMSF}_{<}(\cG'_i)$. Indeed, the two components of $T_e-f$ each contain a leaf of $T_e$, and the corresponding infinite subtrees of $\mathrm{FMSF}_{<}(\cG'_i)$ provide the two infinite tails. Hence $T_e\subseteq\cF_i$. By Lemma~\ref{lem:double_ray_core}, $\cF_i$ is weakly $2$-connected, and has the same ends as $\mathrm{FMSF}_{<}(\cG'_i)$. Since it is a forest, it is therefore leafless and infinitely-ended. Discarding the components of $\cF_i$ that do not intersect $\cE_i$
    yields the desired forest covering $\cE_i$. The forests $(\cF_i)_{i\in\N}$ together cover all edges of $\cG$.
\end{proof}

\subsection{Covering endpoints of wedges in weakly $3$-connected graphings}

We now turn to the weakly $3$-connected case, and prove our main technical tool used to establish Theorem~\ref{thm:graph_isom}. In a graphing $\cG =(X,E,\mu)$ we say that $(x,z,y)\in X^3$ forms a \textbf{wedge} if $xz\in E$ and $zy \in E$. For such a triple $z$ is called the \textbf{center} of the wedge, while $x$ and $y$ are the \textbf{endpoints}. We denote by $W$ the set of wedges, noting that it carries a natural Borel structure. Suppose $I \subseteq W$ is some measurable set of wedges in $\cG$, then by $\widehat{I}$ we denote the set of their centers.

\begin{thm}\label{thm:cycle_or_forest}
        Let $\cG$ be a locally finite, weakly $3$-connected, infinitely-ended graphing. Then there is a Borel coloring $c:W\to\N$ of the wedges of $\cG$, defined on a conull set, such that for any color $i$, letting $I:=c^{-1}(i)$, there is a subgraph $\cF:=\cF(c,I)\subseteq \cG\setminus{\widehat{I}}$ with either of the following properties:
        \begin{enumerate}
            \item\label{thm:cycle} $\cF$ is a collection of vertex-disjoint cycles such that  for every wedge in $I$ there is a cycle in $\cF$ that contains both of its endpoints. 
            \item\label{thm:forest} $\cF$ is a leafless infinitely-ended forest that covers the endpoints of wedges in $I$.
        \end{enumerate}
\end{thm}

The proof of Theorem~\ref{thm:cycle_or_forest} will follow by an argument similar to the proof of Theorem~\ref{thm:union_of_inf_ended_leafless}. However, a couple of places require additional care. In particular, we also have to keep track of cycles covering endpoints of wedges: the following example illustrates that part~\ref{thm:cycle} in the statement is indeed necessary.
\begin{ex} \label{ex:wedge_cycle}
    Consider the wedge $(x,z,y)$ presented in Figure~\ref{fig:nec_cycles}. Notice that in such a bifurcation after the removal of the vertex $z$, the vertices $x$ and $y$ cannot be covered simultaneously by a leafless forest. However, they do lie on a cycle, e.g.~$(u,x,y,z)$.
    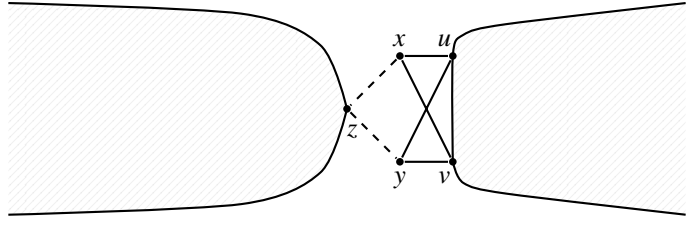
\begin{figure}[h]
    \begin{center}
	\begin{tikzpicture}[thick, scale=0.7]

	\node[big dot] (z) at (0,0) {};
    \draw (0.1,-0.1) node[below] {$z$};
    \node[big dot] (x) at (1,1) {};
    \draw (1,1) node[above] {$x$};
	\node[big dot] (y) at (1,-1) {};
    \draw (1,-1) node[below] {$y$};
    \node[big dot] (u) at (2,1) {};
    \draw (1.85,1) node[above] {$u$};
	\node[big dot] (v) at (2,-1) {};
    \draw (1.85,-1) node[below] {$v$};

	\draw [name path = A] plot [smooth]  coordinates {(0,0) (-0.5,1.2) (-2,1.8) (-6.4,2)};
	\draw [name path = B] plot [smooth]  coordinates {(0,0) (-0.5,-1.2) (-2,-1.8) (-6.4,-2)};
	\tikzfillbetween[of=A and B,split] {pattern=north east lines, opacity=.25, inner sep=1pt}
	\draw [name path = C] plot [smooth] coordinates {(2,-1) (2,1) (2.2,1.4) (3,1.6) (6.4,2)};
	\draw [name path = D] plot [smooth] coordinates {(2,-1) (2.2,-1.4) (3,-1.6) (6.4,-2)};
	\tikzfillbetween[of=C and D,split] {pattern=north east lines, opacity=.25, inner sep=1pt}

    \draw[dashed] (x)--(z)--(y);
    \draw (y)--(u)--(x)--(v)--(y);
    \end{tikzpicture}
    \end{center}
    \caption{Example of a bifurcation induced by $\{u,v,x,y,z\}$ and a wedge $(x,z,y)$ such that after removal of its center the endpoints cannot both be in a leafless forest.}\label{fig:nec_cycles}
\end{figure}
\end{ex}

We first prove that if the removed centers of wedges are sparse enough, the graphing remains weakly $2$-connected and infinitely-ended.

\begin{lem} \label{lem:sparse_removal_from_$3$-connected}
    In the setting of Lemma~\ref{lem:bi_tri} assume further that $\cG$ is weakly $3$-connected and that $\cE$ consists of single vertices.
    Then $\cG''=\cG \setminus \cE$ is weakly $2$-connected and infinitely-ended.  
\end{lem}
\begin{proof}
    By Lemma~\ref{lem:bi_tri}, a.e.~component of the graphing $\cG'=\cG\setminus \{C_{\cH}(v)\}_{v\in \cE}$ contains trifurcations, so a.e.~$\cG'$-component is infinitely-ended. As $E(\cG') \subseteq E(\cG'')$, $\cG''$ is also infinitely-ended. It remains for us to establish weak $2$-connectivity. 

    Assume towards contradiction that removing a vertex $v \in V(\cG)\setminus \cE$ creates a finite $\cG''$-component $F$. Let us list the vertices of the external vertex boundary of $F$ in $\cG$: $\{v, x_1, x_2, \ldots x_k\}$, where $x_i \in \cE$. As $\cG$ is weakly $3$-connected, $k \geq 2$. We claim that $F$ contains a cell from $\cH$. Indeed, pick a path $P$ in $F$ between $x_1$ and $x_2$. Notice that $P$ starts inside $C_{\cH}(x_1)$ and ends up in $C_{\cH}(x_2)$, and let $A,B \in \cH$ be the second and second-to-last $\cH$-cell visited by $P$. By the choice of $\cE$, $A$, $B$, $C_{\cH}(x_1)$, and $C_{\cH}(x_2)$ are distinct. Moreover, $A$ and $B$ cannot contain vertices from $\cE$. Furthermore, either $A$ or $B$ does not contain $v$, and is consequently a subset of $F$.   
    
    We finish the proof similarly to Lemma~\ref{lem:inf_comp}, by picking an $\cH$-cell $A'$ inside $F$ which has maximal distance from $v$. We can then find two $\cH$-cells $C_1$ and $C_2$ that lie on different sides of $A'$ from $v$. Both $C_1$ and $C_2$ must contain a vertex from $\cE$, yielding vertices in $\cE$ whose cells are at $\cQ$-distance 2, a contradiction. 
\end{proof}

\begin{proof}[Proof of Theorem~\ref{thm:cycle_or_forest}]

First, color those wedges $(x,z,y)$ in $\cG$ red, where there is a cycle through $x$ and $y$ avoiding $z$. Color the other wedges blue.

\textbf{Treatment of cycles}.
For each red wedge, choose one of the shortest such cycles to be included in $\cF$, and denote it by $\sigma(x,y)$. We claim that we can find a Borel countable coloring of the red wedges such that for each color class the 1-neighborhoods of the chosen cycles $\sigma(x,y)$ are vertex-disjoint. Indeed, we can add a second coordinate to the color of each red wedge encoding the length of $\sigma(x,y)$, this is still a Borel coloring. By local finiteness of $\cG$, for any given length $k\in\Z_{\ge0}$, the 1-neighborhood of each $\sigma(x,y)$ with length $k$ intersects the 1-neighborhood of at most finitely many other such $k$-cycles. As before, constructing the intersection graph on these cycles yields a locally finite Borel graph, hence it admits a Borel proper coloring. Adding these colors as a third coordinate ensures that the 1-neighborhoods of the chosen cycles are disjoint. After removing the centers of wedges in such a color class, the $\sigma(x,y)$ survive, because avoiding the 1-neighborhoods of other cycles ensures that they do not contain any deleted centers. As even the 1-neighborhoods do not meet, they are clearly vertex-disjoint.

It remains to treat blue wedges. As before, let $\cH$ be the partition of $\cG$ into trifurcation cells as in Lemma~\ref{lem:cells}. Consider any coloring $c:W_{\textrm{blue}} \to \N$ of the wedges in $\cG$ such that each color class satisfies the assumptions of Lemma~\ref{lem:bi_tri}. We now construct a leafless infinitely-ended forest $\mathcal{F}$ for each color class $I$.

{\bf Construction of $T_{x,y}$.}  Analogously to the role of the finite trees $T_e$ from Lemma~\ref{lem:trees_on_edges} in the proof of Theorem~\ref{thm:union_of_inf_ended_leafless}, we will construct for each pair of endpoints $x,y$ of a wedge $(x,z,y)\in I$ a finite forest $T_{x,y}\subset \cG\setminus{\widehat{I}}$ that satisfies the following properties:
     \begin{enumerate}
         \item $x,y\in T_{x,y}$;
         \item $T_{x,y}$ is contained in the union of $B_2^{\mathcal{Q}}(C_{\cH}(x))\cup B_2^{\mathcal{Q}}(C_{\cH}(y))$ and its external vertex boundary;
         \item the only leaves of $T_{x,y}$ lie in at least four distinct infinite sides of $B_2^{\mathcal{Q}}(C_{\cH}(x))\cup B_2^{\mathcal{Q}}(C_{\cH}(y))$.
     \end{enumerate}

Before we proceed we note that since $x$ and $y$ may belong to different connected components of $\cG\setminus{\widehat{I}}$ the associated graph $T_{x,y}$ might be a union of two disjoint finite trees. We also note that this construction will crucially rely on the fact that $x$ and $y$ do not lie on a simple cycle, as Example~\ref{ex:wedge_cycle} shows that constructing the desired $T_{x,y}$ is impossible without such an assumption.

By Lemma~\ref{lem:sparse_removal_from_$3$-connected}, $\cG\setminus \widehat{I}$ is infinitely-ended and weakly $2$-connected. We first construct separate trees $T_x$ and $T_y$ by essentially the same arguments as those used in Lemmas~\ref{lem:disj_path} and \ref{lem:trees_on_edges}, where we replace an edge by a single vertex and work with $B_2^{\cQ}(C_{\cH}(x))\cap [x]_{\cG\setminus \widehat{I}}$ instead of $C_{\cH}(x)\cap [x]_{\cG\setminus \widehat{I}}$ in order to ensure that the chosen set has at least two infinite sides after removal of $\widehat{I}$; see Figure~\ref{fig:Tx}. Such trees $T_x$ satisfy the following properties:
     \begin{enumerate}
         \item $x\in T_x$;
         \item $T_x$ is contained in the union of $B_2^{\cQ}(C_{\cH}(x))\cap [x]_{\cG\setminus \widehat{I}}$ and its external vertex boundary;
         \item the only leaves of $T_x$ lie in four distinct infinite sides of $B_2^{\cQ}(C_{\cH}(x))\cap [x]_{\cG\setminus \widehat{I}}$.
     \end{enumerate}

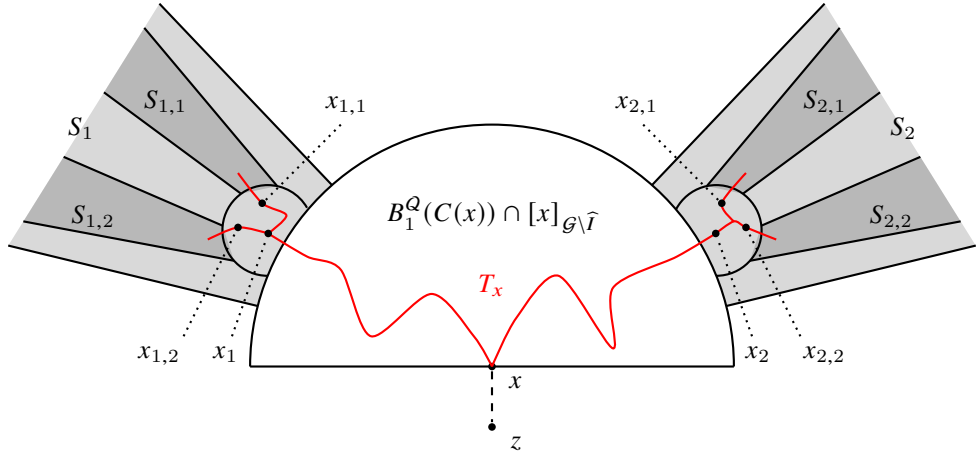
\begin{figure}[htb]
\begin{center}
\begin{tikzpicture}[thick, scale=0.8]
    \draw [name path = S1A] (-3.87298,1) to (-8,2);
    \draw [name path = S1B](-2.64575,3) to (-5.5,6);
    \tikzfillbetween[of=S1B and S1A] {opacity=.15}
    
    \draw [name path = S11A] (-4.2654,1.7482) to (-7.75,2.4);
    \draw [name path = S11B] (-4.4539,2.3284) to (-7.0833,3.4667);
    \tikzfillbetween[of=S11B and S11A] {opacity=0.15}    
    \draw [name path = S12A] (-4.1488,2.8568) to (-6.4167,4.5333);
    \draw [name path = S12B] (-3.5521,2.9836) to (-5.7500,5.6000);
    \tikzfillbetween[of=S12B and S12A] {opacity=0.15} 
    \draw [name path = Arc2, black] (-3.708,1.5) arc (270:30:0.75);
    \draw (-6.8,3.6) node[above] {$S_1$};
    \draw (-5.4,4) node[above] {$S_{1,1}$};
    \draw (-6.6,2.1) node[above] {$S_{1,2}$};
    
    \draw [name path = S2A] (3.87298,1) to (8,2);
    \draw [name path = S2B](2.64575,3) to (5.5,6);
    \tikzfillbetween[of=S2B and S2A] {opacity=.15}
    \draw [name path = S21A] (4.2654,1.7482) to (7.75,2.4);
    \draw [name path = S21B] (4.4539,2.3284) to (7.0833,3.4667);
    \tikzfillbetween[of=S21B and S21A] {opacity=0.15}    
    \draw [name path = S22A] (4.1488,2.8568) to (6.4167,4.5333);
    \draw [name path = S22B] (3.5521,2.9836) to (5.7500,5.6000);
    \tikzfillbetween[of=S22B and S22A] {opacity=0.15} 
    \draw [name path = Arc3, black] (3.708,1.5) arc (270:-90:0.75);
    \draw (6.8,3.6) node[above] {$S_2$};
    \draw (5.5,4) node[above] {$S_{2,1}$};
    \draw (6.6,2.1) node[above] {$S_{2,2}$};
    
    \draw [name path = Arc, fill=white](-4,0) arc (180:0:4);
    \node[big dot] (z) at (0,-1) {};
    \draw (0.4,-1) node[below] {$z$};
    \node[big dot] (x) at (0,0) {};
    \draw (0.4,0) node[below] {$x$};
    \draw (-4,0) to (4,0);
    \draw [dashed] (x) to (z);
    \node at (0,2.5) {$B_1^{\cQ}(C_{\cH}(x))\cap [x]_{\cG\setminus \widehat{I}}$};
    \draw [name path = T1, red] plot [smooth]  coordinates {(0,0) (-0.3,0.5)  (-1,1.2) (-2,0.5) (-2.5,1.6) (-3.0554, 1.8167)(-3.7,2.2)};
    \draw [name path = T2, red] plot [smooth]  coordinates {(0,0) (0.4,0.8)  (1.1,1.5) (2,0.3) (2,1.3) (3.0554, 1.8167) (3.7,2.2)};
    \draw (0,1) node[above,red] {$T_x$};

    \draw [name path = T11, red] plot [smooth]  coordinates {(-3.7,2.2) (-4.2,2.3) (-4.7,2.1)};
    \draw [name path = T12, red] plot [smooth]  coordinates {(-3.7,2.2) (-3.4,2.5) (-3.8,2.7) (-4.2,3.2)};
    \node[big dot] (x1) at (-3.7,2.2) {};
    \draw (-4.4,0.5) node[below] {$x_1$};
    \node[big dot] (x11) at (-4.2,2.3) {};
    \node[big dot] (x12) at (-3.8,2.7) {};
    \draw[dotted] (-4.3,0.5) to (x1);
    \draw[dotted] (-5.1,0.5) to (x11);
    \draw (-5.5,0.5) node[below] {$x_{1,2}$};
    \draw[dotted] (-2.5,4) to (x12);
    \draw (-2.4,4) node[above] {$x_{1,1}$};
    
    
    \draw [name path = T22, red] plot [smooth]  coordinates {(3.7,2.2) (4,2.4) (4.2,2.3) (4.7,2.1)};
    \draw [name path = T21, red] plot [smooth]  coordinates {(4,2.4) (3.8,2.7) (4.2,3.2)};
    \node[big dot] (x21) at (4.2,2.3) {};
    \node[big dot] (x22) at (3.8,2.7) {};
    \node[big dot] (x2) at (3.7,2.2) {};
    \draw[dotted] (4.3,0.5) to (x2);
    \draw (4.4,0.5) node[below] {$x_2$};
    \draw[dotted] (5.1,0.5) to (x21);
    \draw (5.5,0.5) node[below] {$x_{2,2}$};
    \draw[dotted] (2.5,4) to (x22);
    \draw (2.4,4) node[above] {$x_{2,1}$};
    \end{tikzpicture}	
\end{center}
\caption{Illustration of the construction of the desired tree $T_x$ (in red); compare with Figure~\ref{fig:Te}.}\label{fig:Tx}
\end{figure}

Now it remains to consider all possible interactions between $T_x$ and $T_y$ as their union might induce cycles and interfere with the later use of the FMSF (which will be similar to the argument at the end of the proof of Theorem~\ref{thm:union_of_inf_ended_leafless}).

For each wedge $(x,z,y)\in I$ take $T_x$ and $T_y$ as above, keeping the notations as in Figure~\ref{fig:Tx}. We denote the cells that contain vertices of degree 3 by $C_{\cH}(x_1)$ and $C_{\cH}(x_2)$ for $T_x$ and the analogous cells $C_{\cH}(y_1)$ and $C_{\cH}(y_2)$ for $T_y$. Call the parts of $T_x$ (resp.~$T_y$) obtained by removing $x$ from $T_x$ (resp.~$y$ from $T_y$) \emph{branches}. We label these branches based on the cells to which they connect $x$ (resp. $y$), e.g.\ the first branch of $T_x$, denoted $T_x^1$, connects $x$ to $C_{\cH}(x_1)$.

In the case when $T_x\cap T_y=\emptyset$ and all cells $C_{\cH}(x_1),C_{\cH}(x_2),C_{\cH}(y_1),$ and $C_{\cH}(y_2)$ are different, these trees do not interact in a meaningful way, and so we just set $T_{x,y}=T_x\cup T_y$.

If there is a branch of $T_x$ that intersects both branches of $T_y$ then we construct $T_{x,y}$ by taking $T_x$ and replacing a segment in one of its branches with a path that contains $y$, as in Figure~\ref{fig:Txy} on the left. Note that the other branch of $T_x$ cannot intersect this ``detour'' path as it would create a cycle containing $x,y$. Clearly, $T_{x,y}$ is acyclic. Since $C_{\cH}(x_1)\neq C_{\cH}(x_2)$ the leaves of $T_{x,y}$ belong to four distinct infinite sides of $B_2^{\mathcal{Q}}(C_{\cH}(x))\cup B_2^{\mathcal{Q}}(C_{\cH}(y))$. The case when a branch of $T_y$ intersects both branches of $T_x$ is treated in exactly the same way swapping the roles of $x$ and $y$. 

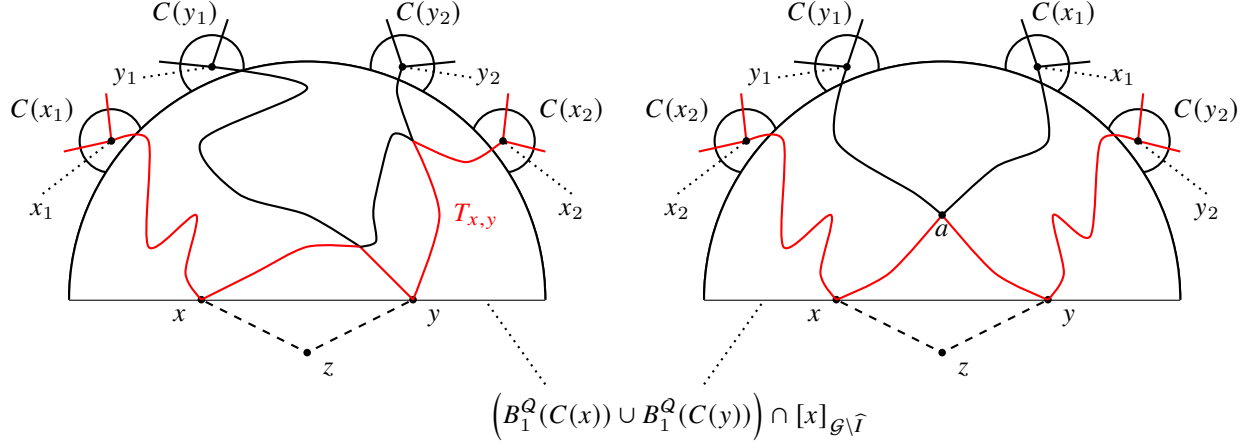
\begin{figure}[htb]
\begin{center}
\begin{tikzpicture}[thick, scale=0.7]    
    \draw [name path = Arc, fill=white](-4.5,0) arc (180:0:4.5);
    \draw (-4.5,0) to (4.5,0);
    \draw [dotted] (4.5,-1.6) to (3, 0.5);
    \draw [dotted] (7.5,-1.6) to (9, 0.5);
    \node at (7,-2.2) {$\left(B_1^{\mathcal{Q}}(C_{\cH}(x))\cup B_1^{\mathcal{Q}}(C_{\cH}(y))\right)\cap [x]_{\mathcal{G}\setminus \widehat{I}}$};
    \draw (-3.7,3) circle (0.6);
    \draw (-5,4) node[below] {$C_{\cH}(x_1)$};
    \draw (3.7,3) circle (0.6);
    \draw (5,4) node[below] {$C_{\cH}(x_2)$};
    \draw (-1.8,4.4) circle (0.6);
    \draw (-2.3,5) node[above] {$C_{\cH}(y_1)$};
    \draw (1.8,4.4) circle (0.6);
    \draw (2.3,5) node[above] {$C_{\cH}(y_2)$};
    \draw [name path = Arc, fill=white](-4.5,0) arc (180:0:4.5);
    
    \node[big dot] (z) at (0,-1) {};
    \draw (0.4,-1) node[below] {$z$};
    \node[big dot] (x) at (-2,0) {};
    \draw (-2.4,0) node[below] {$x$};
    \draw [dashed] (x) to (z);
    \node[big dot] (y) at (2,0) {};
    \draw (2.4,0) node[below] {$y$};
    \draw [dashed] (y) to (z);
    
    \draw[red]  plot [smooth]  coordinates {(x) (-2.3,0.5)  (-2.1,1.6) (-3,1) (-3,3) (-3.7,3)};
    \draw[red]  plot [smooth]  coordinates {(-3.7,3) (-4.6,2.8)};
    \draw[red]  plot [smooth]  coordinates {(-3.7,3) (-3.8,3.9)};
    \node[big dot] (x1) at (-3.7,3) {};
    \draw (-5,2) node[below] {$x_1$};
    \draw[dotted] (-5,2) to (-3.7,3);
    \draw[red]  plot [smooth]  coordinates {(y) (1,1)};
    \draw  plot [smooth]  coordinates {(1,1) (0,1.6) (-1.3,2) (-2,3) (0,4) (-1.8,4.4)};
    \node[big dot] (y1) at (-1.8,4.4) {};
    \draw (-3,4.2) node[left] {$y_1$};
    \draw[dotted] (-3.1,4.2) to (-1.8,4.4);
    \draw  plot [smooth]  coordinates {(-1.8,4.4) (-2.8,4.5)};
    \draw  plot [smooth]  coordinates {(-1.8,4.4) (-1.5,5.3)};
    
    \draw[red] [name path = T2y] plot [smooth]  coordinates {(x) (-1,0.5) (0,1) (1,1)};
    \draw [name path = T2y] plot [smooth]  coordinates {(1,1) (1.3,1.1) (1.25,1.6) (1.3,2) (1.6,3.1) (2,3)};
    \draw[red] [name path = T2y] plot [smooth]  coordinates {(2,3) (3,2.6) (3.7,3)};
    
    \draw [name path = T21y,red] plot [smooth]  coordinates {(3.7,3) (4.5,2.8)};
    \draw [name path = T22y,red] plot [smooth]  coordinates {(3.7,3) (3.8,3.9)};
    \node[big dot] (x2) at (3.7,3) {};
    \draw (5,2) node[below] {$x_2$};
    \draw[dotted] (5,2) to (3.7,3);
    
    \draw plot [smooth]  coordinates {(2,3) (1.7,3.9) (1.8,4.4)};
    \draw[red] plot [smooth]  coordinates {(y) (2.5,1.6) (2,3)};
    \draw plot [smooth]  coordinates {(1.8,4.4) (2.8,4.5)};
    \draw plot [smooth]  coordinates {(1.8,4.4) (1.5,5.3)};
    \node[big dot] (y2) at (1.8,4.4) {};
    \draw (3,4.2) node[right] {$y_2$};
    \draw[dotted] (3.1,4.2) to (1.8,4.4);

    \draw (3.2,2) node[below,red] {$T_{x,y}$};
    \draw [name path = Arc, fill=white](7.5,0) arc (180:0:4.5);
    \draw (7.5,0) to (16.5,0);
    \draw (8.3,3) circle (0.6);
    \draw (7,4) node[below] {$C_{\cH}(x_2)$};
    \draw (15.7,3) circle (0.6);
    \draw (17,4) node[below] {$C_{\cH}(y_2)$};
    \draw (10.2,4.4) circle (0.6);
    \draw (9.7,5) node[above] {$C_{\cH}(y_1)$};
    \draw (13.8,4.4) circle (0.6);
    \draw (14.3,5) node[above] {$C_{\cH}(x_1)$};
    \draw [name path = Arc, fill=white](7.5,0) arc (180:0:4.5);

    \node[big dot] (zR) at (12,-1) {};
    \draw (12.4,-1) node[below] {$z$};
    \node[big dot] (xR) at (10,0) {};
    \draw (9.6,0) node[below] {$x$};
    \draw [dashed] (xR) to (zR);
    \node[big dot] (yR) at (14,0) {};
    \draw (14.4,0) node[below] {$y$};
    \draw [dashed] (yR) to (zR);

    \draw [red] plot [smooth]  coordinates {(xR) (9.7,0.5)  (9.9,1.6) (9,1) (9,3) (8.3,3)};
    \draw [red] plot [smooth]  coordinates {(8.3,3) (7.4,2.8)};
    \draw [red] plot [smooth]  coordinates {(8.3,3) (8.2,3.9)};
    \node[big dot] (x2R) at (8.3,3) {};
    \draw (7,2) node[below] {$x_2$};
    \draw[dotted] (7,2) to (8.3,3);
    \draw[red]  plot [smooth]  coordinates {(xR) (11,0.5)  (12,1.6)};
    \draw plot [smooth]  coordinates {(12,1.6) (13,2.5) (14,3) (13.8,4.4)};

    \node[big dot] (y1R) at (10.2,4.4) {};
    \draw (9,4.2) node[left] {$y_1$};
    \draw[dotted] (8.9,4.2) to (10.2,4.4);
    \draw  plot [smooth]  coordinates {(10.2,4.4) (9.2,4.5)};
    \draw  plot [smooth]  coordinates {(10.2,4.4) (10.5,5.3)};

    \draw[red] plot [smooth] coordinates {(yR) (13,0.5) (12,1.6)};
    \draw  plot [smooth]  coordinates {(12,1.6) (11.5,2) (10,3) (10.2,4.4)};
    \node[big dot] (aR) at (12,1.6) {};
    \draw (12,1.6) node[below] {$a$};

    \draw [name path = T2y, red] plot [smooth]  coordinates {(yR) (14.3,0.5)  (14.1,1.6) (14.8,1.2) (15,3) (15.7,3)};
    \draw [name path = T21y, red] plot [smooth]  coordinates {(15.7,3) (16.5,2.8)};
    \draw [name path = T22y, red] plot [smooth]  coordinates {(15.7,3) (15.8,3.9)};
    \node[big dot] (y2R) at (15.7,3) {};
    \draw (17,2) node[below] {$y_2$};
    \draw[dotted] (17,2) to (15.7,3);

    \draw plot [smooth]  coordinates {(13.8,4.4) (14.8,4.5)};
    \draw plot [smooth]  coordinates {(13.8,4.4) (13.5,5.3)};
    \node[big dot] (x2R) at (13.8,4.4) {};
    \draw (15,4.2) node[right] {$x_1$};
    \draw[dotted] (15.1,4.2) to (13.8,4.4);
    \end{tikzpicture}	
\end{center}
\caption{Illustration of the construction of the desired tree $T_{x,y}$ in several cases. On the left a branch of $T_x$ intersects both branches of  $T_y$; on the right exactly one branch of $T_x$ intersects exactly one branch of $T_y$.}\label{fig:Txy}
\end{figure}

Suppose now that neither branch of these trees intersects both branches of the other tree. In such a case, if a branch of $T_x$ intersects a branch of $T_y$ the other branches of these trees cannot intersect. To see this, without loss of generality, assume that $T^1_x\cap T^1_y\neq\emptyset$. Then, if we had $T^2_x \cap T^2_y\neq \emptyset$, then one could find a path from $x$ to $y$ within both $T^1_x \cup T^1_y$ and $T^2_x \cup T^2_y$. By our assumption that no branch intersects both branches of the other tree, these two paths are vertex-disjoint, except for $x$ and $y$. So one could find a cycle containing $x$ and $y$.
To construct $T_{x,y}$ for this case, let $a\in T_x\cap T_y$ be the closest point to $x$ in $T_x$-distance. Set $T_{x,y}$ to consist of the union of the branches of $T_x$ and $T_y$ that do not contain $a$ and segments that connect $x$ to $a$ and $a$ to $y$ from $T_x$ and $T_y$, respectively, see Figure~\ref{fig:Txy} on the right. Such $T_{x,y}$ is acyclic by the construction, and so it remains to show that $C_{\cH}(x_2)$ and $C_{\cH}(y_2)$ are distinct cells. If $C_{\cH}(x_2)=C_{\cH}(y_2)$, then there is a path within $C_{\cH}(x_2)$ that connects $T_x$ and $T_y$. By the acyclicity of $T_{x,y}$ such a path would create a cycle that contains $x$ and $y$.

Finally, we note that the trick, where we inserted a path within $C_{\cH}(x_2)$, also allows us to treat the last type of interaction between $T_x$ and $T_y$: the case when $T_x\cap T_y=\emptyset$ but not all cells $C_{\cH}(x_1),C_{\cH}(x_2),C_{\cH}(y_1),$ and $C_{\cH}(y_2)$ are different. For example, if $\{C_{\cH}(x_1),C_{\cH}(x_2)\}=\{C_{\cH}(y_1),C_{\cH}(y_2)\}$ inserting such paths in these cells (or equivalently collapsing these cells into single points) immediately yields a cycle that contains $x$ and $y$. Hence, the construction of $T_{x,y}$ in this case is analogous to the one presented in the case with intersections, see Figure~\ref{fig:Txy2}.

\begin{figure}[htb]
\begin{center}
\begin{tikzpicture}[thick, scale=0.6]    
    \draw [name path = Arc, fill=white](-4.5,0) arc (180:0:4.5);
    \draw (-4.5,0) to (4.5,0);
    \draw [dotted] (-3.5,-1.6) to (-4, 0.5);
    \node at (0,-2.2) {$\left(B_1^{\mathcal{Q}}(C_{\cH}(x))\cup B_1^{\mathcal{Q}}(C_{\cH}(y))\right)\cap [x]_{\mathcal{G}\setminus \widehat{I}}$};
    \draw (-3.7,3) circle (0.6);
    \draw (-5,4) node[below] {$C_{\cH}(x_2)$};
    \draw (0,4.5) circle (1.5);
    \draw (0,6) node[above] {$C_{\cH}(y_1)=C_{\cH}(x_1)$};
    \draw (3.7,3) circle (0.6);
    \draw (5,4) node[below] {$C_{\cH}(y_2)$};
    \draw [name path = Arc, fill=white](-4.5,0) arc (180:0:4.5);
    
    \node[big dot] (z) at (0,-1) {};
    \draw (0.4,-1) node[below] {$z$};
    \node[big dot] (x) at (-2,0) {};
    \draw (-2.4,0) node[below] {$x$};
    \draw [dashed] (x) to (z);
    \node[big dot] (y) at (2,0) {};
    \draw (2.4,0) node[below] {$y$};
    \draw [dashed] (y) to (z);
    
    \draw [red] plot [smooth]  coordinates {(x) (-2.3,0.5)  (-2.1,1.6) (-3,1) (-3,3) (-3.7,3)};
    \draw [red] plot [smooth]  coordinates {(-3.7,3) (-4.6,2.8)};
    \draw [red] plot [smooth]  coordinates {(-3.7,3) (-3.8,3.9)};
    \node[big dot] (x2) at (-3.7,3) {};
    \draw (-5,2) node[below] {$x_2$};
    \draw[dotted] (-5,2) to (-3.7,3);
    \node[big dot] (x1) at (-1,5) {};
    \draw (-2.5,4.2) node[left] {$x_1$};
    \draw[dotted] (-2.6,4.4) to (-1,5);
    \draw [red] plot [smooth]  coordinates {(x) (-1,0.5)  (-0,1.6) (-1.3,2) (-2,3) (0,4) (-0.7,4.7)  (-1,5)};
    \draw [blue] plot [smooth]  coordinates {(-1,5) (-2,5)};
    \draw [blue] plot [smooth]  coordinates {(-1,5) (-1,6)};
    \draw [red] plot [smooth]  coordinates {(-0.7,4.7) (-0.5,5.2) (0.5,4.8) (0.7,4.7)};
    \node[big dot] at (-0.7,4.7) {};
    \node[big dot] at (0.7,4.7) {};
    
    \draw [name path = T2y, red] plot [smooth]  coordinates {(y) (2.3,0.5)  (2.1,1.6) (2.8,1.2) (3,3) (3.7,3)};
    \draw [name path = T21y, red] plot [smooth]  coordinates {(3.7,3) (4.5,2.8)};
    \draw [name path = T22y, red] plot [smooth]  coordinates {(3.7,3) (3.8,3.9)};
    \draw [red] plot [smooth]  coordinates {(y) (1,0.5)  (1,1.6) (1.7,3) (0.9,4) (0.7,4.7) (1,5)};
    \draw [blue] plot [smooth]  coordinates {(1,5) (2,5)};
    \draw [blue] plot [smooth]  coordinates {(1,5) (1,6)};
    \node[big dot] (y1) at (1,5) {};
    \draw (2.5,4.2) node[left] {$y_1$};
    \draw[dotted] (2.6,4.4) to (1,5);
    \node[big dot] (y2) at (3.7,3) {};
    \draw (5,2) node[below] {$y_2$};
    \draw[dotted] (5,2) to (3.7,3);

    \end{tikzpicture}	
\end{center}
\caption{Illustration of the construction of the desired tree $T_{x,y}$ in the case where $T_x$ and $T_y$ do not intersect but not all cells $C_{\cH}(x_1),C_{\cH}(x_2),C_{\cH}(y_1),$ and $C_{\cH}(y_2)$ are different. Compare with Figure~\ref{fig:Txy} on the right.}\label{fig:Txy2}
\end{figure}
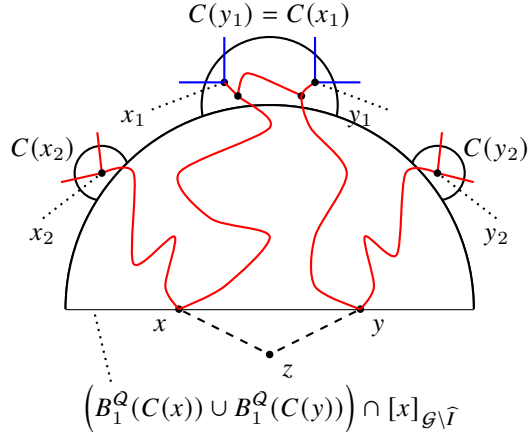

\textbf{Construction of the forest.} 
As in the proof of Lemma~\ref{lem:trees_on_edges}, associate to each endpoint $x$ of a wedge from the color class $I$ a tree $T(x)$ that is given either by $T_x$ or $T_{x,y}$ depending on which case from above applies to $x$. Note that $T(x)$ lies in $B_3^{\mathcal{Q}}(C_{\cH}(x)\cup C_{\cH}(y))$ and enjoys the following two properties:
     \begin{enumerate}
         \item $x\in T(x)$ and it is not a leaf;
         \item the set of leaves $L_x$ of $T(x)$ lies in four distinct infinite sides of $C_{\cH}(T(x)\setminus L_x)$. 
     \end{enumerate}

We now construct the forest $\cF$ as we did in the proof of Theorem~\ref{thm:union_of_inf_ended_leafless}. Let $\cT:=\bigcup_{x\in \operatorname{EndPoints}(I)} T(x)$, and let $\cG'\subseteq\cG$ be obtained from $\cG$ by removing all non-$\cT$-edges of $C_{\cH}(\cT)\cup \partial_E C_{\cH}(\cT)$. Clearly, $\cT\subseteq \cG'$. Moreover, $T(x)$ is a trifurcation set in $\cG'$ because it has leaves in at least four distinct infinite sides of $C_{\cH}(T(x))$ and each of those sides is infinite in $\cG'$ by Lemma~\ref{lem:inf_comp}. Taking $\cF$ to be the double-ray core of $\mathrm{FMSF}_{<}(\cG')$ for an arbitrary Borel linear order $<$ concludes the proof by exactly the same argument as presented in the proof of Theorem~\ref{thm:union_of_inf_ended_leafless}.
\end{proof}

\subsection{Proof of measurable Whitney rigidity}

We are now ready to prove Theorem~\ref{thm:graph_isom}.

\begin{proof}[Proof of Theorem~\ref{thm:graph_isom}]
    Let $\rho_1$ and $\rho_2$ denote the rank functions of the cycle matroids of $\cG_1$ and $\cG_2$, respectively. We split the proof into a sequence of claims.
    \begin{cl}\label{lem:w_to_w}
        It is enough to show that $\fg$ maps wedges of $\cG_1$ to wedges of $\cG_2$.
    \end{cl}

\begin{proof}[Proof of Claim~\ref{lem:w_to_w}]
As usual, we discard invariant null sets so that all the assertions below hold pointwise.

Weak $3$-connectivity implies that every vertex of $\cG_1$ has degree at least $3$. For $x\in V(\cG_1)$, the edges in $\varphi(\operatorname{star}_{\cG_1}(x))$ are pairwise incident. A pairwise incident family of at least three edges in a simple graph either has a common endvertex, or consists of the three edges of a triangle. The latter possibility is ruled out by cycle preservation, since $\operatorname{star}_{\cG_1}(x)$ is acyclic. Hence there is a unique vertex $h(x)\in V(\cG_2)$ incident with every edge of $\varphi(\operatorname{star}_{\cG_1}(x))$. This defines a Borel map $h:V(\cG_1)\to V(\cG_2)$.

First observe that $\cG_2$ has minimum degree at least $2$, and all its components are infinite. Indeed, by Theorem~\ref{thm:union_of_inf_ended_leafless}, write $E(\cG_1)=\bigcup_i \cF_i$, where every $\cF_i$ is a leafless infinitely-ended forest. By Lemma~\ref{lem:forest_superfluous} and Corollary~\ref{cor:phi_superfluous}, $\varphi(\cF_i)$ is again a leafless infinitely-ended forest. Since $\varphi$ is surjective and $\cG_2$ has no isolated vertices, the claim follows.

We also see that the map $h$ is finite-to-one. Indeed, for $z\in V(\cG_2)$ every star centered at some $x\in h^{-1}(z)$ maps injectively to the star of $z$, and each target edge can arise in this way from at most two source stars. The Lusin--Novikov theorem therefore gives a countable Borel partition $V(\cG_1)=\bigsqcup_i A_i$ such that $h|_{A_i}$ is injective for every $i$.

If $A\subseteq V(\cG_1)$ is Borel and $h|_A$ is injective, then $\mu_1(A)=\mu_2(h(A))$. Indeed, choose one incident edge in a Borel way at every $x\in A$, and orient it away from $x$. Orient its $\varphi$-image away from $h(x)$. Since $h|_A$ is injective, this gives exactly one outgoing edge at every point of $h(A)$. Preservation of edge measure implies that the oriented edge sets have the same measure, which in turn means the vertex sets have the same measure. 

Put $m(z):=|h^{-1}(z)|$. Our aim is to show that $m(z)=1$ almost everywhere. Using the equality above on the sets $A_i$ gives
\[
\mu_1(V(\cG_1))
=\int_{V(\cG_2)}m(z) \ d\mu_2(z).
\]
Since all components of both graphings are infinite, the rank formula and rank preservation give 
\[
\mu_1(V(\cG_1))=\rho_1(E(\cG_1))=\rho_2(E(\cG_2))=\mu_2(V(\cG_2)).
\]
Thus $\int m(z) \ d\mu_2(z)=\mu_2(V(\cG_2))$, and it suffices to establish $m(z) \geq 1$. 

Call an edge $xy\in E(\cG_1)$ \emph{folded} if $h(x)=h(y)$. If $h(x)=h(y)=z$ and $\varphi(xy)=zw$, orient $\varphi(xy)$ from $z$ to $w$. Let $q(z)$ and $r(z)$ denote the outdegree and indegree of $z$ in this oriented set of folded edges. The graph on $h^{-1}(z)$ whose edges are the folded edges is a forest, because a cycle in it would be mapped by $\varphi$ to a set of distinct edges all incident with $z$, contradicting cycle preservation. Consequently, for every $z$ with $m(z)>0$, we have $q(z)\leq m(z)-1$. 

Set $Z:=\{z:m(z)=0\}$. Every edge of $\cG_2$ incident with a vertex $z\in Z$ is an incoming folded edge. Indeed, if $\varphi(xy)=uz$ and $h(x)\neq h(y)$, then the two endpoints of $\varphi(xy)$ must be $h(x)$ and $h(y)$, contradicting $m(z)=0$. Hence for $z\in Z$ we have $r(z)=\deg_{\cG_2}(z)\geq2$.

By our estimates above and the mass transport principle, we have

\[2 \mu_2(Z) \leq \int r\,d\mu_2 = \int q\,d\mu_2 \leq \int_{\{m>0\}}(m-1)\,d\mu_2=\int m\,d\mu_2-\mu_2(\{m>0\})
      =\mu_2(Z).\]

Thus $\mu_2(Z)=0$, and $m=1$ almost everywhere. It follows in particular that folded edges form a null set. Hence, for almost every $xy\in E(\cG_1)$, the edge $\varphi(xy)$ is incident with the two distinct vertices $h(x)$ and $h(y)$, and therefore $\varphi(xy)=h(x)h(y)$. Thus $h$ is a measure-preserving bijection modulo null sets inducing $\varphi$.
\end{proof}

Consider the coloring $c$ given by Theorem~\ref{thm:cycle_or_forest} and any color class $I:=c^{-1}(i)$. Without loss of generality we may assume that the $\fg$ images of any two wedges from $I$ are at least distance $4$.
We will show that for any color class $I$, $\fg$ maps wedges in $I$ to wedges. Let $\cF:=\cF(c,I)$ be the subgraph constructed in Theorem~\ref{thm:cycle_or_forest} that corresponds to $I$. There are two cases: $\cF$ is a disjoint union of cycles or $\cF$ is a leafless, infinitely-ended forest (see parts~\ref{thm:cycle} and \ref{thm:forest} of Theorem~\ref{thm:cycle_or_forest}).
    
    \begin{cl}\label{lem:w_to_w_cycle}
    Suppose $\cF$ is as in part~\ref{thm:cycle} of Theorem~\ref{thm:cycle_or_forest}. Then $\fg$ maps wedges from $I$ to wedges.
    \end{cl}
    \begin{proof}[Proof of Claim~\ref{lem:w_to_w_cycle}]
         The proof of this claim is identical to the corresponding proof for finite graphs, see e.g.\ \cite[Problem 15.9]{lovasz2007combinatorial}. We present the argument for completeness.

         Consider a wedge $(x,z,y)\in I$ and suppose $x$ and $y$ are covered by a cycle $C\in\cF$. Let $e_1 =xz$ and $e_2=zy$. As $C \cup \{e_1\}$ contains no cycle other than $C$, the same holds for $\varphi(C)\cup\{\varphi(e_1)\}$. Therefore $\varphi(e_1)$ is not a chord of $C$, hence it has an endpoint $u$ not covered by $\varphi(C)$. On the other hand, $C \cup\{e_1,e_2\}$ contains a cycle $C'$ through $e_1$ and $e_2$. Consequently, $\varphi(C') \subseteq \varphi(C)\cup\{\varphi(e_1),\varphi(e_2)\}$ is a cycle containing $\varphi(e_1)$, hence it has another edge (not equal to $\varphi(e_1)$) covering $u$. This can only be $\varphi(e_2)$, hence $\varphi(e_1)$ and $\varphi(e_2)$ are adjacent in $\cG_2$. That is, they form a wedge.
    \end{proof}

    \begin{cl}\label{lem:w_to_w_forest}
    Suppose $\cF$ is as in part~\ref{thm:forest} of Theorem~\ref{thm:cycle_or_forest}. Then $\fg$ maps wedges from $I$ to wedges.
    \end{cl}
    \begin{proof}[Proof of Claim~\ref{lem:w_to_w_forest}]
        Our argument here emulates the previous, classical one, but uses the infinitely-ended leafless forests and the notion of superfluous/non-superfluous edges in the same way as the previous argument used cycles and their chords.
        
        By Lemma~\ref{lem:forest_superfluous}, $E(\cF)$ is $\rho_1$-superfluous. By Corollary~\ref{cor:phi_superfluous} the same holds for $\fg(\cF)$ with respect to $\rho_2$. First, for each wedge in $I$ decide which edge is `left' and which is `right'. (As wedges are finite, we can choose measurably.) Denote the sets of left and right edges by $L(I)$ and $R(I)$, respectively. Now, a.e.\ edge in $L(I)$ has exactly one endvertex in $\cF$, in other words, these edges form leaves in $\cF\cup L(I)$, and hence $L(I)$ has no disposable edges at all in $\cF\cup L(I)$. Consequently, $\fg(L(I))$ has the same property inside $\varphi(\cF\cup L(I))$. The edges in $\fg(L(I))$ cannot have common endvertices by the distance assumption; and by the previous observation, they cannot have both of their endvertices be covered by the infinitely-ended leafless forest $\varphi(\cF)$, as in that case the set of such edges would be disposable in $\varphi(\cF\cup L(I))$. Hence, $\fg(L(I))$ are leaves in $\varphi(\cF\cup L(I))$. (Notice that we do not argue at this point that one endvertex of an edge from $\varphi(L(I))$ is covered by $\varphi(\cF)$. A priori, it could happen that $\varphi(L(I))$ contains edges that are isolated edges in $\varphi(\cF\cup L(I))$.)

        The same holds for $\fg(R(I))$. On the other hand, the set $\cF\cup L(I)\cup R(I)$ is $\rho_1$-superfluous. Hence, the same holds for $\varphi(\cF\cup L(I)\cup R(I))$, in particular, it has no leaves. Let $e \in L(I)$, and let $f$ denote the pair of $e$ in $R(I)$. By the distance assumption, the endvertices of $\varphi(e)$ cannot be covered by any edge from $\varphi(R(I))$, except $\varphi(f)$. But $\varphi(e)$ is not a leaf in $\varphi(\cF\cup L(I)\cup R(I))$, and we saw that at least one of its endvertices is \emph{not} covered by $\varphi(\cF)$, so it has to be covered by $\varphi(f)$, that is, $\varphi(e)$ and $\varphi(f)$ form a wedge. (Notice that now we also see that the other endvertex of $\varphi(e)$ has to be covered by $\varphi(\cF)$, so in fact there were no isolated edges in $\varphi(\cF\cup L(I))$.)
    \end{proof}
    This concludes the proof of Theorem~\ref{thm:graph_isom}.
\end{proof}

\subsection{Hyperfiniteness determines cycles} \label{subsec:hyperfinite_determines_cycle}

In this subsection, we prove Theorem~\ref{thm:hyperfinite_sees_cycle}, that is, that between weakly $2$-connected, infinitely-ended graphings, preservation of hyperfiniteness already implies preservation of cycles. Our approach is based on the following observation.

\begin{obs} \label{obs:cycle_property}
Consider a Borel family $\mC$ of cycles of positive measure and a hyperfinite subset $\cB$ such that $|\cB\cap C|=|C|-1$ for all $C\in\mC$. For each $C\in\mC$, the unique edge of $C\setminus \cB$ has its endpoints joined by the path $C\cap \cB$. Consequently, $\cB$ and $\cB\cup\bigcup\mC$ induce exactly the same connectedness relation on the vertices, and therefore the edge set $\cB\cup\bigcup\mC$ is also hyperfinite.    
\end{obs}

Our goal is to find a contradiction if the images of cycles in $\mC$ fail to have this property. The main technical step in the proof of   Theorem~\ref{thm:hyperfinite_sees_cycle} will be the following lemma.

\begin{lem} \label{lem:construction_step_laci}
Let $\cG$ be weakly $2$-connected and infinitely-ended, and let $\mathscr P$ be a Borel family of pairwise vertex-disjoint non-empty finite forests of uniformly bounded size, whose union has positive edge measure. After restricting to a positive-measure subfamily, one can choose an edge
$d_P\in P$ for every $P\in\mathscr P$ in a Borel fashion and find a Borel forest $\cT \subseteq \cS \subseteq
E(\cG)$ such that 
\begin{enumerate}
    \item $d_P\in \cT$ for every $P\in \mathscr P$;
    \item $\cT$ is infinitely-ended and leafless;
    \item \label{constuction_still_infty}all the components of $\cT\setminus\{d_P:\ P\in \mathscr P\}$ are still infinite;
    \item $P\subseteq \cS$ for every $P\in \mathscr P$;
    \item all components of $\cS$ are infinite.
\end{enumerate}
\end{lem}

To prove this lemma, our main tool will again be a suitable cover by infinitely-ended leafless forests, which we get as a corollary 
of Theorem~\ref{thm:union_of_inf_ended_leafless}. Given an infinitely-ended leafless tree $T$, we call a path connecting neighboring trifurcation vertices a \emph{segment} of
$T$.

\begin{cor}[of Theorem~\ref{thm:union_of_inf_ended_leafless}]\label{cor:union_of_inf_ended_leafless_sparsed_laci}
Assume the graphing $\cG$ is weakly $2$-connected and infinitely-ended, and let $K\in\N$ be given. Then there exist Borel leafless infinitely-ended forests  $\cF_1,\cF_2,\dots$ in $\cG$ such that for almost every $e\in E(\cG)$ there exists an $n$ such that $e\in \cF_n$ and, inside the $\cF_n$-component of $e$, only the segment of $e$ contains edges that are within $\cG$-distance $K$ from $e$. That is, $B_{\cG}(e,K) \cap [e]_{\cF_n} \subseteq \sigma_{\cF_n}(e)$, where $\sigma_{\cF_n}(e)$ denotes the segment containing $e$ in $\cF_n$. 
\end{cor}

\begin{proof}[Proof of Corollary~\ref{cor:union_of_inf_ended_leafless_sparsed_laci}]
The real content is the existence of such $\cF_n$'s; the measurable choice is then automatic by choosing the smallest possible $n$.

Take one of the countably many forests $\cT=\cT_n$ from Theorem~\ref{thm:union_of_inf_ended_leafless} and an edge $e\in\cT$. Collapse every segment of $\cT$ to an edge, obtaining the trifurcation tree $Q$ of $\cT$. Consider the finite set of edges in the $\cT$-component of $e$ that are at most $\cG$-distance $K$ from $e$, and color their full segments blue. Let $H_e$ be the convex hull of these blue segments in $Q$. In particular, $H_e$ is a finite subtree containing the segment of $e$.

Choose a finite path $\gamma_e$ in $Q$ containing the segment of $e$, whose two endpoints lie outside $H_e$. At every vertex of $\gamma_e\cap H_e$, remove all incident segments except the two belonging to $\gamma_e$. This does not create leaves, as vertices on $\gamma_e\cap H_e$ have degree $2$ afterwards, while every other affected vertex loses at most one incident segment and hence still has degree at least $2$.

Moreover, every blue segment that remains in the same component as $e$ must lie on $\gamma_e$. Indeed, if a blue segment lies outside $\gamma_e$, then the unique path in $H_e$ joining it to the segment of $e$ leaves $\gamma_e$ through an edge deleted by the procedure. Since all trifurcation vertices of $\gamma_e\cap H_e$ have been reduced to degree $2$, all blue segments that remain in the component of $e$ now belong to the same segment as $e$.

It remains to perform these modifications simultaneously. Partition the edges $e$ according to the size of $H_e$. For each fixed size $m$, put a Borel graph on the corresponding edges by joining $e_1$ and $e_2$ whenever the $Q$-distance between $H_{e_1}$ and $H_{e_2}$ is at most $4$. As $m$ is fixed, this is a locally finite Borel graph, so it admits a countable Borel proper coloring. For each color class perform the above removal procedure simultaneously. The corresponding hulls are separated by at least four segments, so each trifurcation vertex is affected by at most one procedure. Thus the remaining forest is leafless; moreover, the separation leaves infinitely many unaffected trifurcations in every component, so it is still infinitely-ended.

Doing this for every $\cT_n$, every possible $m$, and every color class gives a countable family of Borel leafless infinitely-ended forests satisfying the required property. Relabeling this family as $\cF_1,\cF_2,\dots$ concludes the proof.
\end{proof}

Finally, we need a short lemma controlling the effect of deleting a few segments of an infinitely-ended leafless tree and then pruning the remaining forest until it becomes leafless again. 

\begin{lem}\label{lem:pruning_tree_laci}
Let $T$ be a locally finite tree with minimum degree at least $3$, and $F \subseteq E(T)$ with $|F|=k$. We delete the edges in $F$, and prune the remaining forest $T \setminus F$ until it has no leaves. Denote by $F'$ the edges that the pruning deletes. Then
$F'$ has size at most $k-1$ and it is contained in the $\lfloor\log_2(k) \rfloor$-ball around $F$.
\end{lem}

\begin{proof}
We use the following isoperimetric inequality: for every nonempty finite \(A\subseteq V(T)\), $|\partial_E A|\geq |A|+2c(A)$, where \(c(A)\) is the number of components of \(T[A]\). Starting with \(T\setminus F\), repeatedly delete the unique edge incident to a leaf. Suppose that \(r\geq1\) such deletions have been made, and let \(A\) be the set of the \(r\) vertices that became isolated. Since \(T[A]\) is a forest, $|E(T[A])|=r-c(A)$. Thus the number of edges of \(T\) incident to \(A\) is at least 
\[
|E(T[A])|+|\partial_E A|\geq r-c(A)+r+2c(A)\geq 2r+1.
\]
All these edges are among the $k+r$ edges deleted, hence $k+r\geq2r+1,$ and therefore \(r\leq k-1\). So the procedure terminates after at most $k-1$ deletions. If an edge deleted by the pruning lies at distance $d$ from $F$, then tracing the pruning dependencies backward yields a binary tree of depth $d$ with at least $2^d$ distinct edges in $F$, hence $d \leq \lfloor\log_2(k) \rfloor$.
\end{proof}

We are now ready to prove Lemma~\ref{lem:construction_step_laci} and then Theorem~\ref{thm:hyperfinite_sees_cycle}.

\begin{proof}[Proof of Lemma~\ref{lem:construction_step_laci}] 
Note that we do not assume that a forest $P \in \mP$ contains edges only from one component of $\cG$. Let us define the componentwise diameter of $P$, denoted $\text{diam}_{\text{cw}}(P)$, as the supremum of $\text{diam}(P \cap O)$ as $O$ ranges over the connected components of $\cG$. As $\text{diam}_{\text{cw}}(P)$ is a natural number for each $P \in \mP$, by passing to a positive measure subset of $\mathscr P$, we can assume that $\text{diam}_{\text{cw}}(P) < K$ for some uniform bound $K$.

One of the forests $\cF=\cF_n$ from Corollary~\ref{cor:union_of_inf_ended_leafless_sparsed_laci} intersects $\mathscr P$ in a positive edge measure set. Reduce $\mathscr P$ to this positive measure subset, and from each such $P\in\mathscr P$ choose $d^0_P\in P\cap \cF$ in a Borel fashion.  For $P\in\mathscr P$, let $m(P)$ denote the number of segments of $\cF$ that have a common vertex with $P$. Since $P$ is finite and $\cF$ is locally finite, $m(P)<\infty$. Thus, by passing to a positive measure subfamily of $\mathscr P$, we may assume that $m(P)\leq M$ for some uniform $M\in\N$. Set $r= \lfloor\log_2 M\rfloor$.

We now thin out $\mathscr P$ one last time. Say that $P_1,P_2\in\mathscr P$ are incompatible if they touch segments in the same $\cF$-component whose distance in the trifurcation tree is at most $2r+3$. The resulting Borel graph on $\mathscr P$ is locally finite, so by passing to a positive measure color class we may assume that no two members of $\mathscr P$ are incompatible.

Our idea would be to take $d = d^0$, and the infinitely-ended leafless forest $\cT=\cF$, and then extend $\cF \cup \bigcup\mP$ to a large forest $\cS$. Unfortunately, $\cF\cup \bigcup\mP$ might not be acyclic. In that case, we modify $\cF$ and the choice $d$.

Fix some $P \in \mP$, and consider its connected components. It may happen that some components of $P$ touch an $\cF$-component which does not contain $d^0_P$. In such components, we delete all segments of $\cF$ that have a common vertex with $P$, then prune these components until they have no leaves. The pruning removes complete segments of $\cF$, and by Lemma~\ref{lem:pruning_tree_laci} (applied to the trifurcation tree), the pruning terminates after finitely many steps, and by our assumption of the $P$'s being $(2r+3)$-sparse in the trifurcation graph, the deletions cannot overlap. Therefore we can execute them simultaneously.

It remains to treat the component of $d^0_P$ in $\cF$, which we denote by $M$. If $P\cup M$ is not acyclic, we replace the appropriate segment of $M$ using edges from $P$, and choose $d_P$ from among them. More precisely, denote by $L$ the segment containing $d^0_P$ between two neighboring trifurcation vertices in $M$, denoted $x$ and $y$. $P$ can touch $M$ only at vertices along $L$, because $\text{diam}_{\text{cw}}(P) <K$ and all edges of $\cF$ outside $L$ are at $\cG$-distance at least $K+1$ from $d^0_P$. Extend $P$ into a maximal forest inside $L\cup P$, denoted by $H$. Denote the path from $x$ to $y$ inside $H$ by $L'$, and replace $L$ in $M$ by $L'$. Clearly, $L'$ contains an edge from $P$ (otherwise $P \cup M$ would be acyclic), choose one of them to be $d_P$. Now $M \cup P$ is acyclic. Again, the finite sets in $\mP$ are sparse, so we can do this local modification simultaneously at each $P$, the effects are disjoint.

The resulting tree $\cT$ is infinitely-ended and leafless, and contains $d_P$ for all $P \in \mP$. Furthermore, removing the $d_P$'s cannot create finite components in $\cT$, as there are at least 3 unaltered trifurcation vertices between the modified parts. Therefore~\ref{constuction_still_infty} is also satisfied. Each $P \in \mathscr P$ touches at most one component of $\cT$ (the one containing $d_P$), so the local acyclicity of $M \cup P$ implies the global acyclicity of $\cT \cup \bigcup \mathscr P$.

Finally, we construct $\cS$ by an FMSF argument. In $\cG$ choose a Borel linear order $<$ in which every edge of $ \cT \cup \bigcup \mathscr P$ precedes every edge outside it, and let $\cS= \mathrm{FMSF}_{<}(\cG)$. Lemma~\ref{lem:FMSF_inf_sets} implies that $\cS$ has infinite components. 
\end{proof}

\begin{proof}[Proof of Theorem~\ref{thm:hyperfinite_sees_cycle}]
We prove by induction on $k$ that both $\varphi$ and $\varphi^{-1}$ preserve $k$-cycles. There is nothing to prove for $k<3$. Suppose the
assertion holds for all smaller lengths and that there is a positive-measure Borel family $\mathscr C$ of $k$-cycles of
$\cG_1$ whose $\varphi$-images are not $k$-cycles in $\cG_2$. Then, for every $C\in\mathscr C$, the
edge set $\varphi(C)$ is acyclic in $\cG_2$.  Indeed, $\varphi(C)$ itself is not a cycle, and a $\cG_2$-cycle properly
contained in $\varphi(C)$ would, by the induction hypothesis, have a preimage that is a $\cG_1$-cycle
properly contained in the $\cG_1$-cycle $C$, a contradiction.

Put a Borel graph on $\mathscr C$ by joining two members when their edge
sets meet either at a vertex in $\cG_1$ or their images meet at a vertex in $\cG_2$.  This graph is locally
finite.  After a countable Borel coloring and restriction to a
positive-measure color class, the members of $\mathscr C$  and their images are vertex-disjoint in the appropriate graphings. Thus  in $\cG_2$, $\mP=\{\varphi(C) : C \in \mathscr C\}$
is a Borel family of pairwise vertex-disjoint forests with at most $k$ edges. Applying Lemma~\ref{lem:construction_step_laci}, we get a positive measure restriction of $\mP$, and write
\[
    \mathcal{D}:=\{d_P:P\in\mP\},
    \qquad
    \cR:=\bigcup_{P\in\mathscr P}(P\setminus\{d_P\}).
\]
We have $\mathcal{D}\subseteq \cT$ and $\cR \cup \mathcal{D} \subseteq \cS$. As all the components of $\cT\setminus \mathcal{D}$ are infinite, so are the components of $\cS\setminus \mathcal{D}$. Take $\cR\subseteq \cS\setminus \mathcal{D}$. $\cR$ is a union of vertex-disjoint finite forests. Extend $\cR$ into a hyperfinite forest $\cB$ with infinite components inside $\cS\setminus \mathcal{D}$. This can be done by taking the WMSF (Wired Minimal Spanning Forest) with minimal weight on $\cR$; see for example \cite[Corollary 4.3]{berczi2026cycle} for more details.

Now $\cB$ is hyperfinite, $|\cB \cap P|=|P\setminus\{d_P\}|=|P|-1$ for all $P \in \mP$, and all $\cB$-components are infinite. Furthermore, $\cB\cup \mathcal{D}\subseteq \cS$ is still acyclic. Moreover, $\cB$ and $\cB \cup \mathcal{D}$ cover the same vertices, namely the ones also covered by $\cS$. Thus, $\cB \cup \mathcal{D}$ cannot be hyperfinite, as the average degree of $\cB$ on the covered vertices is already 2, and therefore the average degree of $\cB \cup \mathcal{D}$ is strictly larger than 2. This is a contradiction, as by Observation~\ref{obs:cycle_property},  $\varphi^{-1}(\cB\cup \mathcal{D})=\varphi^{-1}(\cB)\cup\bigcup \mathscr C$ is hyperfinite.    
\end{proof}

\section{Decomposition and minor tools for the classification} \label{sec:assembling_tools}

\subsection{Whitney's theorem and Tutte decomposition for locally finite graphs} \label{subsec:locally_finite_tools}

In this subsection, we introduce extensions of standard tools from graph theory to locally finite graphs. We will mostly use them to treat the 1- and 2-ended cases of Theorem~\ref{thm:Whitney_2-isom_intro}, which are simpler than the infinitely-ended case, because the argument works without assuming that the edge-bijection preserves hyperfiniteness of edge sets. We try to keep the exposition brief and provide references instead of technical details where possible.

Unsurprisingly, Whitney's theorem has already been extended to countable, locally finite graphs (without a measurable structure).

\begin{thm}[Thomassen, {\cite[Theorem 4.1]{thomassen1982duality}}] \label{thm:thomassen}
    Let $G$ and $H$ be countable $2$-connected graphs, and $\varphi:E(G) \to E(H)$ a cycle-preserving bijection. Then $H$ is isomorphic to a graph $H'$ obtained by a (possibly empty) sequence of Whitney twists of $G$. 
    
    Moreover, the sequence of switches implements $\varphi$, that is, if $\psi:E(G) \to E(H')$ denotes the edge bijection defined by the sequence of switches, and $\theta: E(H')\to E(H)$ the edge bijection induced by the isomorphism, then $\varphi=\theta \circ \psi$.
\end{thm}

\begin{rem}
    The sequence of twists might be infinite, but at every vertex edges can only be switched finitely many times. The ``moreover'' part is not stated in Thomassen's work, but follows from the proof given there.    
\end{rem}

We aim to make use of this result in the measurable setting, applying it componentwise to a graphing. To be able to do so measurably across all components, we will exploit the fact that this sequence of twists is unique in some sense. Formally, we will upgrade Theorem~\ref{thm:thomassen}, and show that cycle-preserving edge bijections respect the so-called \emph{Tutte decomposition} of graphs.

Given two (disjoint) graphs $H_1$ and $H_2$ with distinguished oriented edges $e=(v_1,v_2)$ and $f=(u_1,u_2)$, the \emph{amalgam of $H_1$ and $H_2$ along the virtual edges $e \leftrightarrow f$} is obtained by identifying $v_1$ with $u_1$ and $v_2$ with $u_2$ in the disjoint union, and removing $e$ and $f$. Tutte showed that any finite $2$-connected graph $G$ has a unique representation as the amalgam of cycles, $k$-links (i.e.\ $k$ parallel edges between two vertices), and $3$-connected graphs \cite{tutte1966connectivity}. The amalgamated graphs are called the Tutte components of $G$. The amalgamations define a tree on the Tutte components, called the Tutte tree. To have uniqueness, one has to assume that no cycles are adjacent in the Tutte tree, and similarly for $k$-links. 

The connection between Whitney's theorem and the Tutte decomposition in the
finite setting was exploited by Truemper \cite{truemper80whitney} to give a short proof
of Whitney's 2-isomorphism theorem. The extension of Tutte's result to the locally finite case was done by Droms, Servatius, and Servatius. 
\begin{thm}[Droms, Servatius, and Servatius, {\cite[Theorem 1]{droms1995structure}}] \label{thm:droms_servatius_servatius}
    Every locally finite $2$-connected graph admits a unique Tutte decomposition. Moreover, for each vertex, there is a finite subtree of the Tutte tree in which all edges at \(v\) are already unamalgamated.
\end{thm}

\begin{rem}
    The condition in the ``moreover'' part ensures that the amalgam given by a Tutte tree becomes locally finite, even if there are infinitely many amalgamations. In \cite{droms1995structure} the authors introduce another technical assumption on Tutte trees, which they need to ensure $2$-connectedness of the amalgam. Both of these assumptions are automatically satisfied by the decompositions produced by Theorem~\ref{thm:droms_servatius_servatius}, but only the first will play a role in our proof. 
\end{rem}

Observe that one can only perform Whitney twists of a $2$-connected graph along cut pairs, which correspond to either an amalgamation in the Tutte tree, or a cut pair in one of those Tutte components that are cycles. In either case, the Tutte decomposition does not change, only the orientation of the amalgamation, or the order of the edges along the cycle. Consequently, Theorems~\ref{thm:thomassen} and \ref{thm:droms_servatius_servatius} have the following corollary.

\begin{lem} \label{lemma:tutte_decomp_preserved}
    Let $G$ and $H$ be $2$-connected locally finite (countable) graphs, and $\varphi:E(G) \to E(H)$ a cycle-preserving bijection. Then $\varphi$ respects the Tutte decompositions of $G$ and $H$. That is, $\varphi$ naturally extends to the virtual edges; maps Tutte components of $G$ to those of $H$, inducing an isomorphism of the Tutte trees; and preserves cycles inside each Tutte component.
\end{lem}

\begin{rem}
    One might expect that Lemma~\ref{lemma:tutte_decomp_preserved} is a more fundamental result that one needs to establish during the proof of Theorem~\ref{thm:thomassen}, not obtain it as a corollary.  Surprisingly, however, the situation appears to be reversed: we are not aware of a proof of Lemma~\ref{lemma:tutte_decomp_preserved} that does not essentially go through Theorem~\ref{thm:thomassen}.
    We also note that both for finite and locally finite countable graphs, Whitney's theorem predates Tutte's theorem (1930s vs 1960s for finite, 1980s vs 1990s for infinite).
\end{rem}

\begin{rem} \label{rem:Tutte_decomp_borel}
    We aim to use Lemma~\ref{lemma:tutte_decomp_preserved} on graphings componentwise. To be able to do that, one has to argue that the Tutte decomposition and the extension to virtual edges are Borel in the appropriate sense. This is not surprising, given that the decomposition and the extension are canonical. This follows from the proof of Theorem~\ref{thm:droms_servatius_servatius} by Droms–Servatius–Servatius; after rooting at an edge, they build the decomposition constructively, so the procedure can be run in a Borel fashion. Crucially, because of the uniqueness, one can root simultaneously at every edge, so no Borel choice of one root per connected component is needed. A similar argument shows that the extension is also Borel. 
\end{rem}

\subsection{Banana decomposition of infinitely-ended graphings}

In the infinitely-ended case the natural cut-block decomposition, and further the Tutte decomposition of the $2$-connected components will not be as useful. As an example, consider a 3-regular treeing $\mathcal{T}$. This is rigid by Theorem~\ref{thm:graph_isom}, so we do not want to decompose it at all, whereas the cut-block decomposition disassembles it into edges.

An example that is only slightly more complicated is an infinitely-ended treeing $\mathcal{T}'$ with degrees 2 and 3. In each component, the trifurcation vertices form a copy $T_3$, and these vertices are connected by finite paths. As the 3-regular treeing $\mathcal{T}$ is rigid, intuition suggests that $\mathcal{T}'$ is also rigid to some extent: a weak isomorphism should induce an isomorphism on the 3-regular treeing on the trifurcation vertices, but should be able to freely permute the edges along the paths connecting them. This intuition will be confirmed in Section~\ref{sec:general_Whitney_for_graphings}: it follows from  Theorem~\ref{thm:infty_ended_whitney_twists} applied to $\mathcal{T}'$.

Motivated by this example, in general infinitely-ended graphings we will consider finite subgraphs that we call \emph{bananas}, which will play the part of the finite paths. 

Given a graph $G$, and $F \subseteq E(G)$ a subset of edges, the \emph{vertex boundary of} $F$ is the set of vertices incident to both $F$ and $E(G)\setminus F$. We denote it by $\partial F$. Similarly, by $V(F)$ we denote the set of vertices incident to $F$. 

\begin{defn}
Given an infinite, connected, locally finite graph $G$, and a finite connected subset of edges $F \subseteq E(G)$ with $|\partial F|=2$, we say the graph $B=(V(F), F)$ is a \textbf{banana}. 
\end{defn}

Notice that an edge is a banana. Also, we will consider bananas in weakly $2$-connected graphings, where any finite edge set $F \subseteq E(G)$ has $|\partial F| \geq 2$. Finally, observe that removing the edges of a banana creates at most two infinite components in $G$.

We call a banana \textbf{maximal} if no strictly larger banana contains it. On a bi-infinite path, any finite subpath is a banana, and hence we do not have maximal bananas. We now proceed to show that if the components of $\mathcal{G}$ are infinitely-ended and weakly $2$-connected, maximal bananas not only exist, but we can decompose each component into maximal bananas with an underlying weakly $3$-connected structure.

\begin{lem}\label{lem:banana_trifurc}
A banana can meet at most two pairwise disjoint vertex-trifurcations.

\end{lem}
\begin{proof}
    Let $B$ denote a banana, and let $A$ be a vertex-trifurcation meeting $V(B)$. If $A$ did not meet $\partial B$, then, since $A$ is connected, it would be contained in the interior $V(B)\setminus\partial B$. But every infinite component of $G-A$ must then contain one of the two boundary vertices of $B$, so $G-A$ has at most two infinite components. This contradicts that $A$ is a trifurcation. Hence every vertex-trifurcation meeting $B$ contains one of the two boundary vertices, and therefore at most two disjoint ones can meet $B$.
\end{proof}

\begin{lem}\label{lem:banana_union}
Suppose $G$ is weakly $2$-connected, and $B_1$ and $B_2$ are bananas in $G$ such that $E(B_1)\cap E(B_2) \neq \emptyset$. Then $B_1 \cup B_2$ is also a banana.
\end{lem}
\begin{proof}
By the connectivity of bananas and $E(B_1) \cap E(B_2) \neq \emptyset$, at least one of the following holds:
\begin{enumerate}
    \item $\partial B_1 \subseteq V(B_2)$;
    \item $\partial B_2 \subseteq V(B_1)$;
    \item $|\partial B_2 \cap (V(B_1)\setminus\partial B_1)| = |\partial B_1 \cap (V(B_2)\setminus\partial B_2)| =1$. 
\end{enumerate}

In the first case $B=B_1 \cup B_2$ is a banana with $\partial B = \partial B_2$. In the second, $B$ is a banana with $\partial B = \partial B_1$. In the third case $B$ is a banana with $\partial B = \{x,y\}$, where $x$ and $y$ are the boundary vertices of $B_1$ and $B_2$ not contained in the interior of the other, respectively. 
\end{proof}

Consequently, distinct maximal bananas are edge-disjoint. 

\begin{lem}\label{lem:banana_largest}
Suppose $\cG$ is infinitely-ended and weakly $2$-connected. Then any banana in $\cG$ is contained in a unique maximal banana.
\end{lem}
\begin{proof}
    By Lemma~\ref{lem:banana_union} it is enough to show that there is no infinite increasing sequence of bananas.

    Recall the Borel partition $\cH$ of $\cG$ into trifurcation cells from Lemma~\ref{lem:cells}. If $B_1\subset B_2\subset\ldots\subset \cG$ were an infinite increasing sequence of bananas, some $B_n$ for $n$ large enough could not be contained in the union of two cells, contradicting Lemma~\ref{lem:banana_trifurc}.
\end{proof}

\begin{lem}\label{lem:path_in_a_banana}
    For every banana $B$ and edge $e\in E(B)$, there exists a path connecting the vertices of $\partial B$ that passes through the edge $e$.
\end{lem}
\begin{proof}

By Lemma~\ref{lem:double_ray_core}, the double-ray core of the ambient weakly $2$-connected graph is the whole graph, therefore every edge lies
on a bi-infinite path. Let $D$ be such a path containing $e$, and set
$P=D\cap E(B)$. 
\end{proof}

Consequently, in infinitely-ended, weakly $2$-connected graphings, maximal bananas canonically partition the edge set. Given such a graphing $\cG$, we define its \emph{banana decomposition graphing} $\texttt{Ban}(\cG)$ as the graphing whose vertices are the boundary vertices of maximal bananas, and two such vertices are connected exactly when they are the two boundary vertices of the same banana. Our earlier example $\cT'$ (the $3$-regular treeing with subdivided edges) is weakly $2$-connected, and its banana decomposition is simply a $3$-regular treeing. See Figure~\ref{fig:banana_decomp} for an illustration.

\begin{figure}[htb]
\begin{center}

\begin{tikzpicture}[
    line join=round,
    line cap=round,
    every path/.style={thick},
    big dot/.style={circle, fill, inner sep=1.2pt}
]

\def\R{1.8}   
\def\s{1.25}  
\def\d{0.24}  

\def\bridgefrac{0.28}
\pgfmathsetmacro{\halfbridgefrac}{\bridgefrac/2}
\pgfmathsetmacro{\leftbridgepos}{0.5-\halfbridgefrac}
\pgfmathsetmacro{\rightbridgepos}{0.5+\halfbridgefrac}

\newcommand{\bridge}[2]{%
    \coordinate (brL) at ($(#1)!\leftbridgepos!(#2)$);
    \coordinate (brR) at ($(#1)!\rightbridgepos!(#2)$);

    \coordinate (brLv) at ($(brL)+(#2)-(#1)$);
    \coordinate (brRv) at ($(brR)+(#2)-(#1)$);

    \coordinate (ba) at ($(brL)!\halfbridgefrac!90:(brLv)$);
    \coordinate (bb) at ($(brL)!\halfbridgefrac!-90:(brLv)$);
    \coordinate (bc) at ($(brR)!\halfbridgefrac!90:(brRv)$);
    \coordinate (bd) at ($(brR)!\halfbridgefrac!-90:(brRv)$);

    \draw (#1) -- (ba);
    \draw (#1) -- (bb);
    \draw (#2) -- (bc);
    \draw (#2) -- (bd);

    \draw (ba) -- (bc) -- (bd) -- (bb) -- cycle;

    \node[big dot] at (ba) {};
    \node[big dot] at (bb) {};
    \node[big dot] at (bc) {};
    \node[big dot] at (bd) {};
}

\newcommand{\bridgeShaded}[2]{%
    \coordinate (brL) at ($(#1)!\leftbridgepos!(#2)$);
    \coordinate (brR) at ($(#1)!\rightbridgepos!(#2)$);

    \coordinate (brLv) at ($(brL)+(#2)-(#1)$);
    \coordinate (brRv) at ($(brR)+(#2)-(#1)$);

    \coordinate (ba) at ($(brL)!\halfbridgefrac!90:(brLv)$);
    \coordinate (bb) at ($(brL)!\halfbridgefrac!-90:(brLv)$);
    \coordinate (bc) at ($(brR)!\halfbridgefrac!90:(brRv)$);
    \coordinate (bd) at ($(brR)!\halfbridgefrac!-90:(brRv)$);

    \fill[yellow!40] (#1) -- (ba) -- (bc) -- (#2) -- (bd) -- (bb) -- cycle;

    \draw (#1) -- (ba);
    \draw (#1) -- (bb);
    \draw (#2) -- (bc);
    \draw (#2) -- (bd);

    \draw (ba) -- (bc) -- (bd) -- (bb) -- cycle;

    \node[big dot] at (ba) {};
    \node[big dot] at (bb) {};
    \node[big dot] at (bc) {};
    \node[big dot] at (bd) {};
}

\begin{scope}

    \coordinate (A1) at ({\R*cos(90)},{\R*sin(90)});
    \coordinate (B1) at ({\R*cos(210)},{\R*sin(210)});
    \coordinate (C1) at ({\R*cos(330)},{\R*sin(330)});

    \bridgeShaded{A1}{B1}
    \bridge{B1}{C1}
    \bridge{C1}{A1}

    \node[big dot] at (A1) {};
    \node[big dot] at (B1) {};
    \node[big dot] at (C1) {};

    \foreach \P/\ang/\L/\Q in {
        A1/90/LA1/RA1,
        B1/210/LB1/RB1,
        C1/330/LC1/RC1
    }{
        \coordinate (\L) at ($(\P)+({\s*cos(\ang-30)},{\s*sin(\ang-30)})$);
        \coordinate (\Q) at ($(\P)+({\s*cos(\ang+30)},{\s*sin(\ang+30)})$);

        \bridge{\P}{\L}
        \bridge{\L}{\Q}
        \bridge{\Q}{\P}

        \node[big dot] at (\L) {};
        \node[big dot] at (\Q) {};

        \foreach \k in {1,2,3} {
            \fill ($(\L)+({\k*\d*cos(\ang-60)},{\k*\d*sin(\ang-60)})$)
                circle (0.7pt);
            \fill ($(\Q)+({\k*\d*cos(\ang+60)},{\k*\d*sin(\ang+60)})$)
                circle (0.7pt);
        }
    }
\end{scope}


\begin{scope}[shift={(8,0)}]
    \coordinate (A) at ({\R*cos(90)},{\R*sin(90)});
    \coordinate (B) at ({\R*cos(210)},{\R*sin(210)});
    \coordinate (C) at ({\R*cos(330)},{\R*sin(330)});

    \draw (A) -- (B) -- (C) -- cycle;

    \node[big dot] at (A) {};
    \node[big dot] at (B) {};
    \node[big dot] at (C) {};

    \foreach \P/\ang in {A/90, B/210, C/330} {
        \coordinate (L\P) at ($(\P)+({\s*cos(\ang-30)},{\s*sin(\ang-30)})$);
        \coordinate (R\P) at ($(\P)+({\s*cos(\ang+30)},{\s*sin(\ang+30)})$);

        \draw (\P) -- (L\P) -- (R\P) -- cycle;

        \node[big dot] at (L\P) {};
        \node[big dot] at (R\P) {};

        \foreach \k in {1,2,3} {
            \fill ($(L\P)+({\k*\d*cos(\ang-60)},{\k*\d*sin(\ang-60)})$)
                circle (0.7pt);
            \fill ($(R\P)+({\k*\d*cos(\ang+60)},{\k*\d*sin(\ang+60)})$)
                circle (0.7pt);
        }
    }
\end{scope}


\draw[dashed,->] (3, 1) -- (5, 1);
\end{tikzpicture}
\end{center}
\caption{An example of an infinitely-ended graph and its banana decomposition.}\label{fig:banana_decomp}
\end{figure}
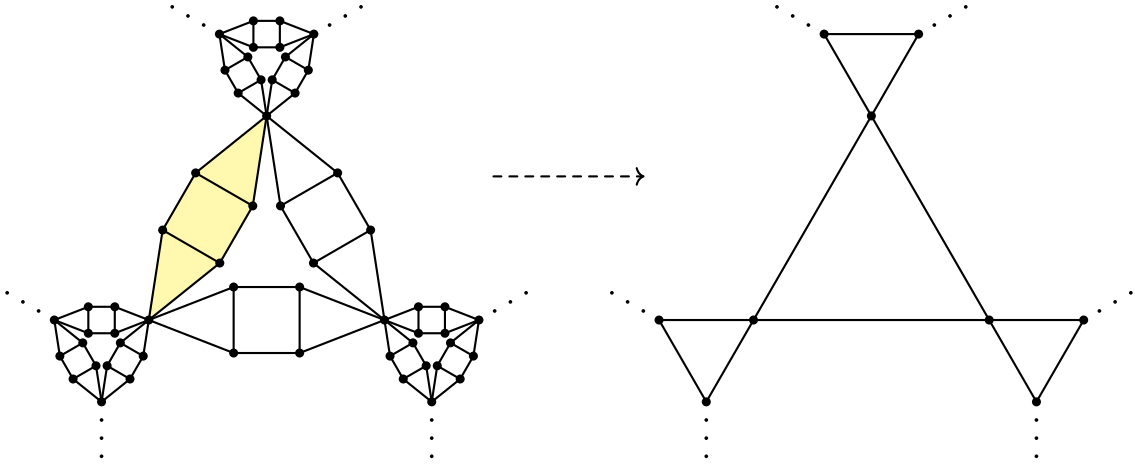

\begin{lem} \label{lem:banana_$3$-connected}
        Suppose $\cG$ is infinitely-ended and weakly $2$-connected. Then $\texttt{Ban}(\cG)$ is weakly $3$-connected.
\end{lem}

\begin{proof}
    If one could disconnect a nontrivial finite graph from $\texttt{Ban}(\cG)$ by cutting at two vertices, then the union of the corresponding maximal bananas would be a larger banana in $\cG$, contradicting maximality.
\end{proof}

\subsection{Minors of graphings and their cycle matroids} \label{subsec:minors}

One can produce minors of finite graphs and matroids by deleting some edges and contracting some others. This operation is a natural and important tool of finite combinatorics. Here we will also use such arguments, but in order to do so, we need to set this up for graphings and their cycle matroids as well. Our main observation is that the theory works nicely as long as we contract only finite components, and in the crucial case of infinitely-ended, $2$-connected graphings, the $\varphi$-image of the edges contracted in $\cG_1$ can again be contracted in $\cG_2$.

To be precise, given a graphing $\cG=(X, E, \mu)$, and edge sets $D,C \subseteq E$ to be deleted and contracted, respectively, we define the corresponding minor of the cycle matroid on the set $A=E(\cG) \setminus (D \cup C) $ by its rank function:

\[\hat\rho(F)=\rho_{\cG}(F\cup C)-\rho_{\cG}(C), \textrm{ for any $F \subseteq A$ Borel.}\]

For finite graphs this corresponds to the classical notion of deleting and contracting edges, so $\hat \rho$ is the rank function of the cycle matroid of the corresponding minor of the graph. In the measurable setting, however, the graph we get by contracting infinite components of edges of a graphing might not be a graphing anymore. 

\begin{rem}\label{rem:bad_contraction}
For an example, consider a free p.m.p.\ action of $\Gamma=\mathbb{Z} \times (\mathbb{Z}/2\mathbb{Z}) = \langle a,b\mid [a,b]=1, b^2=1 \rangle$ such that $b$ defines a perfect matching between two disjoint $\langle a \rangle$-invariant measurable subsets $A,B \subseteq X$ with $\mu(A) = \mu(B)=1/2$. In the corresponding Schreier graphing $\cG$ let $C$ denote the set of $a$-edges with endvertices in $A$. $C$ is a measurable subset of $E(\cG)$, in each component it is one of the two $\mathbb{Z}$-lines in $\Cay(\Gamma, \{a,b\})$. Yet contracting it would not produce a graphing. If we assume further that $a \curvearrowright A$ is ergodic, then the contraction should in any reasonable sense correspond to contracting the whole of $A$ to a single vertex, with mass $1/2$, connected to all vertices in $B$.
\end{rem}

Note that deleting edges is not a problem, $\cG_D=(X,E\setminus D,\mu)$ is always a graphing. If we also assume that in $\cG_D$ the $C$-components are finite, then contracting the edges produces a locally finite multi-graphing $\cG_{D,C} = (\hat X, \hat E, \hat\mu)$. Here $\hat X$ is the Borel space of $C$-connected components of $\cG_D$, $\hat E$ is a copy of $E\setminus (D\cup C)$ with the appropriate endvertex maps, and we define $\hat\mu$ by setting, for any $\hat Y \subseteq \hat X$ Borel, $\hat{\mu}(\hat Y) = \mu(Y)$ where $Y \subseteq X$ contains exactly one vertex from each $C$-class. (Such a $Y$ exists because $C$ has finite components, and any choice of $Y$ has the same $\mu$-measure by the p.m.p.\ property of $\cG_D$, making $\hat\mu$ well defined.) Note that the edge-contraction might create parallel edges or loops, but in practice we will choose $C$ such that $\cG_{D,C}$ is a simple graphing. We do not need to worry about parallel edges being created in a weakly isomorphic image either. That is, if $\varphi$ is a weak isomorphism between two graphings $\cG_1$ and $\cG_2$, and we take $\varphi$-isomorphic minors $\cG_1' = (\cG_1)_{D,C}$ and $\cG_2' = (\cG_2)_{\varphi(D),\varphi(C)}$ such that $\cG_1'$ is a simple graphing, then so is $\cG_2'$. Indeed, loops and parallel edges are easy to read from the rank function, so if $\cG_2'$ had them, then so would $\cG_1'$.

In our proofs, we will always contract some $C$ with finite components. Observe that we can also assume that $C$ is acyclic: we can measurably choose a spanning tree in each $C$-component, and move the rest of the edges to $D$. This does not change the resulting graphing or its cycle matroid. 

In preparation, we first show that hyperfinite edge sets in infinitely-ended, weakly $2$-connected graphings cannot glue together to form entire connected components under rank-preserving edge bijections.

\begin{lem} \label{lem:no_full_components}
Let $\cG_1$ be an infinitely-ended, weakly $2$-connected graphing, and let $\fg:E(\cG_1)\to E(\cG_2)$ be a rank-preserving Borel bijection. If $A\subseteq E(\cG_1)$ is hyperfinite, then $\fg(A)$ cannot contain the full edge sets of a positive-measure union of connected components of $\cG_2$.
\end{lem}

\begin{proof}
Assume towards contradiction that there is an invariant Borel set $Y\subseteq V(\cG_2)$ such that $E(\cG_2|Y)$ has positive edge measure and
\(        E(\cG_2|Y)\subseteq \varphi(A).
   \)
As $A$ is hyperfinite and $\varphi$ preserves hyperfiniteness, $\cG_2|Y$ is hyperfinite. By Theorem~\ref{thm:union_of_inf_ended_leafless}, write $E(\cG_1)=\bigcup_{i\in\N} F_i$,
where every $F_i$ is a leafless infinitely-ended forest. Let $    B:=\varphi^{-1}\bigl(E(\cG_2|Y)\bigr)$.

Since $B$ has positive edge measure, $F_i\cap B$ has positive edge measure for some $i$. By Lemma~\ref{lem:forest_superfluous} and Corollary~\ref{cor:phi_superfluous}, $\varphi(F_i)$ is again a leafless infinitely-ended forest. Since $Y$ is $\cG_2$-invariant, $\varphi(F_i)\cap E(\cG_2|Y)$ 
is a positive-measure union of full components of $\varphi(F_i)$. Hence it is a treeing whose non-trivial components are infinitely-ended, and therefore it is not hyperfinite. On the other hand, it is a subgraph of the hyperfinite graphing $\cG_2|Y$, a contradiction.
\end{proof}

We now improve Lemma~\ref{lem:no_full_components}, and show that finite components themselves are preserved.

\begin{lem} \label{lem:image_of_finite_is_finite}
Let $\cG_1$ be an infinitely-ended, weakly $2$-connected graphing, and let $C\subseteq E(\cG_1)$ be acyclic and have finite components. If $\fg:E(\cG_1)\to E(\cG_2)$ is a rank-preserving Borel bijection, then $\fg(C)$ also has finite components.
\end{lem}

\begin{proof}
We first prove the statement under the additional assumption that the loopless quotient $\cG_1/C$ is weakly $2$-connected. Let
\[
    D:=\{e\in E(\cG_1)\setminus C:
    \text{the endpoints of $e$ belong to the same $C$-component}\}.
\]
Thus the edges in $D$ become loops after contracting $C$. The loopless quotient $\cG'_1:=\cG_1/C$ has edge set naturally identified with $E(\cG_1)\setminus(C\cup D)$. Finite contraction preserves the ends, so $\cG'_1$ is infinitely-ended as well. By Theorem~\ref{thm:union_of_inf_ended_leafless}, choose leafless infinitely-ended forests $F_1,F_2,\ldots\subseteq E(\cG'_1)$
whose union is $E(\cG'_1)$. By slight abuse of notation, we regard the $F_i$ also as edge sets in $\cG_1$. Then $C\cup F_i$ is a forest.

Moreover, $F_i$ is superfluous in $C\cup F_i$. Indeed, if $A\subseteq F_i$ is disposable in $F_i$, then by the contraction formula for the rank,
\[
    \rho_1(C\cup F_i)-\rho_1(C) = \rho_{\cG'_1}(F_i) = \rho_{\cG'_1}(F_i\setminus A) = \rho_1\bigl(C\cup(F_i\setminus A)\bigr)-\rho_1(C).
\]
Since $F_i$ is superfluous in $\cG'_1$, the claim follows.

Put $H:=\varphi(C)$ and $K_i:=\varphi(F_i)$, and let $H_\infty$ and $H_0$ denote the unions of the infinite and finite $H$-components, respectively. Finally, set
\[
    J_i:=H_0\cup K_i,\qquad
    U_i:=H\cup K_i=H_\infty\cup J_i.
\]
The graph $U_i$ is a forest, since $C\cup F_i$ is a forest and $\varphi$ preserves cycles. Furthermore, $K_i$ is superfluous in $U_i$.

Assume towards contradiction that $H_\infty$ has positive edge measure. We analyze the interaction of $J_i$ with $H_\infty$.

{\bf First, after discarding an invariant null set, no $J_i$-component meets $H_\infty$ in more than one vertex.} Indeed, two such vertices cannot belong to the same $H_\infty$-component, since the paths between them in $J_i$ and in $H_\infty$ would form a cycle in $U_i$. Suppose that the set of contact vertices which lie in a $J_i$-component having at least two contacts has positive measure. By partitioning according to the distance between two contacts and then applying a standard Borel coloring, we can find a positive-measure family $\mathscr P$ of pairwise vertex-disjoint paths of some fixed length $k$ in $J_i$, each joining two distinct $H_\infty$-components. We may moreover assume that the $C$-saturations of their $\varphi$-preimages are pairwise disjoint, as the corresponding conflict graph is locally finite, since the paths and all $C$-components involved are finite.

Let $P$ be the union of the paths in $\mathscr P$. After orienting the paths, write $P=P_1\cup\cdots\cup P_k,$
where $P_j$ consists of the $j$-th edge of every path. Each $P_j$ is disposable in $H_\infty\cup P$, as deleting one edge from every path still leaves both resulting pieces attached to an infinite $H_\infty$-component. Hence $P$ is superfluous in $H_\infty\cup P$.

Consequently $\varphi^{-1}(P)$ is superfluous in $\varphi^{-1}(H_\infty)\cup \varphi^{-1}(P)$. On the other hand, $\varphi^{-1}(H_\infty)\cup \varphi^{-1}(P)$ is an acyclic graph with finite components, as it is contained in $C\cup F_i$, and by our thinning each component is contained in one finite preimage path together with finitely many finite $C$-components. This is impossible, since a finite-component forest has no positive-measure superfluous edge set. This proves the claim.

{\bf Next, no infinite $J_i$-component can touch $H_\infty$ on a positive-measure set.} By the previous claim, it has a unique contact vertex. These contact vertices would therefore form a Borel partial transversal of a positive-measure family of infinite $J_i$-components, which is impossible by the mass transport principle.

{\bf Finally, no positive-measure set of $K_i$-edges lying in finite $J_i$-components can touch $H_\infty$.} Indeed, apart from the null set discarded above, such a finite component has a unique contact with $H_\infty$. Every $K_i$-edge in it is then locally non-disposable in $U_i$, so deleting the edge cuts off a finite component on the side away from the unique contact. This contradicts the fact that $K_i$ is superfluous in $U_i$.

Consequently, for every $i$, almost no edge of $K_i$ is incident with $H_\infty$. Now let $e\in E(\cG_2)\setminus H$ be incident with a vertex of $H_\infty$, and put $d:=\varphi^{-1}(e)$. If $d\notin D$, then $d$ belongs to some $F_i$, and hence $e\in K_i$, contradicting the previous conclusion. Thus $d\in D$. Let $P_d\subseteq C$ be the unique $C$-path joining the endpoints of $d$. Then $P_d\cup\{d\}$ is a finite cycle, and its $\varphi$-image is again a cycle. As $\varphi(P_d)\subseteq H$, the endpoints of $e$ lie in the same $H$-component. In particular, if one endpoint belongs to $H_\infty$, then both belong to the same $H_\infty$-component. It follows that the full edge sets of all $\cG_2$-components meeting $H_\infty$ are contained, modulo null sets, in $H_\infty\cup\varphi(D) \subseteq \varphi(C\cup D)$. But $C\cup D$ has finite components, and is therefore hyperfinite. This contradicts Lemma~\ref{lem:no_full_components}. Thus $\varphi(C)$ has only finite components whenever $\cG_1/C$ is weakly $2$-connected.

We now remove this additional assumption. Contract $C$ and let $Q$ denote the underlying loopless simple quotient. Since the $C$-components are finite, $Q$ is still infinitely-ended. Let $R\subseteq Q$ be its double-ray core from Lemma~\ref{lem:double_ray_core}. Thus $R$ is Borel, weakly $2$-connected and infinitely-ended, while every component of $Q\setminus E(R)$ is finite and is attached to $R$ at a single vertex.

Let $L\subseteq E(\cG_1)\setminus C$ be the set of edges whose image in the quotient belongs to $E(Q)\setminus E(R)$, together with the edges which become loops after contracting $C$. Then $C\cup L$ has finite components. Measurably choose in each of these components a spanning tree containing its $C$-forest, and let $T$ be the union of these spanning trees. Thus $C\subseteq T$, and $T$ is acyclic and has finite components. After contracting $T$ and suppressing the resulting loops, the quotient has $R$ as its underlying simple graph, which is weakly $2$-connected and infinitely-ended. Applying the first part of the proof to $T$, we obtain that $\varphi(T)$ has finite components. Since $\varphi(C)\subseteq\varphi(T)$, the components of $\varphi(C)$ are finite as well.\end{proof}

\subsection{Preserving components, ends, and bananas} \label{subsec:preserving_comps_ends_bananas}

We aim to establish that under the right connectivity assumptions, rank-preserving bijections preserve the connected components of graphings. This will allow us to treat the 1-, 2-, and infinitely-ended cases separately when proving Theorem~\ref{thm:Whitney_2-isom_intro}.

\begin{thm}\label{thm:connected_preserving}
    Let $\cG_1$ and $\cG_2$ be graphings, both having the property that almost every $0$-, $1$- and $2$-ended component is strongly $2$-connected, while almost every infinitely-ended component is weakly $2$-connected. 
    Assume $\fg: E(\cG_1) \to E(\cG_2)$ is a matroid rank-preserving Borel bijection. Then both $\varphi$ and $\varphi^{-1}$ preserve the connectedness relation on the edges modulo null sets.
\end{thm}

While this seems like a warm-up result towards our eventual goal, we can only establish it using quite a bit of technicality. In the proof, we will choose an arbitrary edge from each maximal banana, and show that their $\varphi$-images are in the same component. From here, the following technical lemma will be useful in concluding that the connectivity relation on the edges, which we denote $\cE_{E(\cG)}$, is preserved.

\begin{lem}\label{lem:union_of_pairs}
    Let $\cE'\subset \cE_{E(\cG)}$ be a Borel equivalence relation with finite classes. Then there exist Borel transversal sets to $\cE'$, denoted $A_1,A_2,\dots$, such that if 
    $e_1$ and $e_2$ are in the same $\cE_{E(\cG)}$-class but not the same $\cE'$-class, then there exists an $n$ such that $e_1,e_2\in A_n$.
\end{lem}
\begin{proof}
    Instead of showing the existence of transversal sets with the given property, we show the existence of partial transversals. As the classes of $\cE'$ are finite, we can extend these sets into transversals.
    
    First, let $K_{\cE'}$ be the graph on the classes of $\cE'$ such that two classes are adjacent if they are adjacent in $\cG$. This graph is clearly locally finite, as $\cG$ is locally finite and the classes of $\cE'$ are all finite. Let us denote by $d_{\cE'}$ the graph distance in $K_{\cE'}$. This function extends to a quasi-distance on $E(\cG)$. 
    
    Now define a sequence of graphs $K_n$ in the following way. The vertex set consists of those pairs in $\cE_{E(\cG)}\setminus \cE'$, where the $d_{\cE'}$-distance of the two edges in the pair is at most $n$. Two vertices $(e_1,e_2)$ and $(f_1,f_2)$ are connected if there exists $i,j\in\{1,2\}$ such that $(e_i,f_j)\in\cE'$.
    As $K_{\cE'}$ is locally finite and $\cE'$ has finite classes, $K_n$ is locally finite for every $n \in \N$. This means that $K_n$ has a Borel proper coloring $c_n:V(K_n)\to \N$. Each color class consists of distinct pairs, as if two pairs have a common edge, they have distance $0$. We define the sets $A_n^k=\{e\in E(\cG):\ \exists e'\in E(\cG) \text{ such that } c_n(\{e,e'\})=k \}$. As $c_n$ is a proper coloring, each $A_n^k$ is a partial transversal. Every pair $(e,e')\in\cE_{E(\cG)}\setminus\cE'$ is a vertex in $K_n$ for large enough $n$, therefore there are $n,k$ such that $e,e'\in A_{n}^k$.
\end{proof}

\begin{proof}[Proof of Theorem~\ref{thm:connected_preserving}] A graph is $2$-connected if and only if any pair of edges admits a cycle containing them. After discarding an invariant null set, $\varphi$ preserving cycles implies that for edges $e$ and $e'$ from the same  $0$-, $1$-, or $2$-ended $\cG_1$-components their $\varphi$-images are in the same $\cG_2$-component.

We can therefore reduce to the case in which every component of $\cG_1$ is infinitely-ended. We decompose the edge set into maximal bananas. We measurably pick a distinguished edge $e_B$ from each banana $B$, and a path $P_B \subseteq E(B)$ containing $e_B$ connecting the two vertices in $\partial B$ inside $B$. This can be done, as the bananas are finite (therefore smooth), and by Lemma~\ref{lem:path_in_a_banana} such paths exist. We obtain a minor of $\cG_1$ by deleting all edges outside $P_B$ in all bananas, and contracting all edges of $P_B$ except $e_B$. Denote the set of edges that are deleted and contracted by $D$ and $C$, respectively, and notice that $C$ is acyclic and has finite components. 

The resulting graphing $\cG_1'=(\cG_1)_{D,C}$ is isomorphic to $\texttt{Ban}(\cG_1)$, and hence it is simple, and also weakly $3$-connected by Lemma~\ref{lem:banana_$3$-connected}. By Lemma~\ref{lem:image_of_finite_is_finite} $\varphi(C)$ also has finite components, so we can form the simple graphing $\cG_2'=(\cG_2)_{\varphi(D),\varphi(C)}$ by deleting $\varphi(D)$ and contracting $\varphi(C)$ in $\cG_2$. Moreover, the map $\varphi$ induces a rank-preserving Borel bijection $\varphi':E(\cG_1') \to E(\cG_2')$. (Recall that taking minors can be formulated using the rank function, see Subsection~\ref{subsec:minors}.) Therefore by Theorem~\ref{thm:graph_isom} $\varphi'$ is induced by an isomorphism of graphings. Consequently, we have that for any chosen edges $e_{B_1}$ and $e_{B_2}$ in bananas $B_1$ and $B_2$ from the same component of $\cG_1$, their images $\varphi(e_{B_1})$ and $\varphi(e_{B_2})$ are edges in the same component of $\cG_2$. 

By Lemma~\ref{lem:union_of_pairs}, we can find transversal sets for $\cE_B$ (the equivalence relation of being in the same banana), $A_1,A_2,\dots$. We use the previous argument for every $n\in\N$, choosing the distinguished edges $e_B$ according to $A_n$. This shows that if two edges $e_1,e_2$ are in the same component, but not the same banana, then $\varphi(e_1)$ and $\varphi(e_2)$ are also in the same component. This in turn implies that any two edges that are in the same $\cG_1$-component are mapped to edges in the same $\cG_2$-component. 
\end{proof}

\begin{cor} \label{cor:preserve_no_ends}
    Let $\cG_1,\cG_2$ and $\fg$ be as in Theorem~\ref{thm:connected_preserving}. Then $\fg$ preserves the number of ends of components.
\end{cor}
\begin{proof}
    As $\fg$ and $\fg^{-1}$ both preserve connectedness, it is enough to show that the images of $0$-, $1$-, and $2$-ended components have the same number of ends. These components are $2$-connected, so by Lemma~\ref{lemma:tutte_decomp_preserved} a cycle-preserving bijection also preserves the Tutte decomposition, and the number of ends of a graph can be determined from the Tutte decomposition.
\end{proof}

We also establish the lemmas we will need to treat the infinitely-ended case of Whitney's 2-isomorphism theorem. We do this here because they follow from the same argument and do not rely on the somewhat cumbersome setup of measurable Whitney operations. We start by showing that maximal bananas are preserved under $\varphi$. 

\begin{cor}\label{cor:bananas_preserved}
    Let $\cG_1$ and $\cG_2$ be infinitely-ended, weakly $2$-connected graphings that admit a rank-preserving Borel bijection. Then, for almost every component of $\cG_1$, any pair of edges from different maximal bananas of $\cG_1$ are mapped to edges in different maximal bananas of $\cG_2$.
\end{cor}
\begin{proof}
    We perform the same minor operation as in the proof of Theorem~\ref{thm:connected_preserving}. Note that, after discarding an invariant null set, $\cG_1' \cong \texttt{Ban}(\cG_1)$ does not have parallel edges, as the union of the corresponding maximal bananas would again be a banana, contradicting maximality. (Also, there are no loops: $\texttt{Ban}(\cG_1)$ is a simple graphing, because $|\partial B|=2$ for each banana. Note that, as we assumed that $\cG_1$ is weakly $2$-connected, there are in fact no finite induced subgraphs $F$ in $\cG_1$ with $|\partial F|=1$.)  

    Now assume that $e_{B_1}$ and $e_{B_2}$ are mapped into the same maximal banana $\hat{B}$ in $\cG_2$. Let $X$ denote the image of $\partial\hat{B}$ in $\cG_2'$ after performing the minor operation in $\cG_2$. Clearly $|X|\leq 2$, so removing it cannot create finite components by the weak $3$-connectivity of $\cG_2' \cong \cG_1'$. Consequently $\varphi(e_{B_1})$ and $\varphi(e_{B_2})$ are loops or parallel edges with endpoints in $X$. On the other hand, $\cG_2' \cong \cG_1'$ implies that $\cG_2'$ has no loops or parallel edges. Consequently, modulo null sets, $e_{B_1}$ and $e_{B_2}$ for distinct $B_1$ and $B_2$ can only be mapped to distinct maximal bananas. As before, we conclude the proof by repeating the argument for every choice of $e_B$ provided by the transversals $A_1,A_2, \dots$ from Lemma~\ref{lem:union_of_pairs}. 
\end{proof}

By using Corollary~\ref{cor:bananas_preserved} in both directions, we see that $\varphi$ maps (edge sets of) maximal bananas to (edge sets of) maximal bananas. For a banana $B$ in $\cG_1$, we denote the target banana in $\cG_2$ by $\varphi(B)$. We also prove that $\varphi$ restricted to a maximal banana $B \subseteq \cG_1$ preserves not only cycles, but also paths between the boundary points $\partial B$.

\begin{lem} \label{lem:banana_path_preserving}
    Let $\cG_1$ and $\cG_2$ be infinitely-ended, weakly $2$-connected graphings and $\varphi$ a rank-preserving Borel bijection. Then, for almost every component of $\cG_1$, any path $P$ connecting $\partial B$ of a maximal banana $B$ inside $B\subset\cG_1$ is mapped to a path $Q=\varphi(P)$ connecting $\partial(\varphi(B))$ of the maximal banana $\varphi(B)$ inside $\varphi(B)\subset\cG_2$.
\end{lem} 
\begin{proof}
    By Theorem~\ref{thm:union_of_inf_ended_leafless} applied to $\texttt{Ban}(\cG_1)$ we can find a leafless infinitely-ended forest $F$ of $\cG_1$ such that $P=E(B)\cap F$. Then $\varphi(F)$ is again a leafless forest by Corollary~\ref{cor:phi_superfluous}. Notice that the intersection of a banana and a leafless forest, if nonempty, is a path connecting the boundary points. Hence $\emptyset \neq \varphi(P) = E(\varphi(B)) \cap \varphi(F)$ is a path $Q$ connecting $\partial(\varphi(B))$.   
\end{proof}

Given a maximal banana $B$ we write $\overline{B}$ for the finite graph obtained by adding an edge between the two vertices in $\partial B$. The edge-bijection $\varphi$ naturally extends to a bijection $\varphi: E(\overline{B}) \to E(\overline{\varphi(B)})$ by matching the newly added edges. Note that $\overline{B}$ is $2$-connected. Indeed, deleting any vertex $v \in V(B)\setminus \partial B$ produces components that contain at least one of the vertices in $\partial B$, so $\overline{B} \setminus \{v\}$ is connected. And deleting a vertex from $\partial B$ cannot disconnect $B$: recall that we assume the graphings to be weakly $2$-connected, so one vertex cannot disconnect a finite subgraph. Consequently, Lemma~\ref{lem:banana_path_preserving} has the following corollary.

\begin{cor} \label{cor:extended_banana_cycle_preserving}
    Let $\cG_1$, $\cG_2$ and $\varphi$ be as in Lemma~\ref{lem:banana_path_preserving}. Then, for almost every component of $\cG_1$, $\overline{B}$ is $2$-connected for each maximal banana $B$, and $\varphi: E(\overline{B}) \to E(\overline{\varphi(B)})$ is cycle-preserving.
\end{cor}


\section{Classification by measurable Whitney operations}
\label{sec:general_Whitney_for_graphings}
In this section, we first introduce Whitney operations for graphings. We then establish Theorem~\ref{thm:Whitney_2-isom_intro} by 

\begin{enumerate}
    \item using the tools developed in this paper to prove it in the special case when all components of $\cG_1$ are infinitely-ended and weakly $2$-connected; 

    \item using the tools for locally finite graphs from earlier works recalled in Subsection~\ref{subsec:locally_finite_tools} in the special case when all components are $0$-, $1$-, and $2$-ended and strongly $2$-connected;

    \item performing splitting operations on $\cG_1$ and $\cG_2$ until all components fall in one of the above categories;
    
    \item \label{item:arguing_separately} arguing that we can treat the two types of components of $\cG_1$ separately. 
    
\end{enumerate}
\begin{rem}
Note that part (\ref{item:arguing_separately}) is non-trivial; it relies on Theorem~\ref{thm:connected_preserving}, whose proof makes use of all the tools we developed for the infinitely-ended case, as well as Theorem~\ref{thm:graph_isom}.
\end{rem}

\subsection{Whitney operations for graphings} \label{subsec:whitney_operations}

The last thing we have to set up before we can turn to the proof of Theorem~\ref{thm:Whitney_2-isom_intro} is the collection of basic operations that preserve the rank function of a graphing. By Lemma~\ref{lem:eq_of_rankpres}, preserving the rank is equivalent to preserving cycles and hyperfiniteness of all measurable edge subsets. We call a pair of distinct vertices $\{x,y\}$ in a connected graph $G$ a {\bf cut pair} if $G-\{x,y\}$ is disconnected.

For finite graphs, splitting at cut vertices preserves cycles. In an infinite connected component $G$ of a graphing, however, we will need to split at infinitely many cut vertices simultaneously, and this might fail to preserve non-hyperfiniteness by breaking up an infinitely-ended subgraph into hyperfinite pieces, see Example~\ref{ex:traingle_tree}. This problem does not occur if we only cut finite graphs from $G$. Another example where hyperfiniteness is preserved is when $G$ is 2-ended, and we split at infinitely many cut vertices along a bi-infinite path. The same considerations apply when one twists at cutting vertex pairs. These examples are illustrated in Figures~\ref{Fig:finite_split}--\ref{fig:2ended_twist}. Motivated by these observations, we spend the rest of the subsection defining what we call \emph{Whitney operations} in detail. Each kind of operation from the finite world (split, join, twist) will have two counterparts in the measurable setting, one where only a finite component of the cut is altered, and one where the operation is carried out simultaneously along a bi-infinite path.

{\bf (Finite vertex join.)} Let $\cG=(X,E,\mu)$ be any graphing, and $\cH=(Y,F,\nu)$ a graphing with finite components. Let $Y_0 \subseteq Y$ contain one vertex from every component. (Such $Y_0$ exists for a graphing $\cH$ if and only if $\cH$ has finite components.) Furthermore, let $f:Y_0 \to X_0 \subseteq X$ be a measure preserving Borel bijection. We define the join $\mathcal{K}=(Z,L,\pi)$ of the two graphings by gluing $Y_0$ to $X_0$ along $f$. That is, let $Y_1=Y \setminus Y_0$, and set $Z=X\sqcup Y_1$ to be the vertex set with measure $\pi$ defined by $\pi|_{X}=\mu$ and $\pi|_{Y_1} = \nu|_{Y_1}$. The set of edges is $L=E \cup F(Y_1,Y_1) \cup \big\{uv ~\big|~ wv \in F(Y_0,Y_1) \textrm{ for } w=f^{-1}(u)\big\}$, where $F(A,B)$ denotes the set of edges with one endpoint in $A$ and the other in $B$. (Edges are unoriented, so when an edge $wv$ is added, this implicitly means the reverse pair $vw$ is also present.) 

{\bf (Finite vertex split.)} This is the reverse operation of a finite join. That is, start with a graphing $\mathcal{K}=(Z, L,\pi)$ and a measurable subset of vertices $X_0 \subseteq Z$ that are cut vertices in their components, with the additional property that at least one of the components of the cut is finite, and contains no vertex from $X_0$. Choose one such finite component to be split off at each vertex of $X_0$. The choice can be made measurably, as there are finitely many possibilities at each vertex by local finiteness. Form the set $Y_1 \subseteq Z$ as the collection of the vertices of the chosen components (excluding $X_0$). Let $(Y_0,\nu_0)$ be a disjoint copy of $(X_0, \pi|_{X_0})$, with $f: Y_0 \to X_0$ a measure preserving Borel bijection. (Say, \ $Y_0=X_0\times\{0\}$, $f(x,0)=x$, $\nu_0=\pi\circ f$.) We define the two parts of the split $\cG=(X,E,\mu)$ and $\cH=(Y,F,\nu)$ by setting $X=Z \setminus Y_1$, $E=L(X,X)$, and $\mu=\pi|_{X}$, as well as $Y=Y_0 \cup Y_1$, $F=L(Y_1,Y_1) \cup  \big\{xy ~\big|~ wy \in L(X_0,Y_1) \textrm{ for } w=f(x)\big\}$, and $\nu|_{Y_0} = \nu_0$, $\nu|_{Y_1}= \pi|_{Y_1}$. 

\begin{rem}
One can replace some edges of a graphing with finite graphs, and obtain another graphing. More precisely, assume we are given a graphing $\cG$ and a Borel subset of edges $F$ of $\cG$ with an arbitrary Borel orientation. Denote by $\vec{F}$ these oriented edges. Assume further that we have a measurable map $\psi$ from $\vec{F}$ to the space of doubly rooted \emph{finite} graphs. One can then construct a graphing by deleting all edges $f\in F$ from $E(\cG)$, and replacing each $f$ with $\psi\big(\vec{f}\big)$, where $\vec{f}$ denotes the orientation of $f$. The role of the orientation is to determine how the finite graph is inserted: we glue the first root of $\psi\big(\vec{f}\big)$ to the starting vertex of $\vec{f}$, and the second root of $\psi\big(\vec{f}\big)$ to the endvertex of $\vec{f}$. (Note that we added vertices, so the total vertex measure of $\cG$ might change.) When referring to this procedure, we say that we \emph{replace $\vec{F}$ in $\cG$, gluing according to $\psi$}. This terminology will be convenient when presenting the next two Whitney operations.  
\end{rem}

\begin{figure}[htb]
\centering
\makebox[\textwidth][c]{%
\begin{tikzpicture}[
    line join=round,
    line cap=round,
    every path/.style={thick},
    big dot/.style={circle, inner sep=0pt, minimum size=1mm, fill=black},
    red dot/.style={circle, inner sep=0pt, minimum size=1.2mm, fill=red}
]

\def\xstep{0.85}     
\def\dotsep{0.22}    
\def\rowtop{1.4}
\def\rowmid{0}
\def\rowbot{-1.4}
\def\gsize{0.55}     

\pgfmathsetmacro{\xmTwo}{-2*\xstep}
\pgfmathsetmacro{\xmOne}{-1*\xstep}
\pgfmathsetmacro{\xZero}{0}
\pgfmathsetmacro{\xpOne}{1*\xstep}
\pgfmathsetmacro{\xpTwo}{2*\xstep}

\newcommand{\hdotsleft}[2]{%
    \foreach \k in {1,2,3} {
        \fill ({#1-\k*\dotsep},#2) circle (0.7pt);
    }
}

\newcommand{\hdotsright}[2]{%
    \foreach \k in {1,2,3} {
        \fill ({#1+\k*\dotsep},#2) circle (0.7pt);
    }
}

\newcommand{\vdotsabove}[2]{%
    \foreach \k in {1,2,3} {
        \fill (#1,{#2+\k*\dotsep}) circle (0.7pt);
    }
}

\newcommand{\vdotsbelow}[2]{%
    \foreach \k in {1,2,3} {
        \fill (#1,{#2-\k*\dotsep}) circle (0.7pt);
    }
}

\newcommand{\Zline}[1]{%
    \foreach \i in {-2,-1,0,1} {
        \draw ({\i*\xstep},#1) -- ({(\i+1)*\xstep},#1);
    }

    \foreach \i in {-2,-1,0,1,2} {
        \node[big dot] at ({\i*\xstep},#1) {};
    }

    \hdotsleft{\xmTwo}{#1}
    \hdotsright{\xpTwo}{#1}
}

\newcommand{\attachEdge}[3]{%
    \coordinate (Eroot) at (#1,#2);
    \coordinate (Etip) at ($(Eroot)+(#3:\gsize)$);

    \draw (Eroot) -- (Etip);

    \node[big dot] at (Etip) {};
    \node[red dot] at (Eroot) {};
}

\newcommand{\attachTriangle}[3]{%
    \coordinate (Troot) at (#1,#2);
    \coordinate (Tone) at ($(Troot)+({#3-35}:\gsize)$);
    \coordinate (Ttwo) at ($(Troot)+({#3+35}:\gsize)$);

    \draw (Troot) -- (Tone) -- (Ttwo) -- cycle;

    \node[big dot] at (Tone) {};
    \node[big dot] at (Ttwo) {};
    \node[red dot] at (Troot) {};
}

\newcommand{\attachRhombus}[3]{%
    \coordinate (Rroot) at (#1,#2);
    \coordinate (Rone) at ($(Rroot)+({#3-35}:\gsize)$);
    \coordinate (Rtip) at ($(Rroot)+(#3:{1.15*\gsize})$);
    \coordinate (Rtwo) at ($(Rroot)+({#3+35}:\gsize)$);

    \draw (Rroot) -- (Rone) -- (Rtip) -- (Rtwo) -- cycle;

    \node[big dot] at (Rone) {};
    \node[big dot] at (Rtip) {};
    \node[big dot] at (Rtwo) {};
    \node[red dot] at (Rroot) {};
}

\newcommand{\attachKfour}[3]{%
    \coordinate (Kroot) at (#1,#2);
    \coordinate (Kone) at ($(Kroot)+(#3:{1.05*\gsize})$);
    \coordinate (Ktwo) at ($(Kroot)+({#3-42}:\gsize)$);
    \coordinate (Kthree) at ($(Kroot)+({#3+42}:\gsize)$);

    \draw (Kroot) -- (Kone);
    \draw (Kroot) -- (Ktwo);
    \draw (Kroot) -- (Kthree);
    \draw (Kone) -- (Ktwo);
    \draw (Kone) -- (Kthree);
    \draw (Ktwo) -- (Kthree);

    \node[big dot] at (Kone) {};
    \node[big dot] at (Ktwo) {};
    \node[big dot] at (Kthree) {};
    \node[red dot] at (Kroot) {};
}

\newcommand{\freeEdge}[3]{%
    \coordinate (Eroot) at (#1,#2);
    \coordinate (Etip) at ($(Eroot)+(#3:\gsize)$);

    \draw (Eroot) -- (Etip);

    \node[big dot] at (Eroot) {};
    \node[big dot] at (Etip) {};
}

\newcommand{\freeTriangle}[3]{%
    \coordinate (Troot) at (#1,#2);
    \coordinate (Tone) at ($(Troot)+({#3-35}:\gsize)$);
    \coordinate (Ttwo) at ($(Troot)+({#3+35}:\gsize)$);

    \draw (Troot) -- (Tone) -- (Ttwo) -- cycle;

    \node[big dot] at (Troot) {};
    \node[big dot] at (Tone) {};
    \node[big dot] at (Ttwo) {};
}

\newcommand{\freeRhombus}[3]{%
    \coordinate (Rroot) at (#1,#2);
    \coordinate (Rone) at ($(Rroot)+({#3-35}:\gsize)$);
    \coordinate (Rtip) at ($(Rroot)+(#3:{1.15*\gsize})$);
    \coordinate (Rtwo) at ($(Rroot)+({#3+35}:\gsize)$);

    \draw (Rroot) -- (Rone) -- (Rtip) -- (Rtwo) -- cycle;

    \node[big dot] at (Rroot) {};
    \node[big dot] at (Rone) {};
    \node[big dot] at (Rtip) {};
    \node[big dot] at (Rtwo) {};
}

\newcommand{\freeKfour}[3]{%
    \coordinate (Kroot) at (#1,#2);
    \coordinate (Kone) at ($(Kroot)+(#3:{1.05*\gsize})$);
    \coordinate (Ktwo) at ($(Kroot)+({#3-42}:\gsize)$);
    \coordinate (Kthree) at ($(Kroot)+({#3+42}:\gsize)$);

    \draw (Kroot) -- (Kone);
    \draw (Kroot) -- (Ktwo);
    \draw (Kroot) -- (Kthree);
    \draw (Kone) -- (Ktwo);
    \draw (Kone) -- (Kthree);
    \draw (Ktwo) -- (Kthree);

    \node[big dot] at (Kroot) {};
    \node[big dot] at (Kone) {};
    \node[big dot] at (Ktwo) {};
    \node[big dot] at (Kthree) {};
}

\begin{scope}[shift={(-5,0)}]

    \Zline{\rowtop}
    \Zline{\rowmid}

    \vdotsabove{0}{\rowtop+0.75}
    \vdotsbelow{0}{\rowbot+0.25}

    \attachEdge{\xmOne}{\rowtop}{90}
    \attachTriangle{\xpOne}{\rowtop}{80}

    \attachRhombus{\xZero}{\rowmid}{-90}
    \attachKfour{\xpTwo}{\rowmid}{75}

\end{scope}

\begin{scope}[shift={(5,0)}]

    \Zline{\rowtop}
    \Zline{\rowmid}

    \vdotsabove{0}{\rowtop+0.25}
    \vdotsbelow{0}{\rowbot+1}


\foreach \k in {1,2,3} {
    \fill ({-2.35-\k*\dotsep},{-1.9}) circle (0.7pt);
}

\foreach \k in {1,2,3} {
    \fill ({2.05+\k*\dotsep},{-1.9}) circle (0.7pt);
}

\freeEdge{-1.9}{-2.2}{90}
\freeTriangle{-0.95}{-2.2}{90}
\freeRhombus{0.15}{-2.2}{90}
\freeKfour{1.35}{-2.2}{90}

\end{scope}

\draw[dashed,<->] (-1.1,0.25) -- (1.1,0.25);

\path[use as bounding box] (-8.2,-2.1) rectangle (8.2,2.2);

\end{tikzpicture}
}
\caption{An example of a finite vertex split/join.\label{Fig:finite_split}}
\end{figure}
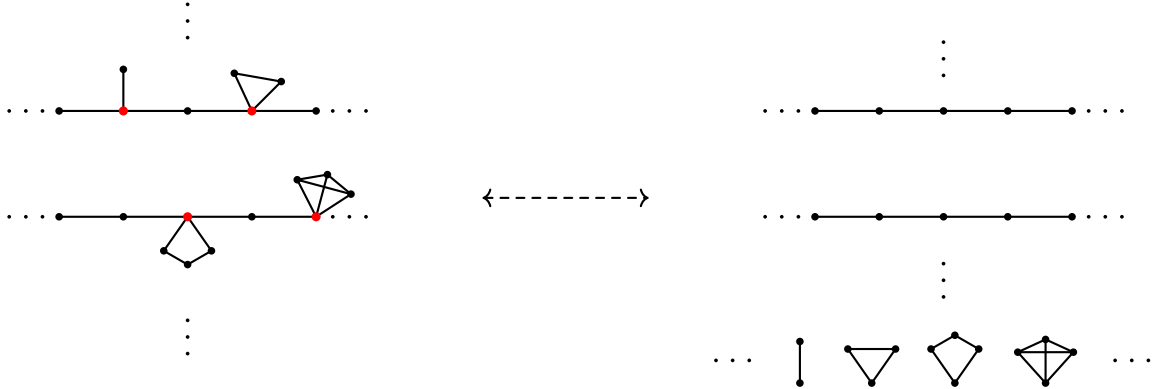
    
{\bf (2-ended infinite join.)} Let $\cG=(X,E,\mu)$ be a graphing with finite components, and $X_1,X_2 \subseteq X$ disjoint vertex sets, both containing one vertex from each connected component of $\cG$. We will measurably arrange the components of $\cG$ into $2$-ended components, joining them at $X_1$ and $X_2$. That is, let $\cH=(Y, F, \nu)$ be a graphing whose connected components are bi-infinite lines. Write $\vec{F}$ for a specified Borel orientation of the edges of $\cH$, and $\psi_0$ for a measure-preserving Borel bijection from $\vec{F}$ to $X_1$. Then $\psi_0$ naturally induces a Borel map $\psi$ from $\vec{F}$, assigning to each $\vec{f}$ the $\cG$-component of $\psi_0(\vec{f})$, with the first root from $X_1$ and the second from $X_2$. Given this input, we build the graphing $\mathcal{K}=(Z, L, \pi)$ by replacing $\vec{F}$ in $\cH$, gluing according to $\psi$.

{\bf (2-ended infinite split.)} The reverse operation of $2$-ended infinite join. That is, let $\mathcal{K}=(Z, L, \pi)$ be a graphing with 2-ended components, and $X_0$ a set of cut-vertices with exactly two infinite connected components in the cut. Assume also that $X_0$ has positive measure and meets every component of $\cK$. (Consequently, it meets every component infinitely many times, towards both ends.) We can define a graphing on $X_0$ by connecting adjacent points of $X_0$. That is, set $\cH=(Y,F,\nu)$ where $Y=X_0$, $\nu=\pi|_{X_0}$, and $xx' \in F$ if and only if $x$ and $x'$ are in the same $\mathcal{K}$-component, and there is no $x^* \in X_0$ separating them. For each such edge, let $V(x,x')$ denote the (finite) set of vertices of $\mathcal{K}$ that are not separated by either $x$ or $x'$ (including these two vertices). We write $C_{\mathcal{K}}(x,x')$ for the induced (doubly rooted) subgraph on $V(x,x')$. 

We also take two disjoint copies of $X_0$ and connect the corresponding vertices to form the graphing $\cG_0$ (every connected component is an edge). We denote its edge set by $J$, and fix a Borel orientation to form $\vec{J}$. We also fix a measure preserving Borel bijection $\psi_0: \vec{J} \to \vec{F}$. With slight abuse of notation, we write $\psi_0$ also for the map induced on the starting- and endvertices. Then $\psi_0$ induces a measurable map $\psi$, assigning $C_\mathcal{K}(\psi_0(x_0),\psi_0(x_1))$ to $(x_0,x_1) \in \vec{J}$. We obtain the split graphing $\cG=(X,E,\mu)$ by replacing $\vec{J}$ in $\cG_0$, gluing according to $\psi$.

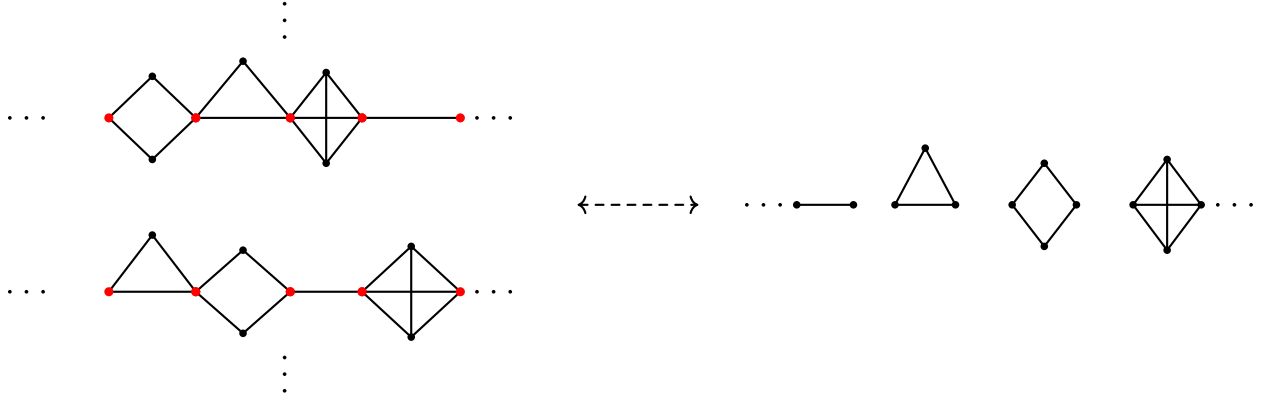
\begin{figure}[htb]
\centering
\makebox[\textwidth][c]{%
\begin{tikzpicture}[
    line join=round,
    line cap=round,
    every path/.style={thick},
    big dot/.style={circle, inner sep=0pt, minimum size=1mm, fill=black},
    red dot/.style={circle, inner sep=0pt, minimum size=1.2mm, fill=red}
]

\def\dotsep{0.22}

\newcommand{\gluedTriangle}[3]{%
    \coordinate (L) at (#1,#3);
    \coordinate (R) at (#2,#3);
    \coordinate (T) at ($(L)!0.5!(R)+(0,0.75)$);
    \draw (L) -- (T) -- (R) -- cycle;
    \node[red dot] at (L) {};
    \node[red dot] at (R) {};
    \node[big dot] at (T) {};
}

\newcommand{\gluedRhombus}[3]{%
    \coordinate (L) at (#1,#3);
    \coordinate (R) at (#2,#3);
    \coordinate (U) at ($(L)!0.5!(R)+(0,0.55)$);
    \coordinate (D) at ($(L)!0.5!(R)+(0,-0.55)$);
    \draw (L) -- (U) -- (R) -- (D) -- cycle;
    \node[red dot] at (L) {};
    \node[red dot] at (R) {};
    \node[big dot] at (U) {};
    \node[big dot] at (D) {};
}

\newcommand{\gluedEdge}[3]{%
    \coordinate (L) at (#1,#3);
    \coordinate (R) at (#2,#3);
    \draw (L) -- (R);
    \node[red dot] at (L) {};
    \node[red dot] at (R) {};
}

\newcommand{\gluedKfour}[3]{%
    \coordinate (L) at (#1,#3);
    \coordinate (R) at (#2,#3);
    \coordinate (U) at ($(L)!0.5!(R)+(0,0.60)$);
    \coordinate (D) at ($(L)!0.5!(R)+(0,-0.60)$);
    \draw (L)--(R);
    \draw (L)--(U);
    \draw (L)--(D);
    \draw (R)--(U);
    \draw (R)--(D);
    \draw (U)--(D);
    \node[red dot] at (L) {};
    \node[red dot] at (R) {};
    \node[big dot] at (U) {};
    \node[big dot] at (D) {};
}

\newcommand{\freeTriangleSeg}[3]{%
    \coordinate (L) at (#1,#3);
    \coordinate (R) at (#2,#3);
    \coordinate (T) at ($(L)!0.5!(R)+(0,0.75)$);
    \draw (L) -- (T) -- (R) -- cycle;
    \node[big dot] at (L) {};
    \node[big dot] at (R) {};
    \node[big dot] at (T) {};
}

\newcommand{\freeRhombusSeg}[3]{%
    \coordinate (L) at (#1,#3);
    \coordinate (R) at (#2,#3);
    \coordinate (U) at ($(L)!0.5!(R)+(0,0.55)$);
    \coordinate (D) at ($(L)!0.5!(R)+(0,-0.55)$);
    \draw (L) -- (U) -- (R) -- (D) -- cycle;
    \node[big dot] at (L) {};
    \node[big dot] at (R) {};
    \node[big dot] at (U) {};
    \node[big dot] at (D) {};
}

\newcommand{\freeEdgeSeg}[3]{%
    \coordinate (L) at (#1,#3);
    \coordinate (R) at (#2,#3);
    \draw (L) -- (R);
    \node[big dot] at (L) {};
    \node[big dot] at (R) {};
}

\newcommand{\freeKfourSeg}[3]{%
    \coordinate (L) at (#1,#3);
    \coordinate (R) at (#2,#3);
    \coordinate (U) at ($(L)!0.5!(R)+(0,0.60)$);
    \coordinate (D) at ($(L)!0.5!(R)+(0,-0.60)$);
    \draw (L)--(R);
    \draw (L)--(U);
    \draw (L)--(D);
    \draw (R)--(U);
    \draw (R)--(D);
    \draw (U)--(D);
    \node[big dot] at (L) {};
    \node[big dot] at (R) {};
    \node[big dot] at (U) {};
    \node[big dot] at (D) {};
}

\begin{scope}[shift={(-7,0)}]

    \foreach \k in {1,2,3} {
        \fill ({-0.65-\k*\dotsep},1.15) circle (0.7pt);
    }

    \gluedTriangle{1.15}{2.40}{1.15}
    \gluedRhombus{0}{1.15}{1.15}
    \gluedEdge{3.35}{4.65}{1.15}
    \gluedKfour{2.40}{3.35}{1.15}

    \foreach \k in {1,2,3} {
        \fill ({4.65+\k*\dotsep},1.15) circle (0.7pt);
    }

    \foreach \k in {1,2,3} {
        \fill ({-0.65-\k*\dotsep},-1.15) circle (0.7pt);
    }

    \gluedTriangle{0}{1.15}{-1.15}
    \gluedRhombus{1.15}{2.40}{-1.15}
    \gluedEdge{2.40}{3.35}{-1.15}
    \gluedKfour{3.35}{4.65}{-1.15}

    \foreach \k in {1,2,3} {
        \fill ({4.65+\k*\dotsep},-1.15) circle (0.7pt);
    }

    \foreach \k in {1,2,3} {
        \fill (2.325,{2+\k*\dotsep}) circle (0.7pt);
        \fill (2.325,{-1.8-\k*\dotsep}) circle (0.7pt);
    }

\end{scope}

\begin{scope}[shift={(4,0)}]

    \foreach \k in {1,2,3} {
        \fill ({-1.90-\k*\dotsep},0) circle (0.7pt);
    }

    \freeEdgeSeg{-1.90}{-1.15}{0}
    \freeTriangleSeg{-0.60}{0.20}{0}
    \freeRhombusSeg{0.95}{1.80}{0}
    \freeKfourSeg{2.55}{3.45}{0}

    \foreach \k in {1,2,3} {
        \fill ({3.45+\k*\dotsep},0) circle (0.7pt);
    }

\end{scope}

\draw[dashed,<->] (-0.8,0) -- (0.8,0);

\pgfresetboundingbox
\path[use as bounding box] (-8.2,-2.7) rectangle (8.2,2.7);

\end{tikzpicture}%
}
\caption{An example of a 2-ended infinite split/join.\label{fig:2ended_split}}
\end{figure}

{\bf (Finite Whitney twist.)} Let $\cG=(X,E,\mu)$ be a graphing, and $X_0,X_1 \subseteq X$ subsets of corresponding cut pairs. That is, we have a measure preserving Borel bijection $f:X_0 \to X_1$, $f(x) \in [x]_{\cG}$, and $\{x,f(x)\}$ is a cut pair of the component $[x]_{\cG}$ for all $x$. Assume further that at least one of the components in these cuts is finite, contains no vertices from $X_0 \cup X_1$, and that such a finite component is already chosen in a Borel way. (Such a Borel choice is possible, as local finiteness implies that there are finitely many components of the cut.) In particular, at each $\{x,f(x)\}$ pair we have (Borel) lists of edges $l(x)$ and $l(f(x))$ listing all the edges going towards a single finite component of $[x]_{\cG} \setminus \{x,f(x)\}$, incident to $x$ and $f(x)$ respectively. We obtain the Whitney twist of $\cG$ at $(X_0,X_1)$ by swapping $x$ and $f(x)$ as endvertices of the edges in $l(x)$ and $l(f(x))$. That is, writing $x' = f(x)$, for each edge $xu \in l(x)$ replace it with an edge $x'u$, and for each $x'u \in l(f(x))$ replace it with $xu$.  

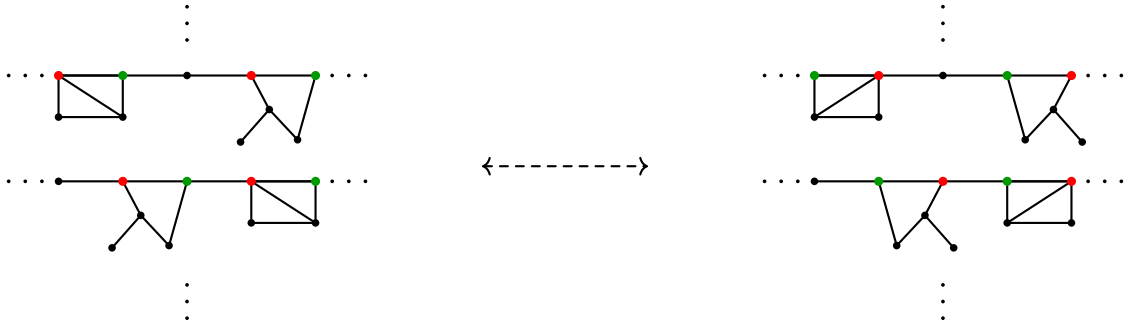
\begin{figure}[htb]
\centering
\makebox[\textwidth][c]{%
\begin{tikzpicture}[
    line join=round,
    line cap=round,
    every path/.style={thick},
    big dot/.style={circle, inner sep=0pt, minimum size=1mm, fill=black},
    red dot/.style={circle, inner sep=0pt, minimum size=1.2mm, fill=red},
    green dot/.style={circle, inner sep=0pt, minimum size=1.2mm, fill=green!60!black}
]

\def\xstep{0.85}
\def\dotsep{0.22}
\def\rowtop{1.4}
\def\rowmid{0}
\def\rowbot{-1.4}

\pgfmathsetmacro{\xmTwo}{-2*\xstep}
\pgfmathsetmacro{\xmOne}{-1*\xstep}
\pgfmathsetmacro{\xZero}{0}
\pgfmathsetmacro{\xpOne}{1*\xstep}
\pgfmathsetmacro{\xpTwo}{2*\xstep}

\newcommand{\hdotsleft}[2]{%
    \foreach \k in {1,2,3} {
        \fill ({#1-\k*\dotsep},#2) circle (0.7pt);
    }
}

\newcommand{\hdotsright}[2]{%
    \foreach \k in {1,2,3} {
        \fill ({#1+\k*\dotsep},#2) circle (0.7pt);
    }
}

\newcommand{\vdotsabove}[2]{%
    \foreach \k in {1,2,3} {
        \fill (#1,{#2+\k*\dotsep}) circle (0.7pt);
    }
}

\newcommand{\vdotsbelow}[2]{%
    \foreach \k in {1,2,3} {
        \fill (#1,{#2-\k*\dotsep}) circle (0.7pt);
    }
}

\newcommand{\Zline}[1]{%
    \foreach \i in {-2,-1,0,1} {
        \draw ({\i*\xstep},#1) -- ({(\i+1)*\xstep},#1);
    }
    \foreach \i in {-2,-1,0,1,2} {
        \node[big dot] at ({\i*\xstep},#1) {};
    }
    \hdotsleft{\xmTwo}{#1}
    \hdotsright{\xpTwo}{#1}
}


\newcommand{\diagSquareLeft}[3]{%
    \coordinate (R) at (#1,#3);
    \coordinate (G) at (#2,#3);
    \coordinate (u) at (#1,{#3-0.55});
    \coordinate (v) at (#2,{#3-0.55});
    \draw (R)--(u)--(v)--(G)--cycle;
    \draw (R)--(v);
    \node[red dot] at (R) {};
    \node[green dot] at (G) {};
    \node[big dot] at (u) {};
    \node[big dot] at (v) {};
}

\newcommand{\diagSquareRight}[3]{%
    \coordinate (G) at (#1,#3);
    \coordinate (R) at (#2,#3);
    \coordinate (u) at (#1,{#3-0.55});
    \coordinate (v) at (#2,{#3-0.55});
    \draw (G)--(u)--(v)--(R)--cycle;
    \draw (R)--(u);
    \node[green dot] at (G) {};
    \node[red dot] at (R) {};
    \node[big dot] at (u) {};
    \node[big dot] at (v) {};
}

\newcommand{\tailPathLeft}[3]{%
    \coordinate (R) at (#1,#3);
    \coordinate (G) at (#2,#3);
    \coordinate (u) at ({0.72*#1+0.28*#2},{#3-0.45});
    \coordinate (v) at ({0.28*#1+0.72*#2},{#3-0.85});
    \coordinate (w) at ({0.72*#1+0.28*#2-0.38},{#3-0.88});
    \draw (R)--(u)--(v)--(G);
    \draw (u)--(w);
    \node[red dot] at (R) {};
    \node[green dot] at (G) {};
    \node[big dot] at (u) {};
    \node[big dot] at (v) {};
    \node[big dot] at (w) {};
}

\newcommand{\tailPathRight}[3]{%
    \coordinate (G) at (#1,#3);
    \coordinate (R) at (#2,#3);
    \coordinate (u) at ({0.72*#1+0.28*#2},{#3-0.85});
    \coordinate (v) at ({0.28*#1+0.72*#2},{#3-0.45});
    \coordinate (w) at ({0.28*#1+0.72*#2+0.38},{#3-0.88});
    \draw (G)--(u)--(v)--(R);
    \draw (v)--(w);
    \node[green dot] at (G) {};
    \node[red dot] at (R) {};
    \node[big dot] at (u) {};
    \node[big dot] at (v) {};
    \node[big dot] at (w) {};
}

\begin{scope}[shift={(-5,0)}]

    \Zline{\rowtop}
    \Zline{\rowmid}

    \vdotsabove{0}{\rowtop+0.25}
    \vdotsbelow{0}{\rowbot+0.25}

    \diagSquareLeft{\xmTwo}{\xmOne}{\rowtop}
    \tailPathLeft{\xpOne}{\xpTwo}{\rowtop}

    \tailPathLeft{\xmOne}{\xZero}{\rowmid}
    \diagSquareLeft{\xpOne}{\xpTwo}{\rowmid}

\end{scope}

\begin{scope}[shift={(5,0)}]

    \Zline{\rowtop}
    \Zline{\rowmid}

    \vdotsabove{0}{\rowtop+0.25}
    \vdotsbelow{0}{\rowbot+0.25}

    \diagSquareRight{\xmTwo}{\xmOne}{\rowtop}
    \tailPathRight{\xpOne}{\xpTwo}{\rowtop}

    \tailPathRight{\xmOne}{\xZero}{\rowmid}
    \diagSquareRight{\xpOne}{\xpTwo}{\rowmid}

\end{scope}

\draw[dashed,<->] (-1.1,0.2) -- (1.1,0.2);

\pgfresetboundingbox
\path[use as bounding box] (-8.3,-2) rectangle (8.3,2.2);

\end{tikzpicture}%
}
\caption{An example of finite Whitney twist.\label{fig:finite_twist}}
\end{figure}

{\bf (2-ended simultaneous Whitney twist.)} Let $\cG$ and $f:X_0 \to X_1$ be as above, except assume all components of $\cG$ are 2-ended, and each cut pair $\{x,f(x)\}$ separates the two ends. Moreover, assume that the pairs are non-crossing: for $x,y \in X_0$ from the same $\cG$-component, the cut pair $\{x,f(x)\}$ does not separate $y$ and $f(y)$. As before, for $x \in X_0$ let $l(x), l(f(x))$ denote the list of the edges towards one of the two ends. This end is chosen in a Borel way at each $\{x,f(x)\}$, but the choice does not need to be consistent across the whole $\cG$-component. The Whitney twist is obtained by swapping $x$ and $f(x)$ as endvertices of the $l(x)$ and $l(f(x))$, as in the previous case.

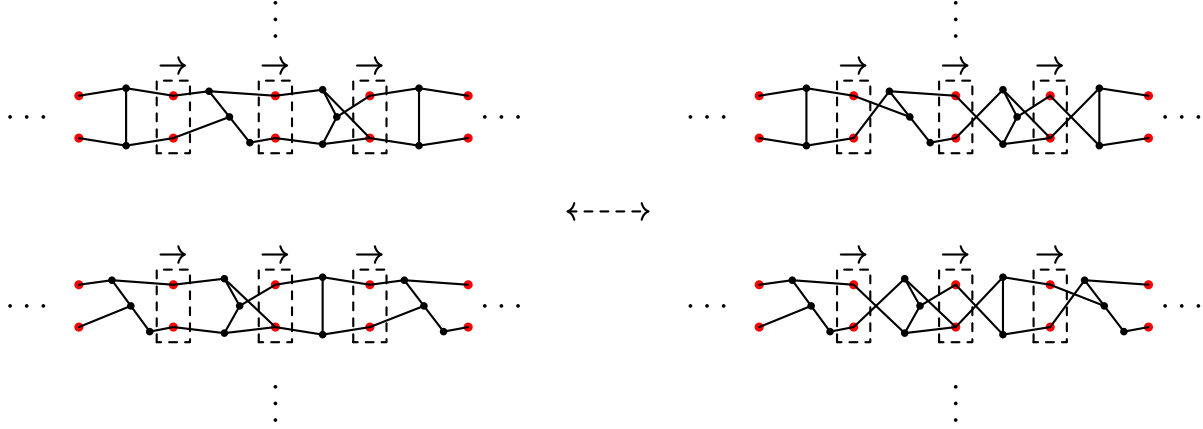
\begin{figure}[htb]
\centering
\makebox[\textwidth][c]{%
\begin{tikzpicture}[
    line join=round,
    line cap=round,
    every path/.style={thick},
    big dot/.style={circle, inner sep=0pt, minimum size=1mm, fill=black},
    red dot/.style={circle, inner sep=0pt, minimum size=1.2mm, fill=red}
]

\def\dotsep{0.22}
\def\vsep{0.28}
\def\ytop{1.25}
\def\ybot{-1.25}

\def\mkcut#1#2#3{%
    \coordinate (#1T) at (#2,{#3+\vsep});
    \coordinate (#1B) at (#2,{#3-\vsep});
    \node[red dot] at (#1T) {};
    \node[red dot] at (#1B) {};
}

\def\markcut#1#2{%
    \draw[dashed] ({#1-0.22},{#2-0.48}) rectangle ({#1+0.22},{#2+0.48});
    \draw[->] ({#1-0.17},{#2+0.68}) -- ({#1+0.17},{#2+0.68});
}

\def\rowdots#1{%
    \foreach \k in {1,2,3} {
        \fill ({-0.25-\k*\dotsep},#1) circle (0.7pt);
        \fill ({5.15+\k*\dotsep},#1) circle (0.7pt);
    }
}

\def\verticaldots{%
    \foreach \k in {1,2,3} {
        \fill (2.60,{\ytop+0.85+\k*\dotsep}) circle (0.7pt);
        \fill (2.60,{\ybot-0.85-\k*\dotsep}) circle (0.7pt);
    }
}


\def\modA#1#2#3#4#5#6#7#8{%
    \coordinate (#1U) at ({0.5*(#6+#7)},{#8+0.38});
    \coordinate (#1D) at ({0.5*(#6+#7)},{#8-0.38});

    \draw (#2) -- (#1U) -- (#4);
    \draw (#3) -- (#1D) -- (#5);
    \draw (#1U) -- (#1D);

    \node[big dot] at (#1U) {};
    \node[big dot] at (#1D) {};
}

\def\modB#1#2#3#4#5#6#7#8{%
    \coordinate (#1U) at ({#6+0.35*(#7-#6)},{#8+0.34});
    \coordinate (#1M) at ({#6+0.55*(#7-#6)},{#8});
    \coordinate (#1D) at ({#6+0.75*(#7-#6)},{#8-0.34});

    \draw (#2) -- (#1U) -- (#1M) -- (#1D) -- (#5);
    \draw (#3) -- (#1M);
    \draw (#1U) -- (#4);

    \node[big dot] at (#1U) {};
    \node[big dot] at (#1M) {};
    \node[big dot] at (#1D) {};
}

\def\modC#1#2#3#4#5#6#7#8{%
    \coordinate (#1U) at ({0.5*(#6+#7)},{#8+0.36});
    \coordinate (#1D) at ({0.5*(#6+#7)},{#8-0.36});
    \coordinate (#1M) at ({#6+0.65*(#7-#6)},{#8});

    \draw (#2) -- (#1U);
    \draw (#3) -- (#1D);
    \draw (#1U) -- (#1M) -- (#1D);
    \draw (#1U) -- (#5);
    \draw (#1M) -- (#4);
    \draw (#1D) -- (#5);

    \node[big dot] at (#1U) {};
    \node[big dot] at (#1M) {};
    \node[big dot] at (#1D) {};
}

\begin{scope}[shift={(-7,0)}]

    \mkcut{LT0}{0.00}{\ytop}
    \mkcut{LT1}{1.25}{\ytop}
    \mkcut{LT2}{2.60}{\ytop}
    \mkcut{LT3}{3.85}{\ytop}
    \mkcut{LT4}{5.15}{\ytop}

    \modA{LTA}{LT0T}{LT0B}{LT1T}{LT1B}{0.00}{1.25}{\ytop}
    \modB{LTB}{LT1T}{LT1B}{LT2T}{LT2B}{1.25}{2.60}{\ytop}
    \modC{LTC}{LT2T}{LT2B}{LT3T}{LT3B}{2.60}{3.85}{\ytop}
    \modA{LTD}{LT3T}{LT3B}{LT4T}{LT4B}{3.85}{5.15}{\ytop}

    \markcut{1.25}{\ytop}
    \markcut{2.60}{\ytop}
    \markcut{3.85}{\ytop}

    \rowdots{\ytop}

    \mkcut{LB0}{0.00}{\ybot}
    \mkcut{LB1}{1.25}{\ybot}
    \mkcut{LB2}{2.60}{\ybot}
    \mkcut{LB3}{3.85}{\ybot}
    \mkcut{LB4}{5.15}{\ybot}

    \modB{LBA}{LB0T}{LB0B}{LB1T}{LB1B}{0.00}{1.25}{\ybot}
    \modC{LBB}{LB1T}{LB1B}{LB2T}{LB2B}{1.25}{2.60}{\ybot}
    \modA{LBC}{LB2T}{LB2B}{LB3T}{LB3B}{2.60}{3.85}{\ybot}
    \modB{LBD}{LB3T}{LB3B}{LB4T}{LB4B}{3.85}{5.15}{\ybot}

    \markcut{1.25}{\ybot}
    \markcut{2.60}{\ybot}
    \markcut{3.85}{\ybot}

    \rowdots{\ybot}
    \verticaldots

\end{scope}

\begin{scope}[shift={(2,0)}]

    \mkcut{RT0}{0.00}{\ytop}
    \mkcut{RT1}{1.25}{\ytop}
    \mkcut{RT2}{2.60}{\ytop}
    \mkcut{RT3}{3.85}{\ytop}
    \mkcut{RT4}{5.15}{\ytop}

    \modA{RTA}{RT0T}{RT0B}{RT1T}{RT1B}{0.00}{1.25}{\ytop}

    \modB{RTB}{RT1B}{RT1T}{RT2T}{RT2B}{1.25}{2.60}{\ytop}

    \modC{RTC}{RT2B}{RT2T}{RT3T}{RT3B}{2.60}{3.85}{\ytop}

    \modA{RTD}{RT3B}{RT3T}{RT4T}{RT4B}{3.85}{5.15}{\ytop}

    \markcut{1.25}{\ytop}
    \markcut{2.60}{\ytop}
    \markcut{3.85}{\ytop}

    \rowdots{\ytop}

    \mkcut{RB0}{0.00}{\ybot}
    \mkcut{RB1}{1.25}{\ybot}
    \mkcut{RB2}{2.60}{\ybot}
    \mkcut{RB3}{3.85}{\ybot}
    \mkcut{RB4}{5.15}{\ybot}

    \modB{RBA}{RB0T}{RB0B}{RB1T}{RB1B}{0.00}{1.25}{\ybot}

    \modC{RBB}{RB1B}{RB1T}{RB2T}{RB2B}{1.25}{2.60}{\ybot}

    \modA{RBC}{RB2B}{RB2T}{RB3T}{RB3B}{2.60}{3.85}{\ybot}

    \modB{RBD}{RB3B}{RB3T}{RB4T}{RB4B}{3.85}{5.15}{\ybot}

    \markcut{1.25}{\ybot}
    \markcut{2.60}{\ybot}
    \markcut{3.85}{\ybot}

    \rowdots{\ybot}
    \verticaldots

\end{scope}

\draw[dashed,<->] (-0.55,0) -- (0.55,0);

\pgfresetboundingbox
\path[use as bounding box] (-7.6,-3.0) rectangle (7.6,3.0);

\end{tikzpicture}%
}
\caption{An example of a 2-ended simultaneous Whitney twist. \label{fig:2ended_twist}}
\end{figure}

\begin{rem}
    The (Borel) choice of the ends does not change, up to isomorphism of graphings, the result of the $2$-ended twist. Indeed, choosing differently at $\{x,f(x)\}$ can be isomorphically obtained from the original choice by swapping $x$ and $f(x)$. (For any two Borel choices of ends, the pairs where they differ forms a Borel set, so swapping these $\{x,f(x)\}$ while fixing all other vertices is a measure preserving Borel bijection on the vertex set.) 
\end{rem}

Note that each of the above operations naturally defines an edge bijection on the graphings. We will refer to it as the edge bijection induced by the operation.

\subsection{Locally finite sequences of Whitney operations} \label{subsec:infinitely_many_operations}

In the finite case, any weak isomorphism can be implemented by a finite number of Whitney operations. For graphings, however, one might need infinitely many operations to achieve this. Indeed, let $\cG_1$  be a graphing that is a perfect matching between two Borel sets (of equal measure), and $\cG_2$ a $1$-ended treeing. Any measure preserving bijection of the edge sets is a weak isomorphism, but we need to apply infinitely many finite vertex splits to decompose the $1$-ended trees into edges.

Therefore, we will allow \emph{locally finite sequences} of Whitney operations. By local finiteness we mean that almost every vertex is affected by only finitely many splits, joins, or twists.

\begin{ex} \label{ex:1-ended_tree_transform}
    Any pair of $1$-ended treeings $\cG_2$ and $\cG_2'$ can be obtained from one another by performing two locally finite sequences of Whitney operations. Indeed, we can decompose $\cG_2$ into the measurable matching $\cG_1$ by a locally finite sequence of finite vertex splits. Similarly, $\cG_2'$ can be decomposed into $\cG_1$ by a locally finite sequence of vertex splits. By reversing the second decomposition, we can reassemble $\cG_2'$ from $\cG_1$ by a locally finite sequence of finite vertex joins. 
\end{ex}

\begin{rem}
    A series of remarks are in order:
\begin{enumerate}
    \item It does not seem obvious to us that one can reverse any locally finite sequence of Whitney operations by another locally finite sequence. In the previous example, however, this is actually the case. We have a locally finite sequence of finite vertex splits $\alpha_1, \alpha_2, \ldots$, turning $\cG'_2$ into $\cG_1$. The reverse operation of each $\alpha_i$ is a finite vertex join, which we denote by $\alpha_i^{-1}$. But doing the $\alpha_i^{-1}$ in reverse order is not possible. Fortunately, if we perform the joins in the original order ($\alpha_1^{-1}$ first, $\alpha^{-1}_2$ second, etc.) we get a locally finite sequence of finite vertex joins that turn $\cG_1$ back into $\cG_2'$.
    
    \item The notation in the previous item is still imprecise (though only slightly). Strictly speaking, it does not make sense to apply the finite vertex join $\alpha^{-1}_1$ to $\cG_1$, since $\cG_1$ is a matching, whereas the graphing obtained by applying $\alpha_1$ to $\cG_2'$ is the union of some isolated edges (that we split off) and $1$-ended trees (that remain). Let us denote by $A_1$ the former, and by $B_1$ the latter. When we apply $\alpha_2$, we split $B_1$ further into isolated edges (denoted $A_2$) and $1$-ended trees (denoted $B_2$), etc. In the end $E(\cG_1) = \bigcup_{n=1}^{\infty} A_n$. When we apply  $\alpha_1^{-1}$ to $\cG_1$, we mean that we glue the edges of $A_1$ appropriately to the edges of $A_2$, and so on for each later $\alpha_i^{-1}$. In this reverse process, after each operation, we still have finite components. The $1$-ended tree only emerges after infinitely many operations. 
    
    \item Fortunately, the only reversal of a locally finite sequence that we will need in our proof of Theorem~\ref{thm:Whitney_2-isom_intro} is of this kind.

    \item Let $\beta_1, \beta_2, \dots$ denote the finite vertex splits disassembling $\cG_2$ into $\cG_1$. It seems possible to run the two locally finite sequences $(\beta_n)_{n\in \N}$ and $(\alpha_n^{-1})_{n\in \N}$ in an intertwining fashion to implement a measure preserving bijection $\varphi: E(\cG_2) \to E(\cG_2')$ by only one locally finite sequence. If for a positive measure set of level 1 vertices of $\cG_2$ the $\varphi$-preimages of all the leaf-edges adjacent to them are already disassembled by the first finitely many $\beta_i$'s, we can assemble them by applying the appropriate restriction of $\alpha_1^{-1}$ to these edges before completing the full disassembly of $\cG_2$.
    
    \item More generally, it could be the case that the result of performing two locally finite sequences can always be obtained by a single locally finite sequence. We did not pursue this, as we felt it would only add technicality to the paper. 

    \item One could also argue that splitting a $1$-ended treeing into single edges should be available as a single Whitney operation. After all, the same is considered a single operation in the 2-ended case. We opted to allow the most basic operations that are necessary to implement weak isomorphisms between graphings. This comes at the cost of having to use multiple locally finite sequences. 
    
    \item One cannot hope to implement weak isomorphisms between graphings by finitely many operations, even if one allows more complicated operations, like splitting $\leq2$-ended graphs into their $2$-connected components, or allowing simultaneous Whitney twists at many non-crossing cut pairs. To see this, notice that any permutation of the edges of an $n$-cycle is a weak isomorphism. Yet we cannot implement all these by boundedly many operations of performing simultaneous Whitney twists at non-crossing cut pairs. The number of possible choices of such cut pairs is the so-called \emph{super Catalan number} \cite{oeisA001003}, which grows exponentially in $n$. So the product of boundedly many of them cannot reproduce $n!$ different permutations. 
\end{enumerate}
\end{rem}
We conclude this subsection by the following observation. 

\begin{lem} \label{lem:locally_finite_sequence_gives_weak_isom}
    Any locally finite sequence of Whitney operations from $\cG_1$ to $\cG_2$ induces a rank-preserving bijection of the edge sets.  
\end{lem}

\begin{proof}
    By Lemma~\ref{lem:eq_of_rankpres} preserving the rank is equivalent to preserving cycles and hyperfiniteness. These are clear from the definitions of the operations, so the edge-bijection is rank-preserving after finitely many operations. We only need to argue that the same holds after performing the whole locally finite sequence. For cycles this is clear, in both direction.
    
    Let $F \subseteq E(\cG_1)$ be a hyperfinite edge set. Equivalently, the equivalence relation $\mathcal{E}_F$ on $F$ of being in the same connected component is hyperfinite. The edge set is not changed by the Whitney operations (up to the induced measure-preserving Borel edge-bijection), so we can take the point of view that the sequence of operations defines a sequence of equivalence relations $(\cE_n)_{n\in \N}$ on $F$, where $\cE_n$ relates edges that are in the same component after performing the first $n$ operations. We argued above that each $\cE_n$ is hyperfinite. Furthermore, let $\cE_{\infty}$ denote the connectedness relation after the completion of the whole sequence. We need to argue that $\cE_{\infty}$ is also hyperfinite. 

    For each $n \in \N$ denote by $F_n \subseteq F$ the set of edges which are not altered after the $n$-th Whitney operation. $F_n $ is increasing, and, by our local finiteness assumption, $\cup_{n \in \N} F_n = F$. Let us denote by $\cE'_n$ the equivalence relation on $F$ defined by the $F_n$-connected components. (Edges in $F \setminus F_n$ are singleton classes in $\cE'_n$.) Clearly $\cE'_n \subseteq \cE_n$, so  $\cE'_n$ is also hyperfinite. Note that $\cE'_n$ is an increasing sequence: these edges are not altered later; in particular, ones that are connected at some point, cannot afterwards get disconnected. (In contrast, $\cE_n$ itself need not be increasing.) Finally, observe that $\cE_\infty = \bigcup_{n \in \N} \cE'_n$. As the increasing union of hyperfinite relations is hyperfinite in the measure-preserving context, our proof is complete.

    To show that hyperfiniteness is preserved in the other direction as well, now let $F$ be a hyperfinite edge set in $\cG_2$, and $F_0$ a treeing of $F$. The preimage of $F_0$ is still acyclic in $\cG_1$. Moreover, because every edge in $F\setminus F_0$ has a base-cycle of $F_0$-edges in $\cG_2$, the same is true in $\cG_1$, so $F_0$ spans the connected components of $F$ in $\cG_1$. Consequently $F$ is hyperfinite in $\cG_1$ if and only if $F_0$ is, so by possibly passing to $F_0$, we may assume that $F$ is acyclic. Assume towards contradiction that $F$ is not hyperfinite in $\cG_1$. Nevertheless, it is a forest, so by taking the double-ray core of $F$ in $\cG_1$, we may also assume that $F$ is an infinitely-ended leafless forest in $\cG_1$, while it remains a hyperfinite subforest of $\cG_2$. 

    We now observe that components of such an $F$ cannot change during a Whitney operation. Components of $F$ cannot be contained in $2$-ended components of $\cG_1$, so only finite splits, joins, and Whitney-twists are possible. The components of $F$ are weakly $2$-connected, therefore a finite split cannot disconnect them, and as they are infinite, a finite join cannot merge $F$-components. Furthermore, a finite Whitney-twist can only be applied at pairs of vertices that are not separated by trifurcations of $F$. These can permute edges within segments between neighboring trifurcation vertices, but do not change the $F$-components. This remains true after finitely many operations. As each edge becomes unaltered eventually, in the limit the image of $F$ has the same trees as connected components. This contradicts the fact that $F$ is hyperfinite in $\cG_2$.
\end{proof}

\begin{rem}
   An alternative proof for the forward hyperfiniteness preservation can be given using Benjamini--Schramm convergence \cite{benjamini2001recurrence}, for more information see \cite{lovasz2012large}. By Lemma~\ref{lem:eq_of_rankpres} it is enough to show that the image of a hyperfinite forest is a hyperfinite forest. Let $\cF\subseteq \cG_1$ be a hyperfinite subforest of $\cG_1$ and $\cF_n$ denote the image of $\cF$ after the first $n$ operations. Clearly these are still hyperfinite forests. Let $\cF_{\infty}$ denote the image after the whole sequence is performed. By the local finiteness of the sequence, $\cF_n \to \cF_{\infty}$ in the Benjamini-Schramm sense. Finally, the Benjamini-Schramm limit of hyperfinite \emph{treeings} is a hyperfinite treeing, see \cite{schramm2011hyperfinite, aldous2007processes}.
\end{rem}

\subsection{Classification in the infinitely-ended case}

We first establish Whitney's $2$-isomorphism theorem for weakly $2$-connected, infinitely-ended graphings. This part of the proof relies on the novel tools developed in this paper, namely covers by infinitely-ended leafless subforests and banana decompositions. The argument here is relatively short, as all preparation was done in Subsection~\ref{subsec:preserving_comps_ends_bananas}, building on Theorems~\ref{thm:graph_isom} and \ref{thm:union_of_inf_ended_leafless}.

\begin{thm}\label{thm:infty_ended_whitney_twists}
    Let $\cG_1$ and $\cG_2$ be infinitely-ended, weakly $2$-connected graphings, and $\varphi: E(\cG_1) \to E(\cG_2)$ a weak isomorphism. Then, modulo null sets, there is a locally finite sequence of finite Whitney twists implementing $\varphi$. 
\end{thm}
\begin{proof}
    By Theorem~\ref{thm:connected_preserving} and Lemma~\ref{cor:bananas_preserved}, $\varphi$ preserves the connected components and the banana decomposition. Consequently, it induces a map $\hat{\varphi}:E(\texttt{Ban}(\cG_1)) \to E(\texttt{Ban}(\cG_2))$. This map is rank-preserving, as we can construct $\texttt{Ban}(\cG_1)$ and $\texttt{Ban}(\cG_2)$ by minor operations that delete and contract $\varphi$-matched edges. As $\texttt{Ban}(\cG_1)$ is weakly $3$-connected, by Theorem~\ref{thm:graph_isom},  $\hat{\varphi}$ is induced by a graphing isomorphism. It remains to apply Whitney twists to $\cG_1$ in order to ensure that $\varphi$ restricted to some maximal banana $B \subseteq \cG_1$ is induced by a graph isomorphism from $B$ to $\varphi(B)$. By Corollary~\ref{cor:extended_banana_cycle_preserving} $\varphi: E(\overline{B}) \to E(\overline{\varphi(B)})$ preserves cycles, so by the finite Whitney theorem $\varphi|_{E(\overline{B})}$ can be implemented by a finite sequence of Whitney twists. We can assume that the auxiliary edge connecting $\partial B$ is always on the unaltered side of the twists. As the bananas are finite, we can choose such a sequence measurably for each maximal banana $B$. Since maximal bananas are edge-disjoint, and the graphings are locally finite, the Borel conflict graph encoding nonempty vertex-intersection of maximal bananas is locally finite. Taking a countable Borel coloring, we can implement one Whitney twist in each maximal banana in one color class simultaneously by a finite Whitney twist of $\cG_1$. We might need a locally finite sequence of such finite Whitney twists, as the size of maximal bananas in $\cG_1$ might not be bounded.
\end{proof}

\subsection{Classification for components with at most two ends}

\begin{thm} \label{thm:1-2-ended}
    Let $\cG_1$ and $\cG_2$ be strongly $2$-connected graphings whose almost every component is $0$-, $1$-, or $2$-ended, and $\varphi: E(\cG_1) \to E(\cG_2)$ a measure-preserving Borel bijection that preserves cycles. Then there is a locally finite sequence of finite  or $2$-ended Whitney twists implementing $\varphi$ modulo null sets. 
\end{thm}

We emphasize that this part of the argument does not rely on hyperfiniteness being preserved, only on the fact that cycles are. Note also that for finite components, the content of the statement is simply the finite Whitney theorem. This part of our proof uses the locally finite tools introduced in Subsection~\ref{subsec:locally_finite_tools}.

\begin{proof}
    A graph is $2$-connected if and only if any pair of edges admits a cycle containing them. Consequently, as the $2$-connected components of $\cG_1$ and $\cG_2$ coincide with the connected components, $\varphi$ preserves connected components.

    We consider the Tutte decomposition of the components of $\cG_1$. By Lemma~\ref{lemma:tutte_decomp_preserved}, $\varphi$ induces cycle-preserving maps on the components of the decompositions, including virtual edges. (This data is all Borel, see Remark~\ref{rem:Tutte_decomp_borel}.)
    
    Tutte components that are $3$-connected or $k$-links are mapped isomorphically under $\varphi$. This need not be the case for cycles, as any bijection of the edges is cycle-preserving. At the same time, any permutation of the edges of a cycle can be achieved by a finite sequence of Whitney twists. Therefore, by applying the appropriate Whitney twists to $\cG_1$, we may assume that $\varphi$ induces isomorphisms on the cycles as well. Indeed, as cycles are finite, for each one we can choose the finite sequence of twists measurably. We can measurably perform twists simultaneously on cycles that are vertex-disjoint. Let \(\mC\) be the Borel family of cyclic Tutte components and put a Borel graph on \(\mC\) by joining two components when their vertex sets intersect. This graph is locally finite, as the ``moreover'' part of Theorem~\ref{thm:droms_servatius_servatius} implies that every vertex of \(G\) belongs to only finitely many Tutte components. Hence \(\mC\) admits a countable Borel proper coloring. We may perform the prescribed twists simultaneously on each color class. Note that we may need $2$-ended Whitney twists, if there are infinitely many cyclic Tutte-components along a bi-infinite path in the Tutte-tree. Also, we may need an infinite sequence of such simultaneous twists, as the cycles might not have bounded length. Nevertheless, as the cycles are finite, the sequence of twists is locally finite.
    
    As the decomposition is preserved, and Tutte-components are mapped isomorphically, it only remains to consider how the amalgamations might differ in $\cG_1$ and $\cG_2$. We say two non-$k$-link components are \emph{neighbors}, if they are neighbors in the Tutte tree, or at distance 2, separated by a $k$-link. Any pair of neighboring components shares a cut pair (of vertices) after the amalgamation. Consider such neighboring components $G,H$ in $\cG_1$, with common vertices $\{x,y\}$. Then $\varphi(G), \varphi(H)$ share two vertices, say $u, v \in V(\cG_2)$. Then $\operatorname{star}_{G}(x)$ and $\operatorname{star}_{G}(y)$ are mapped to $\operatorname{star}_{\varphi(G)}(u)$ and $\operatorname{star}_{\varphi(G)}(v)$, though not necessarily in order; and similarly with $H$ and $\varphi(H)$. Without loss of generality, we can assume $\operatorname{star}_{G}(x)$ is mapped to $\operatorname{star}_{\varphi(G)}(u)$ and $\operatorname{star}_{G}(y)$ to $\operatorname{star}_{\varphi(G)}(v)$. Depending on the orientation of the amalgamation of $\varphi(G)$ and $\varphi(H)$ in $\cG_2$, either $\operatorname{star}_{H}(x)$ is mapped to $\operatorname{star}_{\varphi(H)}(u)$ and $\operatorname{star}_{H}(y)$ to $\operatorname{star}_{\varphi(H)}(v)$, or the other way around. In the first case, we say the orientation of the amalgamation of $G$ and $H$ \emph{agrees} in $\cG_1$ and $\cG_2$, and we say it \emph{disagrees} in the second case. 

    If all orientations of amalgamations agree, $\varphi$ is induced by an isomorphism of graphings, and the proof is complete. On the other hand, disagreements can be resolved by Whitney twists of $\cG_1$ in a straightforward manner. Indeed, when there is disagreement at two non-$k$-link components that are directly amalgamated, we simply twist at the cut pair corresponding to the amalgamated edges. When more components are amalgamated to a $k$-link, the disagreements partition the components into two classes, with all components in the same class agreeing and all components from different classes disagreeing on the orientation of their amalgamation. We twist, again at the appropriate cut pair, all components of one class. As before, choices are made from finite sets, and the twists can be done in parallel as long as the cut pairs are disjoint, which can be done measurably. We need to use countably many steps if there are $k$-links with unbounded $k$. Also, we need to use $2$-ended Whitney twists in case there are disagreements along a bi-infinite line of neighboring components in the Tutte-tree. 
\end{proof}

\subsection{Proof of the measurable Whitney $2$-isomorphism theorem}

We are now ready to prove Theorem~\ref{thm:Whitney_2-isom_intro}.

\begin{proof}[Proof of Theorem~\ref{thm:Whitney_2-isom_intro}]
    As cycles are preserved, $\varphi$ maps the $2$-connected components of $\cG_1$ to $2$-connected components of $\cG_2$. We will first perform vertex splits in both $\cG_1$ and $\cG_2$ to obtain graphings $\cG_1'$ and $\cG_2'$ whose $0$-, $1$- and $2$-ended components are strongly $2$-connected, and infinitely-ended components are weakly $2$-connected. The original weak isomorphism $\varphi$ induces a weak isomorphism $\varphi'$ from $\cG_1'$ to $\cG_2'$. 

    Indeed, first consider components with at least two ends. In each such component let $R$ be the double-ray core from Lemma~\ref{lem:double_ray_core}. The core is Borel, weakly $2$-connected, and has the same ends as the original component, while everything outside it consists of finite subgraphs attached to $R$ at single vertices. We split off these finite attachments, one finite component at a time at their attachment vertices. By local finiteness, these splits can be scheduled as a locally finite sequence of finite vertex splits, and the remaining infinite component is precisely $R$.

    For a $1$-ended component, every failure of weak $2$-connectivity likewise cuts off a finite component, so a locally finite sequence of finite vertex splits makes its remaining infinite component weakly $2$-connected. A $1$-ended component that is weakly $2$-connected is automatically strongly $2$-connected: a vertex cut would create at least two infinite components, contradicting $1$-endedness. 2-ended components might admit cut vertices, creating exactly 2 infinite components at each cut. In this case, however, we can perform a $2$-ended infinite split to partition the edge set into finite components. Therefore, we can ensure strong $2$-connectedness of the 2-ended components as well. In the end, all finite components can be split into their $2$-connected components. This finishes the construction of $\cG_1'$ and $\cG_2'$.

    It suffices to show that $\varphi'$ can be implemented by finitely many locally finite sequences of Whitney operations, as we can pre-compose them with the disassembly of $\cG_1$ into $\cG_1'$ and post-compose with the reverse operation of the disassembly of $\cG_2$ into $\cG_2'$. The latter itself can be implemented by finitely many locally finite sequences: reversing the split of finite components in the end is a finite join; reversing a 2-ended split is a 2-ended join; and for the reversal of the locally finite sequence of finite splits see Example~\ref{ex:1-ended_tree_transform}.

    By Theorem~\ref{thm:connected_preserving}, and Corollary~\ref{cor:preserve_no_ends}, $\varphi'$ preserves the connected components of $\cG_1'$ and $\cG_2'$, as well as the number of ends of components. Hence we can treat the infinitely-ended components separately from $0$-, $1$-, and $2$-ended ones. The former was done in Theorem~\ref{thm:infty_ended_whitney_twists}, and the latter was done in Theorem~\ref{thm:1-2-ended}.
\end{proof}

\section{Asymptotic Whitney rigidity for finite graphs}\label{sec:finite-applications}
In this section we investigate the consequences of our measurable rigidity results for sequences of finite graphs, and establish an asymptotic version of Whitney’s rigidity theorem. We also prove a couple of corollaries for 5-connected graphs, as well as for sequences with essentially large girth and minimum degree at least three. We start by a detailed, self-contained setup of the asymptotic notions. The proofs themselves pass to limit graphings via ultraproduct techniques, and apply the measurable rigidity theorem. 

Throughout this section, $G_n$ and $H_n$ denote finite graphs that are simple and have no isolated vertices. Graph sequences have uniformly bounded degree and satisfy $|V(G_n)|\to\infty$. We write $m_n:=|E(G_n)|$. Since the graphs have no isolated vertices and uniformly bounded degree, $m_n=\Theta(|V(G_n)|)$; in particular, the vertex and edge normalizations have the same negligible sequences.

The asymptotic viewpoint adopted in this section is closely related to
Elek's work on bounded-degree graph sequences \cite{elek2007combinatorial, elek2012finite}. In particular, hyperfiniteness of finite graph sequences and its
relation to amenability and graphing limits are developed in \cite{elek2012finite}.

\subsection{Almost acyclic edge sets}

\begin{defn}\label{def:finite-edge-properties}
Let $(G_n)$ be a graph sequence as above and let $F_n\subseteq E(G_n)$. We say that
\begin{enumerate}
    \item $(F_n)$ has \emph{vanishing nullity} if
    \[
        |F_n|-r_{G_n}(F_n)=o(m_n);
    \]

    \item $(F_n)$ is \emph{almost acyclic} if there are sets $A_n\subseteq F_n$ with $|A_n|=o(m_n)$ such that $(V(G_n),F_n\setminus A_n)$ is a forest;

    \item $(F_n)$ is \emph{almost a bounded-component forest} if, for every $\eps>0$, there is a $K=K(\eps)<\infty$ such that, for all sufficiently large $n$, one can delete a set $B_n\subseteq F_n$ with $|B_n|\leq \eps m_n$ so that $(V(G_n),F_n\setminus B_n)$ is a forest and every component has at most $K$ vertices;

    \item $(F_n)$ has \emph{essentially large girth} if there are sets $C_n\subseteq F_n$ with $|C_n|=o(m_n)$ such that
    \[
        \operatorname{girth}(V(G_n),F_n\setminus C_n)\longrightarrow\infty;
    \]

    \item $(F_n)$ is \emph{hyperfinite} if, for every $\eps>0$, there is a $K=K(\eps)<\infty$ such that, for all sufficiently large $n$, one can delete a set $D_n\subseteq F_n$ with $|D_n|\leq \eps m_n$ so that every component of $(V(G_n),F_n\setminus D_n)$ has at most $K$ vertices.
\end{enumerate}
Here and below, forests are regarded as having infinite girth. 

\end{defn}

\begin{rem}
The normalization in Definition~\ref{def:finite-edge-properties} is by the ambient size $m_n$, rather than by $|F_n|$. Under our standing assumptions, the edge-deletion formulations are equivalent to the corresponding vertex-deletion formulations.
\end{rem}

\begin{lem}\label{lem:almost_acyclic_characterizations}
Let $(G_n)$ be a graph sequence as above and let $F_n\subseteq E(G_n)$. The following are equivalent:
\begin{enumerate}
    \item\label{itm:rank_excess} $(F_n)$ has vanishing nullity;
    \item\label{itm:almost_acyclic} $(F_n)$ is almost acyclic;
    \item\label{itm:almost_bdd_forest} $(F_n)$ is almost a bounded-component forest;
    \item\label{itm:ess_large_girth_hyperfinite} $(F_n)$ has essentially large girth and is hyperfinite.
\end{enumerate}
\end{lem}

\begin{proof}
The quantity $|F_n|-r_{G_n}(F_n)$ is the minimum number of edges that must be deleted from $F_n$ to obtain a forest, which proves the equivalence of~\ref{itm:rank_excess} and~\ref{itm:almost_acyclic}. The implication from~\ref{itm:almost_bdd_forest} to \ref{itm:almost_acyclic} follows by a diagonalization argument. 
Conversely, bounded-degree forests form a hyperfinite family
\cite[Section~1.2]{elek2012finite}, so~\ref{itm:almost_acyclic} implies~\ref{itm:almost_bdd_forest}, and it also immediately implies~\ref{itm:ess_large_girth_hyperfinite}.

Finally, assume~\ref{itm:ess_large_girth_hyperfinite}. Given $\eps>0$, first delete at most $\eps m_n$ edges so that all remaining components have at most $K$ vertices. Then delete $o(m_n)$ additional edges so that the remaining graph has girth larger than $K$. Every remaining component is therefore a tree. This proves~\ref{itm:almost_bdd_forest}.
\end{proof}

\subsection{Almost rank-preserving maps}

Let $(G_n)$ and $(H_n)$ be graph sequences as above, and put $m_n:=|E(G_n)|$. We call a sequence of maps $f_n:E(G_n)\to E(H_n)$ an \emph{almost bijection} if there are sets $A_n\subseteq E(G_n)$ and $B_n\subseteq E(H_n)$ with $|A_n|=o(m_n)$ and $|B_n|=o(m_n)$ such that
\[
    f_n\big|_{E(G_n)\setminus A_n}:E(G_n)\setminus A_n
    \longrightarrow E(H_n)\setminus B_n
\]
is a bijection. In particular, $|E(H_n)|=m_n+o(m_n)$.

We say that $(f_n)$ is \emph{almost rank-preserving} if the rank error is uniformly sublinear over all edge sets, that is,
\begin{equation}\label{eq:uniform-rank-error}
    \sup_{S\subseteq E(G_n)}
    \bigl|r_{H_n}(f_n(S))-r_{G_n}(S)\bigr|
    =o(m_n).
\end{equation}
We say that $(f_n)$ preserves a property $\cP$ of edge-set sequences if, for every sequence $F_n\subseteq E(G_n)$, the sequence $(F_n)$ has $\cP$ if and only if $(f_n(F_n))$ has $\cP$.

We also say that $(f_n)$ almost preserves small cycles if for all $k\in\N$ the images and preimages of almost all $k$-cycles are $k$-cycles. More precisely, $(f_n)$ induces a map between $k$-element subsets of $E(G_n)$ and $E(H_n)$, and the restriction of this map to the $k$-cycles is an almost bijection (with $o(m_n)$ error). Equivalently, for each $k\in \N$, we can discard $o(m_n)$ edges from $E(G_n)$ and $E(H_n)$ such that $f_n$ induces a bijection between $\ell$-cycles of $G_n$ and $H_n$ for each $\ell \leq k$. (Indeed, we can assume $f_n$ is a bijection, and discard one edge and its $f_n$ image or preimage from any problematic $\ell$-cycle of $G_n$ or $H_n$.)

\begin{prop}\label{prop:almost_rank_pres_characterization}
Let $(G_n)$ and $(H_n)$ be graph sequences as above, and let $(f_n)$ be an almost bijection of their edge sets. The following are equivalent:
\begin{enumerate}
    \item\label{itm:rank_pres} $(f_n)$ is almost rank-preserving;
    \item\label{itm:almost_acyc_pres} $(f_n)$ preserves almost acyclicity;
    \item\label{itm:almost_bbd_forest_pres} $(f_n)$ preserves the property of being almost a bounded-component forest;
    \item\label{itm:hyp_plus_ess_large_girth_pres} $(f_n)$ preserves the property of being both hyperfinite and essentially large girth;
    \item\label{itm:hyp_pres_plus_large_girth_pres} $(f_n)$ preserves hyperfiniteness and, separately, preserves essentially large girth; 
    \item\label{itm:almost_small_cycle_plus_hyp} $(f_n)$ preserves hyperfiniteness and, separately, almost preserves small cycles;
    
    \item\label{itm:ultraproduct_weak_iso} for every nonprincipal ultrafilter $\cU$ on $\N$, the Loeb ultraproduct map $f_\cU: E(\mathbf G_{\cU})\to E(\mathbf H_{\cU})$ is a rank-preserving bijection modulo null sets. \end{enumerate}
We postpone the definitions implicit in~\ref{itm:ultraproduct_weak_iso}, and the proof of the equivalence between~\ref{itm:rank_pres} and~\ref{itm:ultraproduct_weak_iso} to Subsection~\ref{subsec:ultraproducts}.
\end{prop}
\begin{proof}

After deleting $o(m_n)$ edges from both sides, we may treat $f_n$ as a bijection. In particular,
\[
    \sup_{S\subseteq E(G_n)}\bigl||f_n(S)|-|S|\bigr|=o(m_n).
\]
It follows from equation~\eqref{eq:uniform-rank-error} that the nullities of $S$ and $f_n(S)$ differ by $o(m_n)$ uniformly in $S$. Hence~\ref{itm:rank_pres} implies~\ref{itm:almost_acyc_pres}.

Conversely, suppose~\ref{itm:rank_pres} fails. After passing to a subsequence, there are sets $S_n\subseteq E(G_n)$ and a constant $c>0$ such that the rank error is at least $cm_n$. If
$r_{H_n}(f_n(S_n))\leq r_{G_n}(S_n)-cm_n$, take a maximal forest $I_n\subseteq S_n$; then $(I_n)$ is acyclic, whereas $(f_n(I_n))$ has nullity at least $cm_n-o(m_n)$. If the reverse rank inequality holds, take a maximal forest in $f_n(S_n)$ and pull it back through the almost bijection; its preimage has nullity at least $cm_n-o(m_n)$. In either case, almost acyclicity is not preserved. Thus~\ref{itm:almost_acyc_pres} implies~\ref{itm:rank_pres}.

Lemma~\ref{lem:almost_acyclic_characterizations} now gives the equivalence of~\ref{itm:almost_acyc_pres}, \ref{itm:almost_bbd_forest_pres}, and~\ref{itm:hyp_plus_ess_large_girth_pres}. Clearly,~\ref{itm:hyp_pres_plus_large_girth_pres} implies~\ref{itm:hyp_plus_ess_large_girth_pres}. 

Next, we show that~\ref{itm:almost_small_cycle_plus_hyp}implies~\ref{itm:hyp_pres_plus_large_girth_pres}. We show that if $f_n$ almost preserves small cycles, then it preserves essentially large girth. Indeed, if $(f_n)$ almost preserves small cycles, by throwing away $o(m_n)$ many edges we can assume it preserves all $k$-cycles. By a diagonal argument, we can find a sequence $g(n)\to\infty$ such that after throwing away the $o(m_n)$ many edges, $f_n$ preserves the cycles up to length $g(n)$. This shows that $f_n$ preserves essentially large girth.

It remains to show that~\ref{itm:rank_pres} implies~\ref{itm:almost_small_cycle_plus_hyp}. We do that by establishing the following two lemmas. The first characterizes hyperfiniteness in terms of the rank function, and the second shows that preserving almost acyclicity implies almost preserving small cycles.

\begin{rem}
    We could prove the $\ref{itm:rank_pres}\Rightarrow \ref{itm:almost_small_cycle_plus_hyp}$ implication by taking ultraproducts and applying Lemma~\ref{lem:eq_of_rankpres} to standard factors, as we will do in the proof of Theorem~\ref{thm:finite-whitney-stability}. Instead, we decided to include a direct proof, because we believe the characterization of hyperfiniteness below is worth pointing out in itself. Also, the arguments are not too long, and this approach allows the proof of Proposition~\ref{prop:almost_rank_pres_characterization} to be constrained to this section.
\end{rem}

\begin{lem}\label{lem:hyperfinite-rank-partitions}
A sequence $(F_n)$ is hyperfinite if and only if, for every $\eps>0$, there are $K<\infty$, sets $Z_n\subseteq F_n$ with
\[
    |Z_n|\leq \eps m_n + o(m_n),
\]
and partitions
\[
    F_n\setminus Z_n=\bigsqcup_{i\in I_n}P_{n,i},
    \qquad 1\leq |P_{n,i}|\leq K,
\]
such that
\begin{equation}\label{eq:partition-rank-defect}
    \sum_{i\in I_n}r_{G_n}(P_{n,i})
    -r_{G_n}\left(\bigcup_{i\in I_n}P_{n,i}\right)
    =o(m_n).
\end{equation}
\end{lem}

\begin{rem}\label{rem:hyperfinite-rank-partitions}
Equivalently, one may require the uniform estimate
\begin{equation} \label{eqn:uniform_defect}
    \sup_{J\subseteq I_n}
    \left(
        \sum_{i\in J}r_{G_n}(P_{n,i})
        -r_{G_n}\left(\bigcup_{i\in J}P_{n,i}\right)
    \right)
    =o(m_n).
\end{equation}

Moreover, the condition is equivalent to the existence of such $K, Z_n$ and $P_{n,i}$ where even the $o(m_n)$ error in equation~(\ref{eq:partition-rank-defect}) vanishes, that is, for every $J \subseteq I_n$, we have
\begin{equation}\label{eqn:partition-rank-equality}
    \sum_{i\in J}r_{G_n}(P_{n,i})
    =r_{G_n}\left(\bigcup_{i\in J}P_{n,i}\right).
\end{equation}
\end{rem}

\begin{proof}[Proof of Remark~\ref{rem:hyperfinite-rank-partitions}]

For $J\subseteq I_n$, put
    \[
        d_n(J):=\sum_{i\in J}r_{G_n}(P_{n,i})
        -r_{G_n}\left(\bigcup_{i\in J}P_{n,i}\right).
    \]
Submodularity implies that $d_n(J)$  is monotone in $J$, which proves the stated equivalence with the uniform estimate~(\ref{eqn:uniform_defect}). The same idea also gives the stronger version~(\ref{eqn:partition-rank-equality}) with zero defect, once failures of additivity are absorbed into $Z_n$. 

More precisely, label $I_n=\{1,\ldots,\ell_n\}$, let $U_{n,j}:=\bigcup_{i=1}^{j}P_{n,i}$, and consider the increments in the failures of additivity of $r_{G_n}$ over the blocks:
    \[
        c_{n,j}:=r_{G_n}(P_{n,j})+r_{G_n}(U_{n,j-1})-r_{G_n}(U_{n,j}).
    \]
These are nonnegative integers, and telescoping gives $ \sum_{j=1}^{\ell_n}c_{n,j}=d_n(I_n)=o(m_n)$. Therefore the set of indices where some positive failure occurs, that is $B_n=\{j:c_{n,j}>0\}$, has size $o(m_n)$. We set
    \[
        Z'_n:=Z_n\cup\bigcup_{j\in B_n}P_{n,j}
        \quad\text{and}\quad I'_n:=I_n\setminus B_n.
    \]
Since every block has at most $K$ edges, $|Z'_n|\leq \eps m_n+o(m_n)$. Now fix $J \subseteq I'_n$ and some $j\in J$. Let $W_{n,j}$ be the union of the retained blocks with indices in $J$ preceding $P_{n,j}$. Then $W_{n,j}\subseteq U_{n,j-1}$, while $c_{n,j}=0$. By submodularity, the rank gain from adding $P_{n,j}$ to $W_{n,j}$ is at least the rank gain from
    adding it to $U_{n,j-1}$, which is $r_{G_n}(P_{n,j})$; the reverse
    inequality is automatic. Thus each retained block contributes its full
    rank. Telescoping once more yields
    \[
        r_{G_n}\left(\bigcup_{j\in J}P_{n,j}\right)
        =\sum_{j\in J}r_{G_n}(P_{n,j}),
    \]
    as required.
\end{proof}

\begin{proof}[Proof of Lemma~\ref{lem:hyperfinite-rank-partitions}]
Suppose first that $(F_n)$ is hyperfinite. After deleting at most $\eps m_n$ edges, all components have uniformly bounded size. Taking the $P_{n,i}$ to be their edge sets, the rank is additive over arbitrary unions of the blocks, so the defect in~\eqref{eq:partition-rank-defect} is zero.

Conversely, suppose that such partitions exist. By Remark~\ref{rem:hyperfinite-rank-partitions}, we may also assume that for every $J\subseteq I_n$, we have   \[ \sum_{i\in J}r_{G_n}(P_{n,i})
    =r_{G_n}\left(\bigcup_{i\in J}P_{n,i}\right).\]

Also, refining every block into its connected components preserves the bound on its size and does not change the rank-additivity. We may therefore assume that all $P_{n,i}$ are connected. Let $\cB_n$ be the bipartite incidence graph whose left vertex set is $I_n$, whose right vertex set is $V(F_n\setminus Z_n)$, and in which $i$ is adjacent to $v$ precisely when $v\in V(P_{n,i})$. The graphs $\cB_n$ have uniformly bounded degree, and their components correspond to the components of $(V(G_n),F_n\setminus Z_n)$. First, we claim the graph $\cB_n$ is a forest. Indeed, assume the opposite, and let $C=(v_1,P_1,v_2,P_2,\dots P_\ell)$ be  a simple cycle. Let $H_k\subseteq P_k$ be spanning trees. The rank-additivity condition implies that $H_1\cup\dots\cup H_\ell$ is also a forest. However, the $C$ in $\cB_n$ defines a cycle in $G_n$ through $v_1,v_2,\dots v_\ell$, using edges from $H_k$ to pass from $v_{k}$ to $v_{k+1}$, which contradicts acyclicity of the union.

The bounded-degree forests $\cB_n$ form a hyperfinite family. Therefore, for every $\delta>0$, one can delete $\delta m_n$ incidence edges so that all remaining components of $\cB_n$ have uniformly bounded size. Pull this deletion back by deleting the whole block corresponding to the left endpoint of every removed incidence edge. Since every block has at most $K$ edges, this deletes at most $K\delta m_n$ edges of $F_n\setminus Z_n$. Choosing the parameters sufficiently small proves that $(F_n)$ is hyperfinite. 
\end{proof}

\begin{lem}\label{lem:almost_acyclic_preserves_small_cycles}
Almost acyclicity-preserving implies almost small cycle preserving.    
\end{lem}
\begin{proof}

Towards contradiction, assume that $f_n$ does not almost preserve small cycles. Take the smallest $k$ such that $f_n$ does not induce an almost bijection on $k$-cycles. By throwing away $o(m_n)$ edges from both sides, we may assume that $f_n$ is a bijection and preserves all cycles of length less than $k$ in both directions.

After passing to a subsequence and, if necessary, replacing $f_n$ by $f_n^{-1}$, there is a $c>0$ such that every edge set meeting all $k$-cycles of $G_n$ whose images are not $k$-cycles has size at least $cm_n$. Let $A_n$ be a maximal edge-disjoint family of such bad $k$-cycles. By maximality, the union of $A_n$ meets every bad $k$-cycle, and hence $k|A_n|\geq cm_n$. Thus $|A_n|=\Omega(m_n)$.

As a cycle does not contain a smaller cycle, the image $f_n(C)$ is acyclic for every $C\in A_n$. Now put a graph on $A_n$: cycles $C_1$ and $C_2$ are adjacent if either $C_1$ and $C_2$, or $f_n(C_1)$ and $f_n(C_2)$, share a vertex. Since the cycles in $A_n$ and their images are edge-disjoint and the graph sequences have bounded degree, this graph has uniformly bounded degree. It therefore has an independent set $I_n\subseteq A_n$ of size $|I_n|=\Omega(m_n)$.

The sets $f_n(C)$, $C\in I_n$, are vertex-disjoint and acyclic, so $f_n\left(\bigcup_{C\in I_n}C\right)$ is acyclic. As $(f_n)$ preserves almost acyclicity, this implies that
$\bigcup_{C\in I_n}C$ is almost acyclic. This is a contradiction, since it contains $\Omega(m_n)$ pairwise edge-disjoint $k$-cycles.
\end{proof}

To complete the proof of Proposition~\ref{prop:almost_rank_pres_characterization}, suppose that $(f_n)$ is almost rank-preserving, and let $(F_n)$ be hyperfinite. Fix $\varepsilon>0$, and let $Z_n$ and $P_{n,i}$, $i\in I_n$, be as in Lemma~\ref{lem:hyperfinite-rank-partitions}, with $|P_{n,i}|\leq K$. By~\ref{itm:rank_pres} $\Rightarrow$ \ref{itm:almost_acyc_pres} and Lemma~\ref{lem:almost_acyclic_preserves_small_cycles}, after removing $o(m_n)$ edges, $f_n$ preserves all cycles of length at most $K$ in both directions. Discarding the blocks meeting these edges removes only $o(m_n)$ further edges. If $J_n\subseteq I_n$ is the set of remaining blocks, then
\[
r_{H_n}(f_n(P_{n,i}))=r_{G_n}(P_{n,i}),\qquad i\in J_n,
\]
since every cycle contained in either block has length at most $K$. Hence
\[
\sum_{i\in J_n} r_{H_n}(f_n(P_{n,i}))-r_{H_n}\left(f_n\left(\bigcup_{i\in J_n}P_{n,i}\right)\right)=\sum_{i\in J_n} r_{G_n}(P_{n,i})-r_{G_n}\left(\bigcup_{i\in J_n}P_{n,i}\right)+o(m_n)=o(m_n),
\]
where the last equality follows from Remark~\ref{rem:hyperfinite-rank-partitions}. Thus Lemma~\ref{lem:hyperfinite-rank-partitions} implies that $(f_n(F_n))$ is hyperfinite. Applying the same argument to the inverse almost bijection gives the converse.\end{proof}

\subsection{Almost isomorphisms, asymptotic connectivity, and ends}

In order to state our asymptotic result, we need asymptotic versions of the notions in Theorem~\ref{thm:graph_isom}. We start with almost isomorphisms.

Let $(G_n)$ and $(H_n)$ be bounded-degree graph sequences. A sequence of maps $g_n:V(G_n)\to V(H_n)$ is a \emph{vertex almost-bijection} if there are sets $U_n\subseteq V(G_n)$ and $W_n\subseteq V(H_n)$ with
\[
    |V(G_n)\setminus U_n|=o(|V(G_n)|),
    \qquad
    |V(H_n)\setminus W_n|=o(|V(H_n)|),
\]
such that $g_n|_{U_n}:U_n\to W_n$ is a bijection. It is an \emph{almost isomorphism} if $g_n|_{U_n}$ is an isomorphism from $G_n[U_n]$ onto $H_n[W_n]$. This is equivalent to saying that there is an almost bijection between the vertex sets that is an isomorphism on the 1-neighborhood of a uniform random vertex with probability tending to 1.

We also give finite formulations of the two hypotheses in Theorem~\ref{thm:graph_isom}, namely weak 3-connectivity and infinitely many ends. For $v\in V(G)$ and $R\geq 1$, let $B_R^{\mathrm w}(G,v)$ be obtained from the induced ball $G[B_R(v)]$ by identifying the sphere $S_R(v)$ to a single boundary vertex $\partial_R$. Call $v$ \emph{$R$-weakly $3$-connected} if no set of at most two vertices in $B_{R-1}(v)\setminus\{v\}$ separates $v$ from $\partial_R$ in $B_R^{\mathrm w}(G,v)$. (If $S_R(v)=\varnothing$, $v$ is not considered $R$-weakly $3$-connected.) Let $a_n(R)$ be the proportion of $R$-weakly $3$-connected vertices of $G_n$. A bounded-degree graph sequence $(G_n)$ is \emph{asymptotically weakly $3$-connected} if $a_n(R)\to 1$ for every fixed $R$.

Similarly, for $0\leq r<R$, call $v$ \emph{$(r,R)$-trifurcating} if there is a connected vertex set $K\subseteq B_r(v)$ containing $v$ such that at least three components of $B_R(v)\setminus K$ meet $S_R(v)$. Let $t_n(r,R)$ denote the proportion of $(r,R)$-trifurcating vertices. $(G_n)$ is \emph{asymptotically infinitely-ended} if for any $\varepsilon > 0$, for all sufficiently large $r$, and any fixed $R>r$ we have $t_n(r,R) > 1- \varepsilon$ for all $n$ large enough. That is, \[\lim_{r\to\infty}\ \inf_{R>r}\ \liminf_{n\to\infty}t_n(r,R)=1.\]

Our proof of Theorem~\ref{thm:finite-whitney-stability} uses standard techniques from non-standard analysis: we will construct the ultraproducts of the finite graph sequences to get measure preserving graphs. Our ``almost'' and ``asymptotic'' assumptions yield all the assumptions of Theorem~\ref{thm:graph_isom} for the ultraproducts. We then transfer the isomorphism in the limit back to the sequences. The only issue is that the vertex set of the ultraproduct graph is not a standard probability space, so we cannot apply Theorem~\ref{thm:graph_isom} directly. The usual way to get around this is to consider standard factors of the ultraproduct. We also take this approach, and the main difficulty is to produce standard factors of both ultraproducts simultaneously, and such that the edge-bijection between the ultraproducts factors through to an edge-bijection of the factors. This technical step is isolated into Proposition~\ref{prop:standard-factor-witness} below.

\subsection{Ultraproducts and standard factors}\label{subsec:ultraproducts}

Fix a nonprincipal ultrafilter $\cU$ on $\N$. For a graph sequence $(G_n)$ as above, write $\mathbf G_{\cU}$ for its ultraproduct. Its vertices are the equivalence classes $[x_n]_{\cU}$ of sequences $x_n\in V(G_n)$, and $[x_n]_{\cU}$ and $[y_n]_{\cU}$ are adjacent if $x_ny_n\in E(G_n)$ for $\cU$-almost every $n$. Sets of the form $[A_n]_{\cU}$ are called \emph{internal}. The internal vertex and edge measures are given by
\[
    \mu_{\cU}([A_n]_{\cU}):=\lim_{\cU}\frac{|A_n|}{m_n}, \qquad\eta_{\cU}([F_n]_{\cU}):=\lim_{\cU}\frac{|F_n|}{m_n}.
\]

Internal sets form a set-algebra, but not a $\sigma$-algebra. In order to get a $\sigma$-algebra, one has to take the so-called Loeb completion, that is, all sets that are internal up to symmetric difference with a null set. The graph  $\mathbf G_{\cU}$ is not a graphing in the sense of Subsection~\ref{subsec:graphings}, because the underlying Loeb space need not be standard. We refer to \cite[Section 2.7]{elek2012measure-theoretic}, \cite{conley2013ultraproducts}, and \cite{carderi2015ultraproducts} for introductions to ultraproducts of finite probability spaces and their associated Loeb measure spaces, particularly in combinatorial and measure-preserving settings. See also \cite{carderi2021non-standard} for a more recent treatment of non-standard limits in the context of graphs and measured equivalence relations. 

An important observation is that each $E(G_n)$ can be decomposed into a uniformly bounded number of oriented matchings, and the corresponding ultraproduct partial bijections generate the edges of $\mathbf{G}_{\cU}$. Also, the definition of the rank function $\rho_{\mathbf G_{\cU}}$ does not rely on the space being standard.

It is straightforward to check that for an internal edge set $F=[F_n]_{\cU}$, the rank satisfies
\begin{equation}\label{eq:ultraproduct-rank}
    \rho_{\mathbf G_{\cU}}(F)
    =\lim_{\cU}\frac{r_{G_n}(F_n)}{m_n}.
\end{equation}
Internal sets are dense in the Loeb edge algebra, and the rank is $1$-Lipschitz with respect to the edge measure, so~\eqref{eq:ultraproduct-rank} determines the rank of every Loeb measurable edge set. The following lemma is straightforward to check.

\begin{lem}\label{lem:ultraproduct-connectivity}
Let $\mathbf G_{\cU}$ be the Loeb ultraproduct of $(G_n)$.
\begin{enumerate}
    \item Almost every component of $\mathbf G_{\cU}$ is weakly $3$-connected if and only if $\lim_{\cU}a_n(R)=1$ for every fixed $R$.
    \item Almost every component of $\mathbf G_{\cU}$ is infinitely-ended if and only if
    \[
        \lim_{r\to\infty}\ \inf_{R>r}\ \lim_{n \to \cU}t_n(r,R)=1.
    \]
\end{enumerate}
Consequently, $(G_n)$ is asymptotically weakly $3$-connected if and only if almost every component of every Loeb ultraproduct is weakly $3$-connected. If $(G_n)$ is asymptotically infinitely-ended, then almost every component of every Loeb ultraproduct is infinitely-ended; the converse holds when $(G_n)$ is locally convergent.
\end{lem}

As stated before, ultraproducts also capture almost rank-preserving edge maps.

\begin{lem} \label{lem:almost_rank_pres_eqv_ultra_weak_iso}
    In the setting of Proposition~\ref{prop:almost_rank_pres_characterization}, $(f_n)$ is almost rank-preserving if and only if $f_{\cU}$ is a rank-preserving bijection for any nonprincipal ultrafilter $\cU$.
\end{lem}
\begin{proof}
    For the forward direction, since the $f_n$ are almost bijections, their ultraproduct $f_{\cU}$ defined by $f_{\cU}([x_n]_{\cU})=[f_n(x_n)]_{\cU}$  is an edge-measure-preserving measurable bijection modulo null sets. By~\eqref{eq:ultraproduct-rank} and~\eqref{eq:uniform-rank-error}, $f_{\cU}$ preserves the rank of every internal edge set. Approximation by internal sets then gives rank preservation for every Loeb measurable edge set.

    For the reverse direction, assume toward contradiction that the $(f_n)$ are not almost rank-preserving. We then have 
    
    \[
    \left(\sup_{S \subseteq E(G_n)} \big|r_{H_n}(f_n(S)) - r_{G_n}(S)\big|\right) \bigg/ {m_n} \nrightarrow 0.\]
    We can pick an infinite $I \subseteq \N$ and $S_n \subseteq E(G_n)$ such that for all $n \in I$ we have 
    $|r_{H_n}(f_n(S_n)) - r_{G_n}(S_n)| \geq \varepsilon m_n$. Setting $\cU$ to be a nonprincipal ultrafilter containing $I$, $[S_n]_{\cU}$ witnesses that $f_{\cU}$ is not rank-preserving, a contradiction. 
\end{proof}

We say that an edge bijection between $\mathbf{G}$ and $\mathbf{H}$ \emph{preserves incidence modulo null sets} if, after discarding null edge sets on both sides, two edges are incident if and only if their images are incident. A graphing factor $q:\mathbf G\to\cG'$ is a \emph{local isomorphism} if, outside an invariant null set, its restriction to every component is a graph isomorphism onto a component of $\cG'$.

\begin{prop}[A standard factor witnessing an incidence defect]\label{prop:standard-factor-witness}
Let $\mathbf G$ and $\mathbf H$ be bounded-degree measure preserving graphs obtained as ultraproducts of finite graphs, and let $\eta:E(\mathbf G)\to E(\mathbf H)$ be the Loeb extension of an internal edge bijection between conull internal sets. Assume that $\eta$ preserves the edge measure and the rank. If $\eta$ does not preserve incidence modulo null sets, then there are standard graphings $\cG$ and $\cH$, factor maps $q_G:\mathbf G\to\cG$ and $q_H: \mathbf H \to \cH$ that are local isomorphisms, and an edge-measure- and rank-preserving Borel bijection $\eta':E(\cG)\to E(\cH)$ such that $q_H^E\circ\eta=\eta'\circ q_G^E$ a.e.\ and $\eta'$ does not preserve incidence modulo null sets.
\[
\begin{tikzcd}[column sep=large,row sep=large]
E(\mathbf{G}) \arrow[r,"\eta"]
     \arrow[d,"q_G^E"']
&
E(\mathbf{H}) \arrow[d,"q_H^E"]
\\
E(\cG) \arrow[r,"\eta'"']
&
E(\cH)
\end{tikzcd}
\]
\end{prop}

\begin{proof} We follow the strategy from \cite{conley2013ultraproducts}  and
\cite[Section 2.2]{carderi2015ultraproducts}, to which one can refer for more details. As we have noted above, the edges of $\mathbf G$ and $\mathbf H$ are generated by finitely many measure-preserving partial bijections. Factors of $\mathbf G$ and $\mathbf H$ correspond to sub-$\sigma$-algebras of their vertex sets that are invariant under these bijections. If $\eta$ does not preserve incidence of edges, we choose a positive measure set  of wedges witnessing this. That is, we have a positive measure set of edges $Z \subseteq E(\mathbf{G})$, and $\theta: Z \to \mathrm{ran}(\theta)$ an edge-measure-preserving bijection such that for almost every $e \in Z$, $\theta(e)$ is incident to $e$, but $\eta(e)$ is not incident to $\eta(\theta(e))$. By passing to a positive measure subset, we can assume that the endvertices of edges in $\eta(Z)$ and $\eta(\theta(Z))$ are distinct. (Indeed, using the fact that $\mathbf{H}$ is generated by a finite family of partial bijections, we can make the images of the wedges sparse in $\mathbf{H}$ as we have done before.) We then write $C_1,\dots, C_4 \subseteq V(\mathbf H)$ for the endvertex sets, and note that by passing to a positive measure subset of $Z$, we can assume that they are disjoint. Hence, they separate the endpoints of these edges. Consequently, when building the factor corresponding to a sub-$\sigma$-algebra that contains these sets, the images of $\eta(Z)$ and $\eta(\theta(Z))$ will not become incident in the factors. 

We also add all domains and ranges of the bijections generating $\mathbf{G}$ and $\mathbf{H}$, and for each finite $r$, the color classes of a measurable coloring of the vertices where monochromatic vertices are at distance greater than $r$. (Such colorings can be constructed by taking the ultraproduct of colorings of the finite graphs.) These colorings ensure that vertices in the same components are separated by a chosen set, so the factors will be isomorphisms when restricted to components. We then extend these to invariant $\sigma$-algebras on both vertex sets via a back-and-forth procedure as follows. 

Let $ \mathcal A_{\mathbf G}^0\subseteq\mathcal B_{\mathbf G}, \ \mathcal A_{\mathbf H}^0\subseteq\mathcal B_{\mathbf H}$ denote the $\sigma$-algebras generated by the sets above, closed under the graph-generating partial bijections and their inverses. Let $s_{G}$, $t_{G}$, $s_H$, and $t_H$ denote the endvertex maps on edges of $\mathbf{G}$ and $\mathbf{H}$. For a vertex $\sigma$-algebra \(\mathcal X_G\), write \[ \mathcal E_G(\mathcal X_G) := \sigma\bigl( s_G^{-1}(\mathcal X_G) \cup t_G^{-1}(\mathcal X_G) \bigr) \] for the associated edge $\sigma$-algebra, and define \(\mathcal E_H(\mathcal X_H)\) analogously. Construct increasing sequences \[ \mathcal A_G^0\subseteq\mathcal A_G^1\subseteq\cdots, \qquad \mathcal A_H^0\subseteq\mathcal A_H^1\subseteq\cdots \] as follows. Given \(\mathcal A_H^n\), take a countable generating family for $\mathcal E_H(\mathcal A_H^n)$, and pull it back through \(\eta\). Intersect each pulled-back edge set with the finitely many partial bijections generating $E_{\mathbf G}$, and consider their endvertices via $s_{G}$ and $t_G$. Adjoin all such endvertex sets to \(\mathcal A_{\mathbf G}^n\), and close under the graph generators, obtaining \(\mathcal A_{\mathbf G}^{n+1}\). Then take a countable generating family for $ \mathcal E_G(\mathcal A_G^{n+1})$, push it forward through \(\eta\), adjoin the endvertex sets of the intersections with each partial bijection, and close invariantly to obtain \(\mathcal A_H^{n+1}\). Set \[ \mathcal A_G = \sigma\left(\bigcup_{n\geq0}\mathcal A_G^n\right), \qquad \mathcal A_H = \sigma\left(\bigcup_{n\geq0}\mathcal A_H^n\right). \] These $\sigma$-algebras are countably generated and invariant, and satisfy 
\begin{equation}  \label{eqn:edge_algebras_coincide}  
\eta^{-1}\bigl(\mathcal E_H(\mathcal A_H)\bigr) = \mathcal E_G(\mathcal A_G) \quad \textrm{modulo null sets.} 
\end{equation}
The countably generated measure algebras \(\mathcal A_G\) and \(\mathcal A_H\) admit standard  realizations \cite{conley2013ultraproducts, carderi2015ultraproducts}. Let $q_G:\mathbf G \to \cG, \ q_H:\mathbf H\to \cH$ be the associated factor maps. 

Because the color classes were included, the maps \(q_G\) and \(q_H\) restrict to graph isomorphisms on every connected component.  Equation~\eqref{eqn:edge_algebras_coincide} yields, after deleting null sets, a Borel edge-measure-preserving bijection $ \eta':E(\cG)\longrightarrow E(\cH)$ satisfying $ q_H^E\circ\eta = \eta'\circ q_G^E$ almost everywhere. Since the factors are local isomorphisms, the ranks of the factors are the push-forwards of the original ranks, and so \(\eta'\) is rank-preserving. Finally, \(Z\) descends to a positive-measure set \(Z'\) witnessing the fact that $\eta'$ does not preserve incidence.
\end{proof}

\subsection{Proof of asymptotic Whitney rigidity}

\begin{proof}[Proof of Theorem~\ref{thm:finite-whitney-stability}]
Suppose the conclusion fails. Passing to a subsequence, we may assume that we have to delete a positive proportion of vertices of $G_n$ and $H_n$ in order to find an isomorphism $g_n$ between them that induces $f_n$ on the remaining edges. Fix a nonprincipal ultrafilter $\cU$ on this subsequence and form the Loeb ultraproducts $\mathbf G$ and $\mathbf H$, normalized by $m_n=|E(G_n)|$. Set $\eta = f_{\cU}$; it is rank-preserving by Lemma~\ref{lem:almost_rank_pres_eqv_ultra_weak_iso}.

By Lemma~\ref{lem:ultraproduct-connectivity}, almost every component of $\mathbf G$ is weakly $3$-connected and infinitely-ended. If $\eta$ failed to preserve incidence, Proposition~\ref{prop:standard-factor-witness} would give standard component-faithful factors $\cG_1'$ and $\cG_2'$ and a rank-preserving edge bijection $\eta':E(\cG_1')\to E(\cG_2')$ that still failed to preserve incidence. The source factor is again weakly $3$-connected and infinitely-ended. By Theorem~\ref{thm:graph_isom}, the map $\eta'$ is induced by an isomorphism of graphings, a contradiction. Thus $\eta$ preserves incidence modulo null sets.

After deleting invariant null sets, $\eta$ is therefore an isomorphism between the line graphs of corresponding components.  We can therefore reconstruct a vertex map exactly as in Claim~\ref{lem:w_to_w}: since $\mathbf{G}$ is weakly $3$-connected, every vertex is the center of a wedge, and incidence preservation implies that the star of every vertex is mapped to the star of a unique vertex. Thus $\eta$ is induced componentwise by a vertex bijection $h:V(\mathbf G)\to V(\mathbf H)$. The map $h$ is measurable, as $h(x)$ is the unique common endpoint of the images of the edges incident with $x$. Since $\eta$ is an edge bijection and there are no isolated vertices, these componentwise maps combine to a bijection between conull vertex sets.

Apply the same finite rule before taking the ultraproduct: define $g_n(x)$ to be the unique common endpoint of the family $\{f_n(e):e\ni x\}$ whenever such an endpoint exists, and define it arbitrarily otherwise. The ultraproduct of the maps $g_n$ agrees with $h$ a.e. Hence the sets of collisions and omitted target vertices have Loeb measure zero, so $(g_n)$ is a vertex almost isomorphism along $\cU$. 
This contradicts the choice of the subsequence.
\end{proof}

\begin{proof}[Proof of Corollary~\ref{cor:min_3_deg_large_girth}]
As $G_n$ has essentially large girth, any subsequential local limit is a tree with minimum degree 3. Therefore $G_n$ is asymptotically weakly $3$-connected and asymptotically infinitely-ended, so Theorem~\ref{thm:finite-whitney-stability} applies.
\end{proof}

Next we turn to the proof of Corollary~\ref{cor:cut_quadruples}. The idea is straightforward: we exploit $\sup_n \operatorname{cq}(G_n) < \infty $ to show that in any ultraproduct almost every $2$-ended component is strongly $3$-connected. We then argue as in the proof of Theorem~\ref{thm:finite-whitney-stability}, but use the more general graphing rigidity result, Theorem~\ref{thm:graph_isom_general}. Doing this precisely requires some effort, so we isolate certain steps into separate lemmas. First, we establish asymptotic weak $3$-connectivity.

\begin{lem}\label{lem:bounded_r_implies_weak_3-conn}    
    Let \((G_n)\) be a bounded-degree sequence of finite connected graphs with \(|V(G_n)|\to\infty\). If $\sup_n \operatorname{cq}(G_n)<\infty$, then \((G_n)\) is asymptotically weakly \(3\)-connected.
\end{lem}

\begin{proof}
    Suppose not. Then for some fixed \(R\) and some \(\varepsilon>0\), along a subsequence at least \(\varepsilon |V(G_n)|\) vertices are not \(R\)-weakly \(3\)-connected.

    For \(n\) large enough, every vertex has nonempty \(R\)-sphere. Hence, for every such bad vertex \(v\), there is a set \(S_v\subseteq B_{R-1}(v)\setminus\{v\}\) of size \(1\) or \(2\) which separates \(v\) from \(S_R(v)\). Thus \(S_v\) cuts off from \(G_n\) a nonempty component contained in \(B_{R-1}(v)\). If \(|S_v|=1\), add to it an arbitrary vertex of \(S_R(v)\), obtaining in either case a \(2\)-set $C_v\subseteq B_R(v)$, which itself is a cutset of \(G_n\).

    By bounded degree, from the linearly many bad vertices we may choose a linearly sized \(2R\)-sparse set of centers. For these centers the sets \(C_v\) are pairwise disjoint. Since each \(C_v\) already separates \(G_n\), so does \(C_v\cup C_w\) for every two distinct chosen centers \(v,w\). Thus the corresponding \(C_v\)'s form a linearly sized family admissible in the definition of \(\operatorname{cq}(G_n)\), contradicting \(\sup_n \operatorname{cq}(G_n)<\infty\). 
\end{proof}

Next, in preparation for the $2$-connected case, we state an easy graph-theoretic lemma.

\begin{lem}\label{lem:nested-cut-pairs}
Let $G$ be a connected, locally finite, strongly $2$-connected, $2$-ended graph. Suppose $\mathcal C$ is an infinite family of pairwise disjoint cut pairs, each separating the two ends of $G$. Then $\mathcal C$ contains an infinite subfamily of pairwise nested cuts.
\end{lem}

\begin{proof}
For $C=\{x,y\}\in\mathcal C$, we say that $D=\{u,v\}\in\mathcal C$ crosses $C$ if $u$ and $v$ lie on the two different infinite sides of $G-C$. If $D$ crosses $C$, then $D$ also separates $x$ from $y$. Indeed, by strong $2$-connectivity, within each infinite component of $G-C$, there exists an infinite path from either $x$ or $y$ towards the appropriate end avoiding $u$ and $v$ (we only have to avoid one, as the other is in a different component). If $D$ did not separate $x$ and $y$, then we would also have an $x$--$y$ path in $G-D$, and so the two ends would be connected in $G-D$. 

Hence every $D$ crossing $C$ meets any fixed finite $x$--$y$ path. Since the members of $\mathcal C$ are pairwise disjoint, only finitely many of them cross $C$. Thus the crossing graph on $\mathcal C$ is locally finite, so it has an infinite independent set; its members are pairwise nested.
\end{proof}

\begin{proof}[Proof of Corollary~\ref{cor:cut_quadruples}]
Suppose the conclusion fails and pass to a subsequence on which the defect is bounded away from zero. Fix a nonprincipal ultrafilter $\mathcal U$ on this subsequence, and let $\mathbf G$ be the Loeb ultraproduct of $(G_n)$. Lemma~\ref{lem:bounded_r_implies_weak_3-conn} implies that $(G_n)$ is asymptotically weakly $3$-connected, and consequently Lemma~\ref{lem:ultraproduct-connectivity} implies that $\mathbf G$ is weakly $3$-connected.

We will prove that almost every $2$-ended component of $\mathbf G$ is strongly $3$-connected. First, almost every $2$-ended component has no cut vertex: two distinct cut vertices in the same component would cut off a finite nonempty subgraph, which would give a cut pair in $G_n$ for $\mathcal U$-many $n$; while a unique cut vertex in a positive-measure family of infinite components is ruled out by the mass transport principle. Hence almost every $2$-ended component is strongly $2$-connected.

Suppose strong $3$-connectivity fails. That is, on a positive-measure set of bad components there are cut pairs separating the two ends. In almost every such component there are infinitely many such cut pairs. Otherwise the union of all of them would give a nonempty finite canonical subset of the component, again contradicting the mass transport principle. Moreover, every vertex belongs to only finitely many cut pairs. Indeed, if $\{v,w\}$ is a cut, then in $[v]_{\mathbf{G}}-v$ the vertex $w$ separates two neighbors of $v$ lying on opposite sides. For each fixed pair of neighbors of $v$, fix a finite path joining them. Then $w$ must lie on the union of these paths, so $v$ belongs to finitely many cut pairs. Hence one can choose infinitely many pairwise disjoint cut pairs.

By Lemma~\ref{lem:nested-cut-pairs}, we can pass to a nested subfamily and write them $C_1,C_2,\ldots$ in their order between the two ends. For $i<j$, the region between $C_i$ and $C_j$ is finite and nonempty; otherwise it would contain a ray and produce a third end. Hence $C_i\cup C_j$ cuts off a finite component.

Fix $k$. Represent $C_1,\ldots,C_k$ by pairs $C_1^{(n)},\ldots,C_k^{(n)}\subseteq V(G_n)$. Since, for every $i<j$, the fact that $C_i\cup C_j$ cuts off a finite component is a finite condition, for $\mathcal U$-many $n$ the pairs $C_1^{(n)},\ldots,C_k^{(n)}$ are disjoint and every $C_i^{(n)}\cup C_j^{(n)}$ is a cutset. Thus $\operatorname{cq}(G_n)\ge k$ for $\mathcal U$-many $n$. As $k$ is arbitrary, $\operatorname{cq}(G_n)$ is unbounded along a subsequence, a contradiction. This proves that $2$-ended components of $\mathbf{G}$ are strongly $3$-connected.

The rest of the proof is the same as in Theorem~\ref{thm:finite-whitney-stability}, replacing Theorem~\ref{thm:graph_isom} with Theorem~\ref{thm:graph_isom_general}. Finally, if the $G_n$ are 5-connected, there are no cut quadruples in $G_n$, so $\operatorname{cq}(G_n)=1$ for all $n$.
\end{proof}

Finally, we show that the $5$-connectivity condition in Corollary~\ref{cor:cut_quadruples} is sharp.

\begin{ex}[A $4$-connected finite obstruction] \label{ex:finite_4-conn}
For \(n\geq 5\), introduce layers $L_i=\{v_i^0,v_i^1\}$, where $i\in\mathbb Z/n\mathbb Z$, and let \(G_n\) be the 4-regular graph in which every vertex of \(L_i\) is adjacent to every vertex of \(L_{i+1}\), with no other edges; see Figure~\ref{fig:finite_4-conn}. ($G_n$ is the Cayley graph of $(\Z/n\Z) \times (\Z/2\Z)$ with the generators $\{(1,0),(1,1)\}$.) $G_n$ is \(4\)-connected, because deleting at most three vertices can empty at most one layer, after which the remaining nonempty layers form a path or a cycle joined by complete bipartite graphs. Set \(H_n=G_n\), and define the edge permutation which swaps the $\{0,1\}$ coordinate in one endpoint of each edge, that is

\[
f_n(v_i^av_{i+1}^b):=v_i^{1-a}v_{i+1}^b.
\]

We claim that $f_n$ has uniformly bounded rank error (in particular, it is almost rank-preserving), yet it is not induced by an almost isomorphism. Let \(D_n=E(L_{n-1},L_0)\). After deleting \(D_n\), the graph is a chain of two-vertex layers, and the restriction of \(f_n\) is obtained by successive Whitney twists at the \(L_i\). It therefore preserves the cycle-matroid rank exactly. Since \(f_n(D_n)=D_n\) and \(|D_n|=4\), for every \(S\subseteq E(G_n)\) we have
$\bigl|r_{G_n}(f_n(S))-r_{G_n}(S)\bigr|\leq 4$. 

Nevertheless, \(f_n\) cannot be induced by an almost isomorphism. At \(v_i^a\), the two incident edges coming from \(L_{i-1}\) are mapped to edges incident with \(v_i^a\), whereas the two edges going to \(L_{i+1}\) are mapped to edges incident with \(v_i^{1-a}\). Hence no three incident edges have images with a common endpoint. Thus any induced subgraph on which \(f_n\) is induced by a vertex map has maximum degree at most \(2\). Writing $N=|V(G_n)|$, if $q$ vertices are deleted, the remaining induced graph has at least $2N-4q$ edges, while a graph of maximum degree $2$ on $N-q$ vertices has at most $N-q$ edges. Therefore $2N-4q\leq N-q$, and hence $q\geq N/3$. Thus a positive proportion of the vertices must be deleted.
\end{ex}

\begin{figure}[ht!]
\centering
\begin{tikzpicture}[x=1.45cm,y=1.25cm, every node/.style={font=\small}]
  \coordinate (v00) at (0,-0.55);
  \coordinate (v01) at (0, 0.55);
  \coordinate (v10) at (1,-0.55);
  \coordinate (v11) at (1, 0.55);
  \coordinate (v20) at (2,-0.55);
  \coordinate (v21) at (2, 0.55);
  \coordinate (v30) at (3,-0.55);
  \coordinate (v31) at (3, 0.55);

  \coordinate (va0) at (6,-0.55); 
  \coordinate (va1) at (6, 0.55);
  \coordinate (vb0) at (7,-0.55); 
  \coordinate (vb1) at (7, 0.55);
  \coordinate (vc0) at (8,-0.55); 
  \coordinate (vc1) at (8, 0.55);

  \foreach \u/\v in {v00/v10,v00/v11,v01/v10,v01/v11,
                     v10/v20,v10/v21,v11/v20,v11/v21,
                     v20/v30,v20/v31,v21/v30,v21/v31,
                     va0/vb0,va0/vb1,va1/vb0,va1/vb1,
                     vb0/vc0,vb0/vc1,vb1/vc0,vb1/vc1} {
    \draw[black!70] (\u)--(\v);
  }

  \coordinate (w40) at (4,-0.55);
  \coordinate (w41) at (4, 0.55);
  \draw[black!70] (v30)--(w40);
  \draw[black!70] (v30)--(w41);
  \draw[black!70] (v31)--(w40);
  \draw[black!70] (v31)--(w41);

  \coordinate (u0) at (5,-0.55);
  \coordinate (u1) at (5, 0.55);
  \draw[black!70] (u0)--(va0);
  \draw[black!70] (u0)--(va1);
  \draw[black!70] (u1)--(va0);
  \draw[black!70] (u1)--(va1);

  \foreach \p in {v00,v01,v10,v11,v20,v21,v30,v31,va0,va1,vb0,vb1,vc0,vc1} {
    \filldraw[white] (\p) circle (2.1pt);
    \draw (\p) circle (2.1pt);
  }

  \foreach \x in {0,1,2,3,6,7,8} {
    \draw[blue, line width=0.9pt]
      (\x-0.18,-0.78) rectangle (\x+0.18,0.78);
  }

  \node at (4.5,0) {$\cdots$};

  \node[below=7pt] at (0,-0.78) {$L_0$};
  \node[below=7pt] at (1,-0.78) {$L_1$};
  \node[below=7pt] at (2,-0.78) {$L_2$};
  \node[below=7pt] at (3,-0.78) {$L_3$};

  \node[below=7pt] at (6,-0.78) {$L_{n-3}$};
  \node[below=7pt] at (7,-0.78) {$L_{n-2}$};
  \node[below=7pt] at (8,-0.78) {$L_{n-1}$};


  \draw[red,dashed,thick]
    (vc0) to[out=-110,in=-70] (v00);
  \draw[red,dashed,thick]
    (vc0) to[out=-110,in=-70] (v01);
  \draw[red,dashed,thick]
    (vc1) to[out=110,in=70] (v00);
  \draw[red,dashed,thick]
    (vc1) to[out=110,in=70] (v01);

  \node[red] at (5.0,-2.0) {$D_n=E(L_{n-1},L_0)$};

\end{tikzpicture}
\caption{The graph \(G_n\) from Example~\ref{ex:finite_4-conn}.}
\label{fig:finite_4-conn}
\end{figure}

\bibliographystyle{alpha}
\bibliography{uniqueness_v2}

\end{document}